\pdfoutput=1 

\documentclass{article}
\usepackage{excludeonly}

\usepackage[disable]{todonotes}
\usepackage{savesym}
\usepackage{amsthm,amsmath}
\usepackage{amssymb,latexsym,graphicx}
\usepackage{amscd}
 \usepackage[all,cmtip]{xy}
\usepackage{tikz-cd}
\usepackage{accents}
\usepackage{cite}
\usepackage{mathtools}
\usepackage{stmaryrd} % provides \llbracket (formal power series)

\usepackage{calligra}
\DeclareMathAlphabet{\mathcalligra}{T1}{calligra}{m}{n}
\DeclareMathAlphabet{\mathpzc}{OT1}{pzc}{m}{it}
\usepackage{hycolor}
\usepackage{xcolor}
\usepackage[
                       colorlinks=true,
                       linkcolor=black, %red,
                       citecolor=black, %green,
                       urlcolor=blue,
                     ]{hyperref}
\usepackage{soul} %provides strikethrough-command st{.. text ..}

\usepackage{makeidx}
\usepackage[intoc]{nomentbl}
\makenomenclature

\usepackage{stackengine} %%% provides Nequation - named and numbered equation
\usepackage{booktabs} % tables

\newtheorem{theoremABC}{Theorem}

\newtheorem{theorem}{Theorem}[section]

\newtheorem{corollary}[theorem]{Corollary}

\newtheorem{lemma}[theorem]{Lemma}
\newtheorem{proposition}[theorem]{Proposition}

\theoremstyle{definition}
\newtheorem{definition}[theorem]{Definition}
\newtheorem{hypothesis}[theorem]{Hypothesis}

\newtheorem{remark}[theorem]{Remark}

\newtheorem{example}[theorem]{Example}

\theoremstyle{remark}

\newcommand{\A}{{\mathbb{A}}}

\newcommand{\E}{{\mathbb{E}}}

\newcommand{\K}{{\mathbb{K}}}

\newcommand{\N}{{\mathbb{N}}}

\newcommand{\R}{{\mathbb{R}}}

\newcommand{\V}{{\mathbb{V}}}
\newcommand{\Wbb}{{\mathbb{W}}}

\newcommand{\Bb}{{\mathcal{B}}}
\newcommand{\Dd}{{\mathcal{D}}}
\newcommand{\Ee}{{\mathcal{E}}}
\newcommand{\Ff}{{\mathcal{F}}}
\newcommand{\Ll}{{\mathcal{L}}}   % Lagrangian planes
\newcommand{\Mm}{{\mathcal{M}}}   % moduli space
\newcommand{\Nn}{{\mathcal{N}}}
\newcommand{\Oo}{{\mathcal{O}}}
\newcommand{\Pp}{{\mathcal{P}}}

\newcommand{\Uu}{{\mathcal{U}}}
\newcommand{\Vv}{{\mathcal{V}}}
\newcommand{\Ww}{{\mathcal{W}}}

\newcommand{\im}{{\rm im\, }}             % image
\newcommand{\id}{{\rm id}}                % identity
\newcommand{\Id}{{\rm Id}}
\newcommand{\codim}{{\rm codim\, }}       % codimension
\newcommand{\diag}{{\rm diag}}            % diagonal matrix
\newcommand{\cl}{{\rm cl\, }}             % closure
\newcommand{\cgraph}[1]{\Gamma_{\kern-.5ex{}#1}}     % contact graph map
\newcommand{\Hess}{\mathrm{Hess}}          % Hessian
\newcommand{\CAP}{\mathop{\cap}}           % cap in formulas: adjusted spaces
 
\renewcommand{\d}{{\rm d}}

\newcommand{\ev}{{\rm ev}}

\newcommand{\norm}{{\rm norm}}

\newcommand{\eps}{{\varepsilon}}

\newcommand{\INNER}[2]{\left\langle #1, #2\right\rangle}

\newcommand{\mbf}[1]{\text{\boldmath $#1$}}  % mbf=mathboldface

\def\NABLA#1{{\mathop{\nabla\kern-.5ex\lower1ex\hbox{$#1$}}}}
\def\Nabla#1{\nabla\kern-.5ex{}_{#1}}
\def\Tabla#1{\Tilde\nabla\kern-.5ex{}_{#1}}
\def\Babla#1{\widebar\nabla\kern-.5ex{}_{#1}}
\def\abs#1{\mathopen|#1\mathclose|}   
\def\Abs#1{\left|#1\right|}            
\def\norm#1{\mathopen\|#1\mathclose\|}
\def\Norm#1{\left\|#1\right\|}

\renewcommand{\Tilde}{\widetilde}

\newcommand{\p}{{\partial}}

\newcommand{\INTO}{\hookrightarrow}              % embedding
\renewcommand{\1}{{{\mathchoice {\rm 1\mskip-4mu l} {\rm 1\mskip-4mu l}
{\rm 1\mskip-4.5mu l} {\rm 1\mskip-5mu l}}}}
\newlength\eqshift
\renewcommand\theequation{\thesection.\arabic{equation}}
\let\savetheequation\theequation

\makeatletter
\renewcommand*\env@matrix[1][\arraystretch]{%
  \edef\arraystretch{#1}%
  \hskip -\arraycolsep
  \let\@ifnextchar\new@ifnextchar
  \array{*\c@MaxMatrixCols c}}
\makeatother

\makeatletter
\let\save@mathaccent\mathaccent
\newcommand*\if@single[3]{%
  \setbox0\hbox{${\mathaccent"0362{#1}}^H$}%
  \setbox2\hbox{${\mathaccent"0362{\kern0pt#1}}^H$}%
  \ifdim\ht0=\ht2 #3\else #2\fi
  }
\newcommand*\rel@kern[1]{\kern#1\dimexpr\macc@kerna}
\newcommand*\widebar[1]{\@ifnextchar^{{\wide@bar{#1}{0}}}{\wide@bar{#1}{1}}}
\newcommand*\wide@bar[2]{\if@single{#1}{\wide@bar@{#1}{#2}{1}}{\wide@bar@{#1}{#2}{2}}}
\newcommand*\wide@bar@[3]{%
  \begingroup
  \def\mathaccent##1##2{%
    \let\mathaccent\save@mathaccent
    \if#32 \let\macc@nucleus\first@char \fi
    \setbox\z@\hbox{$\macc@style{\macc@nucleus}_{}$}%
    \setbox\tw@\hbox{$\macc@style{\macc@nucleus}{}_{}$}%
    \dimen@\wd\tw@
    \advance\dimen@-\wd\z@
    \divide\dimen@ 3
    \@tempdima\wd\tw@
    \advance\@tempdima-\scriptspace
    \divide\@tempdima 10
    \advance\dimen@-\@tempdima
    \ifdim\dimen@>\z@ \dimen@0pt\fi
    \rel@kern{0.6}\kern-\dimen@
    \if#31
      \overline{\rel@kern{-0.6}\kern\dimen@\macc@nucleus\rel@kern{0.4}\kern\dimen@}%
      \advance\dimen@0.4\dimexpr\macc@kerna
      \let\final@kern#2%
      \ifdim\dimen@<\z@ \let\final@kern1\fi
      \if\final@kern1 \kern-\dimen@\fi
    \else
      \overline{\rel@kern{-0.6}\kern\dimen@#1}%
    \fi
  }%
  \macc@depth\@ne
  \let\math@bgroup\@empty \let\math@egroup\macc@set@skewchar
  \mathsurround\z@ \frozen@everymath{\mathgroup\macc@group\relax}%
  \macc@set@skewchar\relax
  \let\mathaccentV\macc@nested@a
  \if#31
    \macc@nested@a\relax111{#1}%
  \else
    \def\gobble@till@marker##1\endmarker{}%
    \futurelet\first@char\gobble@till@marker#1\endmarker
    \ifcat\noexpand\first@char A\else
      \def\first@char{}%
    \fi
    \macc@nested@a\relax111{\first@char}%
  \fi
  \endgroup
}
\makeatother

\long\def\symbolfootnote[#1]#2{\begingroup%
\def\thefootnote{\fnsymbol{footnote}}\footnote[#1]{#2}\endgroup}
\begin{document}
\sloppy
%\author{Urs Frauenfelder \qquad Joa Weber
%                 Instituto de Matem\'{a}tica, Estat\'{\i}stica
%                 e Computa\c{c}\~{a}o Scient\'{\i}fica \\
%                 Universidade Estadual de Campinas \\
%                 Rua S\'{e}rgio Buarque de Holanda~651,
%                 Cidade Universit\'{a}ria "Zeferino Vaz",
%                CEP~13083-859, Campinas-SP, Brasil.
           %     \\
%                 joa@math.sunysb.edu.
%              \\
%              Tel.: +55-19-3521xxxx\\
%              Fax: +55-19-3521xxxx\\
%            }

\author{\quad Urs Frauenfelder \quad \qquad\qquad
             Joa Weber\footnote{
  Email: urs.frauenfelder@math.uni-augsburg.de
  \hfill
%  joa-weber@protonmail.com
  joa@unicamp.br
%    joa@math.uni-bielefeld.de
%                 joa@math.sunysb.edu.
  }
%\footnote{
%        {\bf Financial support:}
%        Funda\c{c}\~{a}o de Amparo
%        \`{a} Pesquisa do Estado de S\~{a}o Paulo
%        (FAPESP), processo $\mathrm{n}^{\rm o}$ 2017/19725-6,
        %and CNPq, Conselho Nacional de Desenvolvimento Cient\'{\i}fico
        %e Tecnol\'ogico - Brasil.
        %Bolsista do CNPq - Brasil.
        %%%\hfill
        %  se publicado individualmente:
        % "O presente trabalho foi realizado com apoio do CNPq, Conselho Nacional de Desenvol          vimento Científico  e Tecnológico - Brasil".
        %  se publicado em co-autoria:
        % "Bolsista do CNPq - Brasil".
        %
        %%%\newline
%        {\bf Address:}
%        Instituto de Matem\'{a}tica, Estat\'{\i}stica
%        e Computa\c{c}\~{a}o Scient\'{\i}fica,
%        Universidade Estadual de Campinas,
%        Rua S\'{e}rgio Buarque de Holanda~651,
        %SP~13083-859 ,
%        Campinas, SP, Brasil.
        % MSC 37Dxx 58E05
        %
%        \hfill
%        joa@ime.unicamp.br
%  }
    \\
    Universit\"at Augsburg \qquad\qquad%\quad
    UNICAMP
}

\title{Local gluing}

%\subtitle{-- Monograph --}  %%% Springer book style only
\date{\today}

%\begin{titlepage}
\maketitle %(to set the title page and copyright page; see note)
                 %\include files (e.g., preface, introduction)
%\thispagestyle{empty}
%\newpage
%
%{\color{red}
%  \subsection*{To do}
%  \begin{itemize}
%  \item
%    ...
%  \end{itemize}
%}

%\end{titlepage}

%%%%%%%%%%%%%%%%%%%%%%%%%%%%%%%%%%
%%%%%%%%%% FRONTMATTER %%%%%%%%%%%%%
%%%%%%%%%%%%%%%%%%%%%%%%%%%%%%%%%%
%\frontmatter
%\include{dedic}
%\include{foreword}
%\include{preface}
%\include{acknow}
%
%\tableofcontents
%
%\include{acronym}

%\frontmatter %• title page and copyright page information
%\include{0-dedic}
% dedication by hand:
%     \clearpage\thispagestyle{empty}
%      \par\vspace*{.35\textheight}{\centering Dedicated to ...\par}\clearpage
%
%\include{0-preface}

% ACKNOWLEDGEMENTS %
% * various anonymous referees?
% * DISCLAIMER: - references based on authors knowledge
%                        - more people contributed etc etc

%%%%%%%%%%%%%%%%%%%%%%%%%%%%%%%%%%
%%%%%%% main matter %%%%%%%%%%%%%%%%%%
%%%%%%%%%%%%%%%%%%%%%%%%%%%%%%%%%%
%\mainmatter
%\include{part}
%\include{chapter}
%\include{appendix}

%\mainmatter %\include files (e.g., main chapters, appendices)

% introduction
%\cleardoublepage
%\phantomsection
%\include{1_sc-smoothness}      % INTRODUCTION

%%%%%%%%%%%%%%%%%%%%%%%%%%%%%%%%%%%
%%%%%%% Abstract %%%%%%%%%%%%%%%%%%%%%
%%%%%%%%%%%%%%%%%%%%%%%%%%%%%%%%%%%
\begin{abstract}
In 
   \todo{actualize references?} 
the local gluing one glues local neighborhoods around the critical
point of the stable and unstable manifolds to gradient flow lines
defined on a finite time interval $[-T,T]$ for large $T$.
If the Riemannian metric around the critical point is locally
Euclidean, the local gluing map can be written down explicitly.
In the non-Euclidean case the construction of the local gluing map
requires an intricate version of the implicit function theorem.

In this paper we explain a functional analytic approach
how the local gluing map can be defined.
For that we are working on infinite dimensional path spaces 
and also interpret stable and unstable manifolds as submanifolds of
path spaces.
The advantage of this approach is that similar functional analytical
techniques can as well be generalized to infinite dimensional
versions of Morse theory, for example Floer theory.

A crucial ingredient is the Newton-Picard map. We work out an abstract
version of it which does not involve troublesome quadratic estimates.
\end{abstract}

%\newpage
\tableofcontents

\boldmath
%%%%%%%%%%%%%%%%%%%%%%%%%%%%%%%%%%%
%%%%%%%%%%%%%%%%%%%%%%%%%%%%%%%%%%%
%%%%%%% Section:  %%%%%%%%%%%%%%%%%%%%%
%%%%%%%%%%%%%%%%%%%%%%%%%%%%%%%%%%%
%%%%%%%%%%%%%%%%%%%%%%%%%%%%%%%%%%%
\section{Introduction and main results}
\label{sec:introduction}
\unboldmath

\boldmath
%%%%%%%%%%%%%%%%%%%%%%%%%%%%%%%%%%%
%%%%%%% Subsection:  %%%%%%%%%%%%%%%%%%
%%%%%%%%%%%%%%%%%%%%%%%%%%%%%%%%%%%
\subsection{Local gluing map for the Euclidean metric}
%\label{sec:neqf}
\unboldmath

Consider a diagonal matrix with monotone decreasing
diagonal entries
\begin{equation*}%\label{eq:intro-A-diag}
   A=\diag(a_1,\dots,a_n),\qquad
   a_1\ge \dots \ge a_{n-k}>0>
   a_{n-k+1}\ge \dots\ge a_n .
\end{equation*}
Consider the smooth function given by the euclidean inner product
\begin{equation}\label{eq:f-eucl}
   f\colon\R^n\to\R,\quad
   z\mapsto \tfrac12 \INNER{z}{Az} .
\end{equation}
This function is Morse and has a unique critical point at the origin
of Morse index $k$.
The gradient of $f$ for the standard metric on $\R^n$ is
$\Nabla{}f(z)=Az$.
Hence the downward gradient flow for time $s$ is given by
$$
   \varphi^{-\Nabla{} f}_s (z)
   =e^{-sA} z
   =\left(e^{-sa_1} z_1,\dots, e^{-sa_n} z_n\right) .
$$
The stable and the unstable manifold of the origin
are given by the sets
\[
   W^{\rm s}=\R^{n-k}\times \{0\},\qquad
   W^{\rm u}=\{0\}\times \R^k .
\]
Each point $z_0=(x_0,0)\in \R^{n-k}\times \{0\}$
determines an element $s\mapsto w_+(s):=e^{-sA} z_0$ in the function
space $W^{1,2}([0,\infty),\R^n)$.
Each point $z_0=(0,y_0)\in \{0\}\times \R^k$
determines an element $s\mapsto w_-(s):=e^{-sA} z_0$ in the function
space $W^{1,2}((-\infty,0],\R^n)$.

In our functional analytic approach to local gluing it is more convenient
for us to think of the stable and the unstable manifold as function space subsets
\[
   \Ww^{\rm s}\subset W^{1,2}([0,\infty),\R^n),\qquad
   \Ww^{\rm u}\subset W^{1,2}((-\infty,0],\R^n).
\]
A further advantage of this point of view is that many techniques
discussed in this article can be generalized from $\R^n$
to the Hardy approach of gluing in the infinite
dimensional case of Floer homology~\cite{Simcevic:2014a}.

With the interpretation of stable and unstable manifolds as function
spaces we can easily recover the traditional interpretation as subsets
of $\R^n$ using the evaluation maps
\[
   \ev_+\colon \Ww^{\rm s}\to \R^{n-k}\times \{0\},\quad
   w_+\mapsto w_+(0)
\]
and
\[
   \ev_-\colon \Ww^{\rm u}\to \{0\}\times \R^k,\quad
   w_-\mapsto w_-(0) .
\]
Given $T>0$, let $\Mm_T\subset W^{1,2}([-T,T],\R^n)$
be the subset of all finite time gradient flow lines $w\colon[-T,T]\to\R^n$.
Note that since in the euclidean case the gradient flow is linear
and a gradient flow line is uniquely determined by its initial
condition, the space $\Mm_T$ is an $n$-dimensional linear subspace
of the infinite dimensional function space $W^{1,2}([-T,T],\R^n)$.

In the \textbf{euclidean case}, that is $\R^n$ endowed with the standard
metric, there are natural linear isomorphisms
$$
   \Gamma_T\colon \Ww^{\rm s}\times \Ww^{\rm u}\to \Mm_T
$$
called the \textbf{local gluing maps} and
given at each time $s\in[-T,T]$ by
\begin{equation}\label{eq:Gamma_T-EUCL}
   \Gamma_T(w_+,w_-)(s)
   =e^{-(s+T)A} w_+(0) + e^{(T-s)A} w_-(0) .
\end{equation}
Consider the evaluation map defined by
\[
   \ev_T\colon W^{1,2}([-T,T],\R^n)\to\R^n\times\R^n,\quad
   w\mapsto \left(w(-T),w(T)\right) .
\]
The composition of the local gluing maps $\Gamma_T$
with the evaluation map $\ev_T$
is a \underline{linear} map, namely
\[
   \ev_T\circ\Gamma_T(w_+,w_-)
   =\left(w_+(0)+e^{2TA}w_-(0),w_-(0)+e^{-2TA}w_+(0)\right) .
\]
Since $w_-$ is in the unstable manifold and $w_+$ in the stable, 
both limits are zero
\[
   \lim_{T\to\infty} e^{2TA} w_-(0)=0
   ,\qquad
   \lim_{T\to\infty} e^{-2TA} w_+(0)=0 .
\]
Therefore it holds that
$\lim_{T\to\infty} \ev_T\circ\Gamma_T=\ev$
where
\[
   \ev=(\ev_+,\ev_-)\colon \Ww^{\rm s}\times \Ww^{\rm u}\to \R^n\times\R^n
   ,\quad
   (w_+,w_-)\mapsto \left(w_+(0),w_-(0)\right) .
\]

\boldmath
%%%%%%%%%%%%%%%%%%%%%%%%%%%%%%%%%%%
%%%%%%% Subsection:  %%%%%%%%%%%%%%%%%%
%%%%%%%%%%%%%%%%%%%%%%%%%%%%%%%%%%%
\subsection{Local gluing map for a general Riemannian metric}
%\label{sec:neqf}
\unboldmath

Given a general Morse function $f$ on a finite dimensional manifold,
by the Morse Lemma one can always find locally around each critical
point coordinates such that $f$ has the form~(\ref{eq:f-eucl})
after subtracting the critical value. In fact, it is even possible,
after some additional scaling, to assume that all diagonal entries of
the matrix $A$ are either $1$ or $-1$. In infinite dimension this is
usually not possible and therefore we don't use this fact.
\\
Unfortunately, even in finite dimension,
it is in general not possible to assume that in Morse
coordinates the Riemannian metric is standard as well.
Indeed curvature is an obstruction.

\smallskip
In this article we explain, based on a special version of Newton-Picard iteration,
a functional analytic construction for
\textbf{local gluing maps \boldmath$\gamma_T$} in the curved case.
In sharp contrast to the Euclidean version $\Gamma_T$, the local gluing maps $\gamma_T$ are
in general not linear.
However, still some of the major properties of the local gluing maps
$\Gamma_T$ in the flat case are preserved in the general case.
More precisely, we have the following theorem.
\begin{figure} %[h]
  \centering
  \includegraphics%[width=0.9\textwidth]
                             %[height=3.5cm]
                             {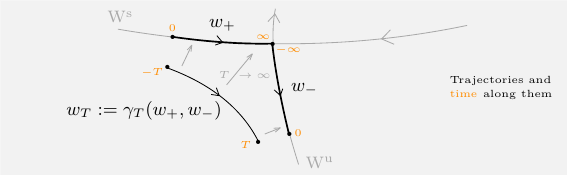}
  \caption{Convergence of local gluing 
  $\ev_T\circ\gamma_T(w_+,w_-)
  \stackrel{T\to\infty}{\longrightarrow}\ev(w_+,w_-)$}
  \label{fig:fig-conv}
\end{figure}
\begin{theoremABC}[Local gluing]\label{thm:main}
There are open neighborhoods $\Uu_+$ and $\Uu_-$ of the origin
in the stable and unstable manifold
and gluing maps $\gamma_T\colon \Uu_+\times\Uu_-\to\Mm_T$ for $T\ge T_0$,
where $\Mm_T$ is the space of downward gradient flow lines on the finite time
interval $[-T,T]$, which have the following properties.
\begin{itemize}
\item[a)]
  For every $T\ge T_0$ the gluing map $\gamma_T$ is a diffeomorphism
  onto its image
\item[b)]
  In the limit $T\to\infty$ in the $C^\infty$ topology the diagram
\begin{equation}\label{eq:comm-diag-C0}
\begin{tikzcd} [column sep=tiny] %[row sep=tiny] 
\Ww^{\rm s}\times \Ww^{\rm u}\;\;\supset
&
\Uu_+\times\Uu_-
\arrow[rrrr, "{\ev}"]
\arrow[drr, "\gamma_T"']
  &&&&\R^n\times\R^n
\\
  &&&\Mm_T
  \arrow[urr, "\ev_T"']
\end{tikzcd} 
\end{equation}
  commutes, as illustrated by Figure~\ref{fig:fig-conv},
  where $ev$ and $\ev_T$ are the evaluation maps at the end points.
\end{itemize}
\end{theoremABC}

\begin{remark}
Our construction of local gluing maps $\gamma_T$
has the following additional properties.

\begin{enumerate}\setlength\itemsep{0ex} 
\item
  In the euclidean case it holds that $\gamma_T=\Gamma_T$.
\item
  In the general Riemannian case this still continues to hold for the differential of
  $\gamma_T$ at the origin, in symbols $d\gamma_T(0_+,0_-)=\Gamma_T$.
\item
  In particular, at the infinitesimal level, our construction is
independent of any choices like the one of a cutoff function
used to construct a pre-gluing map; see~(\ref{eq:Pp_T}).
The construction of the gluing map
depends on the choice of a complement of the kernel $\E_T$ of the linearized
gradient flow equation $D_T\colon \Wbb_T\to\V_T$; see~(\ref{eq:D_T}).
There are different choices for such a complement.
Possible choices are to take the complement orthogonal with respect to
either the $L^2$ or the $W^{1,2}$ metric.
We make a different choice, so that our complement $\K_T$ is not necessarily
orthogonal, but instead has the property that 
the infinitesimal gluing map, see~(\ref{eq:Gamma_T-def}),
does not depend on the choice of the cutoff function.
\item
  Furthermore, our construction uses a version of the Newton-Picard map
which does not need quadratic estimates. 
We discuss properties of the Newton-Picard map
and its derivatives in Appendix~\ref{sec:Newton-Picard-map}.

The results in Appendix~\ref{sec:Newton-Picard-map}
are quite general, so that they should also be applicable to the
infinite dimensional version of the local gluing discussed in this article.
Namely, the general Hardy approach to gluing,
as discussed in the special case of Lagrangian Floer homology
by Tatjana {Sim\v{c}evi\'{c}}~\cite{Simcevic:2014a}.
\end{enumerate}
\end{remark}

We expect that the local gluing theorems will be useful for the construction
of flow category theories~\cite{Cohen:1995a}
by endowing, for Morse-Smale metrics,
the moduli (solution) spaces of broken gradient flow lines with
the structure of a manifold with boundary and corners
\cite{Qin:2011a,Wehrheim:2012a}.

\boldmath
%%%%%%%%%%%%%%%%%%%%%%%%%%%%%%%%%%%
%%%%%%%%%%%%%%%%%%%%%%%%%%%%%%%%%%%
%%%%%%% Section:  %%%%%%%%%%%%%%%%%%%%%
%%%%%%%%%%%%%%%%%%%%%%%%%%%%%%%%%%%
%%%%%%%%%%%%%%%%%%%%%%%%%%%%%%%%%%%
\subsection{Setup -- path spaces and sections}
\label{sec:setup}
\unboldmath

Let $f\colon\R^n\to\R$ be a smooth function
such that the origin $0$ is a Morse critical point of Morse index $k$.
Suppose $g$ is a Riemannian metric on $\R^n$ which is standard at $0$,
notation $g_0$.
Let $\Hess_0 f$ be the Hessian bilinear form of $f$ at $0$.
The Hessian linear operator $A\colon\R^n\to\R^n$
   \todo[color=yellow!40]{\small $A$ is just the Hesse MATRIX since $g_0$ is standard}
of $f$ at $0$ is defined with the help of the metric $g_0$
by the formula $\Hess_0 f(z_1,z_2)=g_0(z_1,A z_2)$
for every $z_1,z_2\in\R^n$.
After a linear change of coordinates we can assume that
$A$ is a diagonal matrix with monotone decreasing
diagonal entries 
\begin{equation}\label{eq:A-diag}
   A=\diag(a_1,\dots,a_n),\qquad
   a_1\ge \dots \ge a_{n-k}>0>
   a_{n-k+1}\ge \dots\ge a_n .
\end{equation}
Consider the $g_0$-orthogonal splitting
\begin{equation}\label{eq:p}
   \R^n=\R^{n-k}\times\R^k\stackrel{p_-}{\longrightarrow}\R^k
   ,\quad (x,y)\mapsto y
   ,\qquad p_+(x,y):=x .
\end{equation}
Then the Hessian at $0$ is positive definite on
$E^+=\R^{n-k}\times\{0\}$ and negative definite on
$E^-=\{0\}\times\R^k$. The Hessian operator at $0$ is of the form
\begin{equation}\label{eq:A}
   A=\begin{pmatrix}A_+&0\\0&-A_-\end{pmatrix}
\end{equation}
where $A_+=\diag(a_1,\dots,a_{n-k})$ and
$A_-=\diag(-a_{n-k+1},\dots,-a_n)$
are positive definite diagonal matrices.
The \textbf{spectral gap} $\sigma=\sigma(A)>0$ is the smallest distance
of an eigenvalue to the origin, in symbols
\begin{equation}\label{eq:spec-gap}
   \sigma=\sigma(A):=\min_{1\le\ell\le n} \abs{a_\ell} .
\end{equation}

\boldmath
%%%%%%%%%%%%%%%%%%%%%%%%%%%%%%%%%%%
%%%%%%% Subsection:  %%%%%%%%%%%%%%%%%%
%%%%%%%%%%%%%%%%%%%%%%%%%%%%%%%%%%%
%\subsection{Path spaces and sections}
%\label{sec:neqf}
\unboldmath
%\medskip
%\noindent
%\textbf{Path spaces and sections.}

Abbreviate $\R_+=(0,\infty)$.
For $T>0$ consider the Sobolev spaces
\[
   \Wbb_+=W^{1,2}([0,\infty),\R^n)
   ,\;
   \Wbb_-=W^{1,2}((-\infty,0],\R^n)
   ,\;
   \Wbb_T=W^{1,2}([-T,T],\R^n),
\]
%and
\[
   \V_+=L^2 ([0,\infty),\R^n)
   ,\quad
   \V_-=L^2 ((-\infty,0],\R^n)
   ,\quad
   \V_T=L^2 ([-T,T],\R^n).
\]

\begin{definition}[Constant maps to the critical point]\label{def:0_+-}
Let $0_+\in\Wbb_+$ and $0_-\in\Wbb_-$,
and $0_T\in\Wbb_T$, be the constant maps to the
critical point, in symbols
\[
   0_+\colon [0,\infty)\to\R^n,\quad s\mapsto 0
   ,\qquad
   0_-\colon (\infty,0]\to\R^n,\mapsto s\mapsto 0 .
\]
Let $0_T\in\Wbb_T$ be the constant map 
$[-T,T]\to\R^n$, $s\mapsto 0$, to the critical point.
\end{definition}

For $i\in\{+,-\}\cup\R_+$ consider the map defined by
\[
   \Ff_i\colon\Wbb_i\to\V_i
   ,\quad
   w\mapsto\p_s w+\Nabla{}f(w) .
\]
The zero sets of these maps are, respectively, the stable and the unstable manifold,
and the set of gradient flow lines along the interval $[-T,T]$, in symbols
\[
   \Ww^{\rm s}:={\Ff_+}^{-1}(0_+)\subset\Wbb_+
   ,\qquad
   \Ww^{\rm u}:={\Ff_-}^{-1}(0_-)\subset\Wbb_-
   ,
\]
and the \textbf{solution space}
\[
   \Mm_T:={\Ff_T}^{-1}(0_T)
   =\{w\colon[-T,T]\stackrel{W^{1,2}}{\to}\R^n\mid \p_s w+\Nabla{}f(w)=0\}
   \subset\Wbb_T
   .
\]
The elements of the tangent spaces at the critical point
\[
   \xi\in \E^+:=T_{0_+} \Ww^{\rm s}
   ,\qquad
   \eta\in \E^-:=T_{0_-} \Ww^{\rm u}
   ,\qquad
   \zeta\in \E_T:=T_{0_T}\Mm_T
   ,
\]
are characterized by the linear autonomous ODEs~(\ref{eq:E^x}) or,
equivalently, by forward (backward) exponential
decay~(\ref{eq:xi-eta-equiv}) of $\xi=(\xi^+,0)$
(of $\eta=(0,\eta^-)$).

\smallskip
\noindent
\textbf{Notation.}
The
\todo[color=yellow!40]{
  \textbf{Spaces}\\
  $E^+$ linear point space\\
  $\mathbb{E}^+$ linear map space\\
  $\Ee^+$ manifold of maps
}
euclidean norm of $v\in\R^\ell$, $\ell\in\N$, is denoted by $\abs{v}$.

\boldmath
%%%%%%%%%%%%%%%%%%%%%%%%%%%%%%%%%%%
%%%%%%% Subsection:  %%%%%%%%%%%%%%%%%%
%%%%%%%%%%%%%%%%%%%%%%%%%%%%%%%%%%%
\subsection{Idea of proof}
%\label{sec:neqf}
\unboldmath

We construct the desired gluing map as
a family of diffeomorphisms onto their images, one diffeomorphism
for each $T\ge T_0$ given by composing two maps
\begin{equation}\label{eq:gamma_T-goal}
   \gamma_T
   \colon
   \Ww^{\rm s}\times \Ww^{\rm u}
   \supset
   \Uu_+\times \Uu_-
   \stackrel{\wp_T}{\longrightarrow} \Wbb_T
   \stackrel{\Nn_T}{\longrightarrow} \Mm_T .
\end{equation}
Here $T_0\ge 3$ is a constant and $\Uu_+\subset \Ww^{\rm s}$
and $\Uu_-\subset \Ww^{\rm u}$
are open neighborhoods of $0_+$ and $0_-$, respectively, sufficiently
small so that the image $\wp_T(\Uu_+\times \Uu_-)$ of the pre-gluing
map $\wp_T$ lies in the domain of the Newton-Picard map $\Nn_T$.

\medskip
\textbf{Newton-Picard.}
The Newton-Picard map on $X=\Wbb_T$
associates to an approximate
zero of a map, here $\Ff_T$, a true zero nearby.
More precisely, after choosing a suitable initial point $x_0$, here $0_T$,
there are three ingredients needed:

\smallskip
\noindent
1) an approximate zero $x_1$ of $\Ff_T$;
\\
2) a uniformly bounded right inverse $Q_T$ of
$D_T:=d\Ff_T(0_T) \colon\Wbb_T\to\V_T$;
\\
3) a slowly varying operator difference
$d\Ff_T(\cdot)-D_T$ near the initial point.

\smallskip
\noindent
The facts that $\Ff_T(0_T)=0$ and that $D_T$ is surjective
suggest to choose as initial point $x_0:=0_T$.
1)~To provide an approximate zero of $\Ff_T$
will be the task of the pre-gluing map as described further below.
2)~Right inverses of the linear operator $D_T\colon \Wbb_T\to\V_T$
correspond to the topological complements of $\ker D_T$.
A natural choice would be the orthogonal complement, but
we shall choose another complement, notation $\K_T$, which represents the
impossible paths for a downward gradient and makes the infinitesimal
gluing map $\Gamma_T=d\gamma_T(0_+,0_-)$ independent of the choice of
cutoff function used to define the pre-gluing map.
The corresponding right inverse $Q_T$ indeed admits a uniform bound $c$.
3)~The operator difference $d\Ff_T(\cdot)-D_T$ is usually controlled
by calculating troublesome quadratic estimates. In Appendix~\ref {sec:NP-map}
we prove continuous differentiability of a version of the
Newton-Picard map which does not require quadratic estimates.

\begin{remark}[Higher smoothness of Newton-Picard map]\label{rmk:higher}
To obtain higher smoothness we use, roughly speaking, the fact that
the supremum of the operator norm $\norm{d\Ff_T(\cdot)-D_T}$ along
smaller and smaller
balls about the initial point~$x_0$ admits bounds closer and closer to zero.
Indeed there is a monotonically decreasing
function $\delta\colon[2,\infty)\to(0,\infty)$, independent of $T$,
such that along the $\delta(\mu)$-ball about $x_0$
the map $\norm{d\Ff_T(\cdot)-D_T}$ is bounded by $1/\mu c$.
See Corollary~\ref{cor:delta} for the case of $\Ff_T$
and Remark~\ref{rem:NP-map-mu} for the abstract theory.

For iteration arguments, such as to prove higher smoothness,
tangent maps are much more suitable than differentials.
Thus we prove in Appendix~\ref{sec:NP-tangent} an estimate for the
tangent map difference $T\Nn-\Id$ and
we show that $T\Nn^F=\Nn^{TF}$, roughly speaking,
where $\Nn^F$ is the Newton-Picard map for a map $F$.
\end{remark}

\smallskip
\textbf{Pre-gluing -- approximate zero.}
Given a real $T\ge3$, called \emph{gluing parameter}, Floer's gluing
construction associates to a pair
$(w_+,w_-)\in \Ww^{\rm s}\times \Ww^{\rm u}$ of an
$(\text{incoming},\text{outgoing})$ flow trajectory the
\emph{pre-glued path} $w_T\colon[-T,T]\to\R^m$ 
defined as follows.
One decomposes the time interval $[-T,T]$ into five subintervals.
Along $[-T,-3]$ follow the backward shifted forward
flow trajectory $w_+(T+\cdot)$, then along $[-3,-1]$ interpolate with
the help of a cutoff function to the constant flow trajectory
$s\mapsto 0$ sitting at the critical point at which $w_T$ then
rests along time $[-1,1]$.
Next along time $[1,3]$ interpolate from the constant map
$s\mapsto 0$ to the forward shifted backward flow trajectory
$w_-(-T+\cdot)$ which then represents $w_T$ along the final time interval $[3,T]$.

The behavior of the pre-glued path $w_T\colon [-T,T]\to\R^n$
along the five time intervals is detailed by formula~(\ref{eq:w_T})
and illustrated by Figure~\ref{fig:fig-z_T}.
Observe that $w_T$ takes on the boundary of its domain $[-T,T]$
values that do not depend on~$T$, namely $w_+(0)$ and $w_-(0)$.
Most importantly, the pre-glued path satisfies the gradient
equation except, possibly, along the subinterval $[-3,-1]$
{\color{gray} (and $[1,3]$)}
of $[-T,T]$ along which it coincides, up to a cutoff function factor,
with the forward flow trajectory $w_+$ along $[T-3,T-1]$.
But $w_+|_{[T-3,T-1]}$ is very close to the critical point for large $T$.
Consequently uniform exponential decay of $\p_sw_+$ takes care of the
$L^2$ norm of $\Ff_T(w_T)=\p_sw_T+\Nabla{}f(w_T)$ along $[-3,-1]$;
same along $[1,3]$ where $w_-$ appears.
With this understood it follows that $w_T$ is an approximate zero of
$\Ff_T$ in the sense that
\begin{equation}\label{eq:approx-0}
   \norm{\Ff_T(w_T)}_{\Vv_T}
   \le C e^{-\eps T}
\end{equation}
whenever $T\ge 3$.
The constant $C$ serves all elements $w_\pm$
of any chosen pair of compact neighborhoods
$K_\pm$ of $0_\pm$ in the stable/unstable manifold.

\medskip
\textbf{Gluing -- smooth convergence.}
With Newton-Picard and pre-gluing in place the gluing map $\gamma_T$,
given by composition~(\ref{eq:gamma_T-goal}), is well defined.
Appendix~\ref{sec:quant-InvFT}
revisits the proof of the usual IFT explained
in~\cite[App.\,A.3]{mcduff:2004a} to extract a quantitative
version. It is applied in Section~\ref{sec:diff-image} to prove
that $\gamma_T$ is a diffeomorphism onto its image along a
sufficiently small domain, uniformly in~$T$.

In the limit $T\to\infty$ the diagram~(\ref{eq:comm-diag-C0})
commutes even after application of the $m$-fold tangent functor $T^m$.
The proof uses techniques described by Remark~\ref{rmk:higher}
and is carried out in Section~\ref{sec:Cm-convergence}.

\boldmath
%%%%%%%%%%%%%%%%%%%%%%%%%%%%%%%%%%%
%%%%%%% Subsection:  %%%%%%%%%%%%%%%%%%
%%%%%%%%%%%%%%%%%%%%%%%%%%%%%%%%%%%
\subsection*{Outline of article}
%\label{sec:neqf}
\unboldmath

Section~\ref{sec:pre-gluing} 
``Pre-gluing map $\Pp_T$ and its restriction $\wp_T$''
introduces for each parameter value $T\ge 3$ the pre-gluing map
as the \emph{linear} map
$\Pp_T\colon \Wbb_+\times\Wbb_-\to\Wbb_T$
defined by~(\ref{eq:Pp_T}), equivalently by~(\ref{eq:w_T}),
and illustrated by Figure~\ref{fig:fig-z_T}.

\smallskip
\noindent
For $0_+\in \Ww^{\rm s}$ and $0_-\in \Ww^{\rm u}$ the pre-glued path
is a true zero, more precisely
$
   \Pp_T(0_+,0_-)
   =0_T\in\Ff_T^{-1}(0)\subset \Wbb_T
$.
This motivates the expectation that pre-gluing pairs near
$(0_+,0_-)\in \Ww^{\rm s}\times \Ww^{\rm u}$ should produce approximate
zeroes.
Thus we consider the \textbf{restriction} of the pre-gluing map $\Pp_T$, notation
\begin{equation}\label{eq:wp-DEL}
   \wp_T:=\Pp_T|_{\Ww^{\rm s}\times \Ww^{\rm u}}
   \colon\Ww^{\rm s}\times\Ww^{\rm u}\to\Wbb_T .
\end{equation}
This map is smooth by linearity of $\Pp_T$.
Whereas the elements of the tangent spaces to $\Ww^{\rm s}$,
$\Ww^{\rm u}$, and $\Mm_T$ at the origins $0_+$, $0_-$, and $0_T$
(notation $\E^\pm$ and $\E_T$)
are the solutions of autonomous linear ODEs~(\ref{eq:E^x}),
at general points $w_+$, $w_-$, and $w_T$ the characterizing linear
ODE's~(\ref{eq:TW^s}) are non-autonomous.\footnote{
  whenever the Hessian operators $A_w(s)$
  along a flow trajectory $w$ depend on time~$s$
  }

\smallskip
\noindent
The linear identifications
$\theta_{w_\pm}\colon T_{w_\pm} \Ww^{\rm s/u}\to\E^\pm$,
defined via asymptotic limits,
are used to prove Theorem~\ref{thm:gluing-diffeo}
(gluing map $\gamma_T$ is diffeomorphism onto its image).

\medskip
Section~\ref{sec:infinit-gluing} ``Infinitesimal gluing''
consists of two subsections.
Subsection~\ref{sec:proj-K_T} introduces a complement $\K_T$ of the
$n$-dimensional subspace $\E_T:=T_{0_T}\Mm_T$ of $\Wbb_T$ and the
corresponding projection $P_{\E_T,\K_T}$ onto $\E_T$ along $\K_T$, notation
$$
   \Pi_T:=P_{\E_T,\K_T}\colon \Wbb_T=\K_T\oplus \E_T\to \E_T .
$$
Lemma~\ref{le:projection-onto-E} provides a formula for $\Pi_T$
and asserts that the operator norm of $\Pi_T\colon \Wbb_T\to \Wbb_T$
is bounded by a constant $d=d(a_1,a_n)$,
depending on the eigenvalues $a_1$ and $a_n$ of the Hessian $A$
in~(\ref{eq:A-diag}), but independent of $T\ge 1$.
To prove this we establish the uniform-in-$T$ Sobolev estimate
$\norm{v}_{L^\infty[-T,T]}\le 2 \norm{v}_{W^{1,2}[-T,T]}$.
Later on the estimate also enters the proof of Corollary~\ref{cor:delta}
on existence of the monotone function $\delta(\mu)$ mentioned in
Remark~\ref{rmk:higher} on higher smoothness of the Newton-Picard map.
\\
Subsection~\ref{sec:Gamma_T} introduces the infinitesimal gluing map,
namely the linear map
\begin{equation*}%\label{eq:Gamma_T-def-intro}
   \Gamma_T:=\Pi_T\circ d\wp_T(0_+,0_-)
   \colon \E^+\times \E^-\to\Wbb_T\to \E_T .
\end{equation*}
For $\Gamma_T$ we obtain formula~(\ref{eq:Gamma_T})
which, firstly, by choice of $\K_T$, does not depend on the choice of
cutoff function $\beta$ in the pre-gluing map~(\ref{eq:Pp_T}) and, secondly,
reproduces the gluing map~(\ref{eq:Gamma_T-EUCL})
in the Euclidean model case.
Lemma~\ref{le:Gamma_T-iso} asserts that $\Gamma_T$ is an isomorphism
with inverse bounded by the constant $k:=1/(1-e^{-12\sigma})$,
independent of $T$, where $\sigma$ is the spectral
gap~(\ref{eq:spec-gap}) of the Hessian $A$.

\medskip
Section~\ref{sec:NP-map-appl} ``Newton-Picard map''
consists of three subsections in which we verify the three
ingredients 1) 2) 3) described earlier.
\\
Subsection~\ref{sec:approx-zero} shows 1) the pre-gluing provides an approximate
zero $w_T:=\wp_T(w_+,w_-)$ of $\Ff_T$ in the sense of~(\ref{eq:approx-0}).
This hinges on Appendix~\ref{sec:exponential}
where we provide suitable exponential decay uniformly in $T$.
\\
Subsection~\ref{sec:surjectivity} shows that the linearization
$D_T:=d\Ff_T(0_T)\colon\Wbb_T\to\V_T$ is surjective and 2) provides a bound
$c=c(a_1,a_n)$ uniformly in $T$ for the right inverse $Q_T$ associated
to the complement $\K_T$ of $\E_T$. Actually $\E_T=\ker D_T$.
\\
Subsection~\ref{sec:def-NP} establishes 3) a bound on the difference
$d\Ff_T(\cdot)-D_T$.
Based on Proposition~\ref{prop:wonder-prop-NEW}
we define the Newton-Picard map $\Nn_T\colon\Wbb_T\to\Wbb_T$
along a neighborhood $U_0(\delta_4)$ of the initial point $x_0:=0_T$.
Then it is shown that for $T\ge 3$ pre-gluing map $\wp_T$ takes values
in the domain of $\Nn_T$.

\medskip
Section~\ref{sec:gluing} ``Gluing map'' provides an open neighborhood
$\Uu_+\times\Uu_-\subset \Ww^{\rm s}\times \Ww^{\rm u}$ of the origin
$(0_+,0_-)$ which serves as domain for all gluing maps $\gamma_T$ with
gluing parameter $T\ge T_0$, see~(\ref{eq:cond-1}), and defined by
pre-gluing $\wp_T$ followed by Newton-Picard zero detection $\Nn_T$,
see~(\ref{eq:gamma_T-goal}).
By Lemma~\ref{le:d-gamma} the linearized gluing map at the origin
coincides with the infinitesimal gluing map
$\Gamma_T$.
\\
Subsection~\ref{sec:diff-image} ``Diffeomorphism onto image''
proves this property of the gluing maps along an open subset
$\Oo_+\times\Oo_-\subset \Uu_+\times\Uu_-$, uniformly in $T$.
This is an application of the quantitative inverse function
Theorem~\ref{thm:InvFT-quant}. Verification of~(\ref{eq:vgvh-1})
uses that $d\gamma_T|_{(0_+,0_-)}$ is the infinitesimal  gluing map
$\Gamma_T$ (Lemma~\ref{le:d-gamma}) and that $\Gamma_T$ has an
inverse bounded uniformly in $T$ (Lemma~\ref{le:Gamma_T-iso}).
Verification of~(\ref{eq:vgvh-2}) uses Remark~\ref{rem:A}
and~\ref{rem:NP-map-mu} on the linearized Newton-Picard map.
\\
Subsection~\ref{sec:Cm-convergence}
``Evaluation maps and convergence in $C^m$''
shows that in the limit $T\to\infty$ the diagram~(\ref{eq:comm-diag-C0})
commutes as illustrated by Figure~\ref{fig:fig-conv}.

\medskip
The appendices provide abstract results
which might be of general interest. 
\\
Appendix~\ref{sec:quant-InvFT} is on the
``Quantitative inverse function theorem''.
\\
Appendix~\ref{sec:Newton-Picard-map} provides the
``Newton-Picard map without quadratic estimates''.
\\
Appendix~\ref{sec:exponential} ``Exponential decay''
proves such, uniformly in $T$, and for all time derivatives.
We use again the tangent map formalism for ease of induction.
The proof is based on Lemma~\ref{le:exp-dec}
in which the exponential decay rate of $\eta$ is inherited by $\xi$,
as opposed to the original~\cite[Le.\,3.1]{Robbin:2001a}.

\smallskip\noindent
{\bf Acknowledgements.}
UF acknowledges support by DFG grant
FR 2637/2-2.

\boldmath
%%%%%%%%%%%%%%%%%%%%%%%%%%%%%%%%%%%
%%%%%%% Subsection:  %%%%%%%%%%%%%%%%%%
%%%%%%%%%%%%%%%%%%%%%%%%%%%%%%%%%%%
\section[Pre-gluing]
% map $\mbf{\Pp_T}$ and restriction $\mbf{\wp_T}$ to 
% $\mbf{\Ww^{\rm s}\times \Ww^{\rm u}}$]
{Pre-gluing map $\Pp_T$ and its restriction $\wp_T$}
\label{sec:pre-gluing}
\unboldmath
Fix a cut-off function $\beta\colon\R\to[0,1]$, that is a smooth
function such that $\beta(s)=0$ for $s\le -1$ and $\beta(s)=1$ for
$s\ge1$. For any real $T\ge 3$, the \textbf{gluing parameter},
   \todo[color=green!40]{$T\ge 3$}
the \textbf{pre-gluing map} is the
\underline{linear} bounded Hilbert space map defined by
\begin{equation}\label{eq:Pp_T}
\begin{split}
%   \#_T=
   \Pp_T\colon \Wbb_+\times\Wbb_-
   &\to\Wbb_T
   \\
   (w_+,w_-)
   &\mapsto 
   \underbrace{
   \left(1-\beta(\cdot+2)\right) w_+(T+\cdot)
   +\beta(\cdot-2)\, w_-(-T+\cdot)
   }_{=:w_T}
   .
\end{split}
\end{equation}

\begin{lemma}[Uniform bound]\label{le:unif-bd-Pp_T}
There is a constant $b>0$, depending on the cut-off function
$\beta$ but not on $T$, such that
$\norm{\Pp_T}\le b$ for every $T\ge 3$.
\end{lemma}

\begin{proof}
The shift map is an isometry in $W^{1,2}(\R)$
and the cut-off function $\beta$ is independent of $T$.
\end{proof}

The pre-gluing map has the two properties that, firstly,
for times $s$ on the boundary of $[-T,T]$ we have
\[
   \Pp_T (w_+,w_-)(-T)=w_+(0)
   ,\qquad
   \Pp_T (w_+,w_-)(T)=w_-(0) ,
\]
and, secondly, during the time interval $[-1,1]$ the map rests in the
critical point
\[
   \Pp_T(w_+,w_-)|_{[-1,1]}\equiv 0 .
\]
More precisely, for fixed $w_\pm$, the \textbf{pre-glued path}
$\mbf{w_T}$ is of the form
\begin{equation}\label{eq:w_T}
   \underbrace{\Pp_T(w_+,w_-)(s)}_{=:w_T(s)}=
\begin{cases}
   w_+(T+s)&\text{, $s\in[-T,-3]$}\\
   \left(1-\beta(s+2)\right) w_+(T+s)&\text{, $s\in[-3,-1]$}\\
   0&\text{, $s\in[-1,1]$}\\
   \beta(s-2) w_-(-T+s)&\text{, $s\in[1,3]$}\\
   w_-(-T+s)&\text{, $s\in[3,T]$}
\end{cases}
\end{equation}
for $s\in[-T,T]$.
The pre-glued path $w_T$ for $w_\pm$ is illustrated by
Figure~\ref{fig:fig-z_T}.
\begin{figure}[h]
  \centering
  \includegraphics%[width=0.9\textwidth]
                             [height=5.5cm]
                             {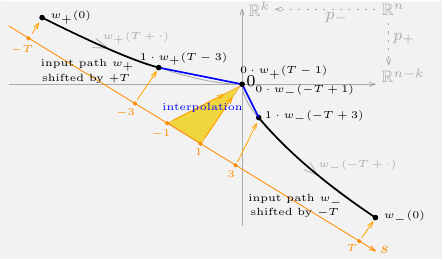}
  \caption{Pre-glued path $w_T(s):=\Pp_T(w_+,w_-)(s)$ for $s\in[-T,T]$}
  \label{fig:fig-z_T}
\end{figure}

\begin{example}[Constant maps to the critical point]
It holds that
\begin{equation}\label{eq:0_T}
   \Pp_T(0_+,0_-)
   =0_T .
\end{equation}
Note that $0_+\in \Ww^{\rm s}$, $0_-\in \Ww^{\rm u}$,
and $0_T\in\Mm_T$.
\end{example}

\boldmath
%%%%%%%%%%%%%%%%%%%%%%%%%%%%%%%%%%%
%%%%%%% Subsection:  %%%%%%%%%%%%%%%%%%
%%%%%%%%%%%%%%%%%%%%%%%%%%%%%%%%%%%
\subsection*{Restricted pre-gluing map}
%\label{sec:neqf}
\unboldmath

We denote the restriction of the pre-gluing map $\Pp_T$ to
the stable and unstable manifolds $\Ww^{\rm s}\subset\Wbb_+$ and
$\Ww^{\rm u}\subset\Wbb_-$ by
\begin{equation}\label{eq:wp}
   \wp_T:=\Pp_T|_{\Ww^{\rm s}\times \Ww^{\rm u}}
   \colon \Ww^{\rm s}\times \Ww^{\rm u}\to \Wbb_T
\end{equation}
where $T\ge3$.
This map is smooth by linearity of $\Pp_T$.

\medskip
\noindent\boldmath
\textbf{Differential of $\wp_T$ at $(0_+,0_-)$.}
\unboldmath
Consider the tangent spaces to the trajectory spaces
$\Ww^{\rm s}$, $\Ww^{\rm u}$, and $\Mm_T$, at the critical point, namely
\begin{equation}\label{eq:E^x}
\begin{split}
   \E^+:=T_{0_+} \Ww^{\rm s}
   &=\{\xi\in\Wbb_+\mid\p_s\xi+A\xi=0\}
   \subset C^\infty([0,\infty),{\color{red}\R^{n-k}\times \{0\}})
   \\
   \E^-:=T_{0_-} \Ww^{\rm u}
   &=\{\eta\in\Wbb_-\mid\p_s\eta+A\eta=0\}
   \subset C^\infty((-\infty,0], {\color{red} \{0\}\times\R^k})
   \\
   \E_T:=T_{0_T}\Mm_T
   &=\{\zeta\in\Wbb_T\mid\p_s\zeta+A\zeta=0\}
   \subset C^\infty([-T,T],\R^n) .
\end{split}
\end{equation}
By the theorem of Picard-Lindel\"of the dimensions are given by
\[
   \dim \E^+=n-k,\qquad
   \dim \E^-=k,\qquad
   \dim \E_T=n.
\]
Then the linearization of $\wp_T$ at $(0_+,0_-)$ is a map
\[
    d\wp_T(0_+,0_-)\colon
   \E^+\times \E^-\to \Wbb_T .
\]
For $\zeta\in \E^+,\E^-,\E_T$ we abbreviate
\begin{equation}\label{eq:zeta^pm}
   \zeta^+(s):=p_+(\zeta(s)),\quad
   \zeta^-(s):=p_-(\zeta(s)),\qquad
   \text{so $\zeta=(\zeta^+,\zeta^-)$.}
\end{equation}
Since $A$ is diagonal, see~(\ref{eq:A}), and by decay of the elements
of $\Wbb_\pm$, given maps $\xi\in C^\infty([0,\infty),\R^n)$ and 
$\eta\in C^\infty((-\infty,0],\R^n)$, there are the
equivalences%\footnote{
%  to increase readability we write terms like $e^{-s(A_+\oplus 0)}\xi(0)$
%  in the form $e^{-sA_+}\xi(0)$
%  }
\begin{equation}\label{eq:xi-eta-equiv}
\begin{split}
   \xi\in \E^+\qquad
   &\Leftrightarrow\qquad
   \xi=(\xi^+,0)
   \quad\wedge\quad
   \xi^+(s)=e^{-sA_+}\xi^+(0)\;\forall s\ge 0 ,
   \\
   \eta\in \E^-\qquad
   &\Leftrightarrow\qquad
   \eta=(0,\eta^-)
   \quad\wedge\quad
   \eta^-(s)=e^{+sA_-}\eta^-(0)\;\forall s\le 0 .
\end{split}
\end{equation}

\medskip
\noindent\boldmath
\textbf{Differential of $\wp_T$ at
$(w_+,w_-)\in \Ww^{\rm s}\times \Ww^{\rm u}$.}
\unboldmath
The tangent spaces to the trajectory spaces $\Ww^{\rm s}$,
$\Ww^{\rm u}$, and $\Mm_T$, at points $w_+$, $w_-$, and $w_T$,
   \todo[color=yellow!40]{\small $\p_s$ not $\nabla$! Ableitung im $\R^n$}
are
\begin{equation}\label{eq:TW^s}
\begin{split}
%   \E^+\simeq 
   T_{w_+} \Ww^{\rm s}
   &=\{\xi\in\Wbb_+\mid\p_s\xi+A_{w_+}\xi=0\}
   \subset C^\infty([0,\infty),\R^n)
   \\
%   \E^-\simeq 
   T_{w_-} \Ww^{\rm u}
   &=\{\eta\in\Wbb_-\mid\p_s\eta+A_{w_-}\eta=0\}
   \subset C^\infty((-\infty,0],\R^n)
   \\
%   \E_T\simeq 
   T_{w_T}\Mm_T
   &=\{\zeta\in\Wbb_T\mid\p_s\zeta+A_{w_T}\zeta=0\}
   \subset C^\infty([-T,T],\R^n) .
\end{split}
\end{equation}
Here, for $w\in\{w_+,w_-,w_T\}$, the family of
\textbf{Hessian operators}
\[
   A_w=\{A_w(s)\}_{s}
\]
is defined by the identities
$\left(\Hess_{w(s)}
  f\right)(\cdot,\cdot)=g_{w(s)}\left(\cdot,A_w(s)\cdot\right)$,
one identity for each $s$.
There are canonical, continuous and linear, identifications\footnote{
Observe that the elements
\[
   \xi\in\E^+=T_{0_+}\Ww^{\rm s}\subset T_{0_+}\Wbb_+=\Wbb_+
   ,\qquad
   \theta_{w_+}\xi
   \in T_{w_+}\Ww^{\rm s}\subset T_{w_+}\Wbb_+=\Wbb_+ ,
\]
lie in the same ambient vector space $\Wbb_+$, hence they can be added.
Similarly $\eta\in\E^-$ and $\theta_{w_-}\eta$ both lie in $\Wbb_-$.
}
\begin{equation}\label{eq:theta_+}
   \theta_{w_+}\colon \Wbb_+\supset T_{w_+} \Ww^{\rm s}
   \stackrel{\simeq}{\longrightarrow} \E^+,\qquad
   \theta_{w_-}\colon \Wbb-\supset T_{w_-} \Ww^{\rm u}
   \stackrel{\simeq}{\longrightarrow} \E^-,
\end{equation}
given by asymptotic limits where
$\theta_{0_\pm}=\Id_{\E^\pm}=:\Id_\pm$ and the linear operators
$\theta_{w_\pm}$ depend continuously on $w_\pm$;
   \todo[color=yellow!40]{\small detail this; cf. method 'glattbuegeln' [Weber:2015c]}
see~\cite[\S 3]{Robbin:2001a}.
Since $\wp_T:=\Pp_T|_{\Ww^{\rm s}\times \Ww^{\rm u}}$
is defined by restricting a linear map, the linearization is
the linear map's restriction
\begin{equation}\label{eq:dwp=P_T}
   d\wp_T(w_+,w_-)
   =d\left(\Pp_T|\right)(w_+,w_-)
   =\Pp_T|
   \colon T_{w_+} \Ww^{\rm s}\times T_{w_-} \Ww^{\rm u}\to \Wbb_T .
\end{equation}
Thus, given $(w_+,w_-)\in \Ww^{\rm s}\times \Ww^{\rm u}$, the map defined by
\begin{equation}\label{eq:Theta_T}
\begin{split}
   \Theta_T(w_+,w_-)
   \colon \E^+\times \E^-
   &\to \Wbb_T\\
   (\xi,\eta)
   &\mapsto\Pp_T
   \left(\xi-\theta_{w_+}^{-1}\xi, \eta-\theta_{w_-}^{-1}\eta\right)
\end{split}
\end{equation}
is, by~(\ref{eq:dwp=P_T}), equal to the difference
\begin{equation*}
\begin{split}
   \Theta_T(w_+,w_-)
   &=
   \Pp_T
   \left(
      (\Id_+,\Id_-)
      -(\theta_{w_+}^{-1}, \theta_{w_-}^{-1})
   \right)
\\
   &=
   d\wp_T|_{(0_+,0_-)}
   -d\wp_T|_{(w_+,w_-)}\circ (\theta_{w_+}^{-1}, \theta_{w_-}^{-1}) .
\end{split}
\end{equation*}

\begin{lemma}\label{le:Theta_T}
For any $\eps>0$ there are neighborhoods $\Oo^+_\eps$ of
$0_+$ in $\Ww^{\rm s}$ and
$\Oo^-_\eps$ of $0_-$ in $\Ww^{\rm u}$ 
such that for every $T\ge 3$ and every
$(w_+,w_-)\in \Oo^+_\eps\times \Oo^-_\eps$
the operator norm of
$\Theta_T(w_+,w_-) \colon \E^+\times \E^- \to \Wbb_T$
is less or equal $\eps$.
\end{lemma}

\begin{proof}
Lemma~\ref{le:unif-bd-Pp_T}
and continuous dependence of $\theta_{w_\pm}$ on $w_\pm$
and $\theta_{0_\pm}=\Id_\pm$.
\end{proof}

\boldmath
%%%%%%%%%%%%%%%%%%%%%%%%%%%%%%%%%%%
%%%%%%% Subsection:  %%%%%%%%%%%%%%%%%%
%%%%%%%%%%%%%%%%%%%%%%%%%%%%%%%%%%%
\section{Infinitesimal gluing}
\label{sec:infinit-gluing}
\unboldmath
%\medskip
%\noindent
%\textbf{Pre-gluing map.}

\boldmath
%%%%%%%%%%%%%%%%%%%%%%%%%%%%%%%%%%%
%%%%%%% Subsection:  %%%%%%%%%%%%%%%%%%
%%%%%%%%%%%%%%%%%%%%%%%%%%%%%%%%%%%
\subsection{Projection associated to a particular complement}
\label{sec:proj-K_T}
\unboldmath

\boldmath
%%%%%%%%%%%%%%%%%%%%%%%%%%%%%%%%%%%
%%%%%%% Subsection:  %%%%%%%%%%%%%%%%%%
%%%%%%%%%%%%%%%%%%%%%%%%%%%%%%%%%%%
\subsubsection*{Complement $\K_T$ of $n$-dimensional linear solution space $\E_T$}
%\label{sec:neqf}
\unboldmath 

For $T>0$ we choose a Hilbert space complement $\K_T$ of $\E_T$ in
$\Wbb_T$ of the form
\begin{equation}\label{eq:Kk_T}
   \K_T
   :=\{\zeta=(\zeta^+,\zeta^-)\in\Wbb_T\mid
   \text{$\zeta^+ (-T)=0$,\, $\zeta^- (T)=0$}\}.
\end{equation}
Note that the elements $\zeta$ of $\K_T$ start at points $\zeta(-T)$ in the
negative definite space $\{0\}\times\R^k$ and end at points
$\zeta(T)$ in the positive definite space $\R^{n-k}\times\{0\}$.
Roughly speaking, the linear subspace $\K_T$ of $\Wbb_T$
represents impossible paths for a downward gradient flow.
The following lemma tells that $\codim \K_T=n$.

The complement $\K_T$ of $\E_T$ is not necessarily orthogonal.
But it has the useful property that the infinitesimal gluing map
$\Gamma_T$ in~(\ref{eq:Gamma_T-def}) will not depend on the cutoff
function $\beta$ that was used to define the pre-gluing map $\Pp_T$
   \todo[color=yellow!40]{\small $\E_T=\ker D$, $\K_T=\im Q$, $\1-QD=P_{\ker D,\im Q}=P_{\E_T,\K_T}$}
in~(\ref{eq:Pp_T}).

\begin{lemma}[Complement]\label{le:complement}
a) $\K_T\CAP \E_T=\{0\}$ and
b) $\Wbb_T=\K_T+\E_T$.
\end{lemma}

\begin{proof}
a) Pick $\zeta=(\zeta^+,\zeta^-)\in \E_T$. Then
$A(0,\zeta^-)=(0,-A_-\zeta^-)$ and $A(\zeta^+,0)=(A_+\zeta^+,0)$.
Hence $\p_s\zeta^-=A_-\zeta^-$ and $\p_s\zeta^+=-A_+\zeta^+$,
and therefore
\[
   \zeta^-(s)=e^{(s-T)A_-}\zeta^-(T)
   ,\qquad
   \zeta^+(s)=e^{(-s-T)A_+}\zeta^+(-T) ,
\]
for every $s\in[-T,T]$.
Let $\zeta=(\zeta^+,\zeta^-)\in \K_T\CAP \E_T$.
Then $\zeta^+(-T)=0$ and $\zeta^-(T)=0$,
so $\zeta^+\equiv 0$ and $\zeta^-\equiv 0$ since each one solves a
first order ODE. Hence $\zeta=(\zeta^+,\zeta^-)\equiv 0$.
b) Let $\zeta\in\Wbb_T$. Then the map defined by
$Z(s):=\left(e^{(-s-T)A_+}p_+\zeta(-T),e^{(s-T)A_-}p_-\zeta(T)\right)$
is element of $\E_T$ and the difference
$\zeta-Z$ lies in $\K_T$ since $p_+(\zeta-Z)(-T)=0$
and $p_-(\zeta-Z)(T)=0$.
\end{proof}

\boldmath
%%%%%%%%%%%%%%%%%%%%%%%%%%%%%%%%%%%
%%%%%%% Subsection:  %%%%%%%%%%%%%%%%%%
%%%%%%%%%%%%%%%%%%%%%%%%%%%%%%%%%%%
\subsubsection*{Uniform Sobolev estimate}
%\label{sec:neqf}
\unboldmath

\begin{lemma}\label{le:sob-emb-T}
Let $T\ge 1$.
   \todo[color=green!40]{$T\ge 1$}
Then any $v\colon[-T,T]\to\R$ of class $W^{1,2}$
   \todo[color=yellow!40]{\small can we get $1$?}
satisfies\footnote{
  In~\cite[(4.57)]{Frauenfelder:2022f}
  we proved the case $T=\infty$ with constant $1$, not $2$.
  }
\begin{equation}\label{eq:sob-emb-T}
   \norm{v}_\infty\le 2 \norm{v}_{1,2}
\end{equation}
where the norms are over the domain $[-T,T]$.
\end{lemma}

Estimate~(\ref{eq:sob-emb-T}) continues to hold for vector-valued
maps $v\colon[-T,T]\to\R^\ell$ of class $W^{1,2}$ since
\[
   \norm{(v_1,\dots,v_\ell)}_\infty
   =\max_{i=1,\dots,\ell} \norm{v_i}_\infty
   \le 2 \max_{i=1,\dots,\ell}\norm{v_i}_{1,2}
   \le 2 \norm{v}_{W^{1,2}([-T,T],\R^\ell)} .
\]

\begin{proof}[Proof of Lemma~\ref{le:sob-emb-T}]
The proof has 4 steps.

\medskip \noindent
\textbf{Step 1.}
Let $T>0$.
Suppose $\norm{v}_{1,2}\le 1$
and at $s_0\in[-T,T]$ we have $\kappa:=\abs{v(s_0)}>0$.
Then for $s\in[-T,T]\cap [s_0-\kappa^2/4,s_0+\kappa^2/4]$ it holds
$\abs{v(s)}\ge\kappa/2$.

\medskip\noindent
Pointwise at $s$ we have
\begin{equation*}
\begin{split}
   \abs{v(s)}
   &=\Abs{v(s_0)+\int_{s_0}^s v^\prime(\sigma)\cdot 1\, d\sigma}\\
   &\ge\abs{v(s_0)}-\sqrt{\int_{s_0}^s \left(v^\prime(\sigma)\right)^2\, d\sigma}
   \sqrt{\int_{s_0}^s 1\, d\sigma}\\
   &\ge\kappa- \biggl(\underbrace{\int_{-T}^{T} \left(v^\prime(\sigma)\right)^2\,
       d\sigma}_{\le\norm{v}_{1,2}^2\le 1}\biggr)^{1/2}
   \cdot\sqrt{\abs{s-s_0}}\\
   &\ge\kappa-\sqrt{\abs{s-s_0}}\\
   &\ge\kappa-\sqrt{\kappa^2/4}=\kappa/2 .
\end{split}
\end{equation*}
This proves Step~1.

%\medskip
%\noindent
%\textbf{Step 2.}
%Under the assumption of Step~1 suppose, in addition, the inclusion
%$[s_0,s_0+\kappa^2/4]\subset[-T,T]$, then $\abs{v(s_0)}\le 2$.
%
%\medskip\noindent
%To prove Step~2 use Step~1 to obtain that
%\begin{equation*}
%\begin{split}
%   1
%   \ge\norm{v}_{1,2}^2
%   \ge\norm{v}_2^2
%   \stackrel{\mathrm{inclusion}}{\ge}\int_{s_0}^{s_0+\kappa^2/4} 
%   \underbrace{v(\sigma)^2}_{\stackrel{\rm St.1}{\ge}\kappa^2/4}\, d\sigma
%   \ge\kappa^4/16 .
%\end{split}
%\end{equation*}
%Therefore $2\ge\kappa:=\abs{v(s_0)}$ which proves Step~2.
%
%
%\medskip
%\noindent
%\textbf{Step 3.}
%Under the assumption of Step~1 suppose, in addition, the inclusion
%$[s_0-\kappa^2/4,s_0]\subset[-T,T]$, then $\abs{v(s_0)}\le 2$.
%
%\medskip\noindent
%The same argument as in Step~2 proves Step~3.

\medskip
\noindent
\textbf{Step 2.}
Under the assumption of Step~1 suppose, in addition, the inclusion
$[s_0,s_0+1]\subset[-T,T]\cap [s_0-\kappa^2/4,s_0+\kappa^2/4]$,
then $\abs{v(s_0)}=2$.

\medskip\noindent
To prove Step~2 use Step~1 to obtain that
\begin{equation*}
\begin{split}
   1
   \ge\norm{v}_{1,2}^2
   \ge\norm{v}_2^2
   \stackrel{\mathrm{inclusion}}{\ge}\int_{s_0}^{s_0+1}
   \underbrace{v(\sigma)^2}_{\stackrel{\rm St.1}{\ge}\kappa^2/4}\, d\sigma
   \ge\kappa^2/4 .
\end{split}
\end{equation*}
Therefore $2\ge\kappa$.
In view of the inclusion this implies $2=\kappa:=\abs{v(s_0)}$.

\medskip
\noindent
\textbf{Step 3.}
Under the assumption of Step~1 suppose, in addition, the inclusion
$[s_0-1,s_0]\subset[-T,T]\cap [s_0-\kappa^2/4,s_0+\kappa^2/4]$,
then $\abs{v(s_0)}=2$.

\medskip\noindent
The same argument as in Step~2 proves Step~3.

\medskip
\noindent
\textbf{Step 4.}
Let $T\ge1$.
If $\norm{v}_{1,2}\le1$, then $\abs{v(s_0)}\le 2$ for every $s_0\in[-T,T]$.

\medskip\noindent
The assumption $T\ge1$ guarantees
$[s_0-1,s_0]\subset[-T,T]$ or $[s_0,s_0+1]\subset[-T,T]$.
We argue by contradiction and assume that $\kappa:=\abs{v(s_0)}>2$.
Then $[s_0-1,s_0]$ or $[s_0,s_0+1]$ is contained in the intersection
$
   [-T,T]\cap[s_0-\kappa^2/4,s_0+\kappa^2/4]
$.
Therefore by Step~2 or Step~3 we have $\kappa=2$. Contradiction.

\medskip
\noindent
By homogeneity of the norm Step~4 implies Lemma~\ref{le:sob-emb-T}.
\end{proof}

\boldmath
%%%%%%%%%%%%%%%%%%%%%%%%%%%%%%%%%%%
%%%%%%% Subsection:  %%%%%%%%%%%%%%%%%%
%%%%%%%%%%%%%%%%%%%%%%%%%%%%%%%%%%%
\subsubsection*{Projection $\Pi_T$ onto $\E_T$ along $\K_T$}
%\label{sec:neqf}
\unboldmath

We denote the linear \textbf{projection}
in the path space $\Wbb_T:=W^{1,2}([-T,T],\R^n)$
onto the $n$ dimensional subspace $\E_T$ along the
(not necessarily orthogonal) complement $\K_T$ by
\begin{equation}\label{eq:Pi_T}
   P_{\E_T,\K_T}\stackrel{\text{(\ref{eq:zeta_E})}}{=}\Pi_T
   \colon\Wbb_T=\E_T\oplus \K_T\to \E_T.
\end{equation}

\begin{lemma}\label{le:projection-onto-E}
The projection $P_{\E_T,\K_T}$ is given by the map
$\Pi_T\colon\zeta\mapsto\zeta_{\E}$ where
\begin{equation}\label{eq:zeta_E}
\begin{split}
   \zeta_\E(s)
   :=\left(e^{-(s+T)A_+}\zeta^+(-T),
   e^{(s-T)A_-}\zeta^-(T)\right)
\end{split}
\end{equation}
for $s\in[-T,T]$.There is
   \todo[color=green!40]{$T\ge 1$}
a constant $d=d(a_1,a_n)$,
depending on the eigenvalues $a_1$ and $a_n$ of $A$
   \todo{\small $a_1$ does NOT generalize to dim $\infty$}
in~(\ref{eq:A-diag}), but independent of $T\ge 1$, such that
$
  \norm{\Pi_T}\le d
$.
\end{lemma}

\begin{proof}
Let $\zeta\in\Wbb_T$.
The map $\zeta\mapsto\zeta_\E$ is linear. Moreover
$(\zeta_\E)_\E(s)=\zeta_\E(s)$ since
$\zeta_\E^+(-T)=\zeta^+(-T)$ and $\zeta_\E^-(T)=\zeta^-(T)$.
This proves identity 1 in the following
\[
   \Pi_T\circ\Pi_T\stackrel{{\color{gray}1}}{=}\Pi_T ,\qquad
   \im\Pi_T\stackrel{{\color{gray}2}}{=}\E_T ,\qquad
   \ker\Pi_T\stackrel{{\color{gray}3}}{=}\K_T.
\]
Identity 2:
One readily checks that $\p_s\zeta_\E=-A\zeta_\E$, therefore
$\zeta_\E\in\E_T$. Hence $\im\Pi_T\subset\E_T$.
Vice versa, given $\zeta\in\E_T$, the two components
$\zeta^\pm:=p_\pm\zeta$
satisfy for $s\in[-T,T]$, and using~(\ref{eq:A-diag}), the ODE
$\p_s\zeta^+=-A_+\zeta^+$ with initial value $\zeta^+(-T)$ at $s=-T$
and the ODE
$\p_s\zeta^-=A_-\zeta^-$ with initial value $\zeta^-(T)$ at $s=T$.
The solutions are given by $s\mapsto e^{-(s+T)A_+}\zeta^+(-T)$ and
by $s\mapsto e^{(s-T)A_-}\zeta^-(T)$, respectively.
Their direct sum is $\zeta_\E$, hence $\E_T\subset\im \Pi_T$.
\smallskip
\\
%\smallskip
Identity 3:
Pointwise in $s$ vanishing of the vector valued map
$(\Pi_T\zeta)(s)=\left(
e^{-(s+T)A_+}\zeta^+(-T),e^{(s-T)A_-}\zeta^-(T)\right)=(0,0)$
happens iff both components vanish,
that is iff $\zeta\in\K_T$.

\smallskip
To find a bound for $P$, pick $\zeta\in\Wbb_T$.
Straightforward calculation shows that
\begin{equation*}
\begin{split}
   \norm{\zeta_\E}_{\Wbb_T}^2
   &\stackrel{{\color{white}\text{(3.11)}}}{=}
   \norm{\zeta_\E}_{\V_T}^2
   +\norm{\tfrac{d}{ds}\zeta_\E}_{\V_T}^2
\\
   &\stackrel{(\ref{eq:zeta_E})}{=}
   \int_{-T}^T
   \Abs{\left(e^{-(s+T)A_+}\zeta^+(-T),e^{(s-T)A_-}\zeta^-(T)\right)}^2\, ds
   \\
   &\qquad +\int_{-T}^T
   \Abs{\left(-A_+ e^{-(s+T)A_+}\zeta^+(-T),A_- e^{(s-T)A_-}\zeta^-(T)\right)}^2\, ds
\\
   &\stackrel{(\ref{eq:p})}{=}
   \int_0^{2T}
   \Abs{e^{-sA_+}\xi^+(-T)}^2 +\Abs{A_+ e^{-sA_+}\xi^+(-T)}^2\, ds
   \\
   &\qquad +\int_{-2T}^0 \Abs{e^{sA_-}\zeta^-(T)}^2
   +\Abs{A_- e^{sA_-}\zeta^-(T)}^2
   \, ds
\\
   &\stackrel{(\ref{eq:A})}{=}
   \sum_{i=1}^{n-k}\int_0^{2T}
   (1+a_i^2) e^{-2s a_i}\zeta_i(-T)^2 \, ds
   \\
   &\qquad +\sum_{j=n-k+1}^n\int_{-2T}^0
   (1+a_j^2) e^{-2s a_j}\zeta_j(T)^2 \, ds
\\
   &\stackrel{{\color{white}\text{(3.1)}}}{=}
   \sum_{i=1}^{n-k}(1+a_i^2) \zeta_i(-T)^2\tfrac{1-e^{-4Ta_i}}{2a_i}
%   \\
%   &\quad
   +\sum_{j=n-k+1}^n (1+a_j^2) \zeta_j(T)^2 \tfrac{1-e^{4Ta_j}}{-2a_j}
\\
   &\stackrel{(\ref{eq:A-diag})}{\le}
   \tfrac{1+a_1^2}{2\sigma}\sum_{i=1}^{n-k}\zeta_i(-T)^2
   +\tfrac{1+a_n^2}{2\sigma}\sum_{j=n-k+1}^n\zeta_j(T)^2
\\
   &\stackrel{{\color{white}\text{(3.1)}}}{\le}
   \tfrac{d^2}{8}
   \left(\abs{p_+\zeta(-T)}^2+\abs{p_-\zeta(T)}^2\right)
   \quad {\color{gray} , 
   \tfrac{d^2}{8}:= \tfrac{\max\{{\color{red}1+a_1^2},1+a_n^2\}}{2\sigma}}
\\
\end{split}
\end{equation*}
\begin{equation*}
\begin{split}
   &\stackrel{{\color{white}\text{(3.1)}}}{\le}
   \tfrac{d^2}{4}  \norm{\zeta}_{L^\infty([-T,T])}^2
\\
   &\stackrel{(\ref{eq:sob-emb-T})}{\le}
   d^2 \norm{\zeta}_{\Wbb_T}^2 .
\end{split}
\end{equation*}
Here equality two is by definition~(\ref{eq:zeta_E}) of $\zeta_\E$.
Equality three is by the $g_0$-orthogonal splitting~(\ref{eq:p})
which makes the mixed inner products zero.
Equality four uses that $A=\diag(A_+,-A_-)$, by~(\ref{eq:A}),
and the $a_i>0>a_j$ are ordered by~(\ref{eq:A-diag}).
Equality five is by integration.
   \todo[color=yellow!40]{\small can't do (\ref{eq:W-norm-solution}) -- $\zeta$ not ODE solution}
The first inequality uses the order~(\ref{eq:A-diag}) of the matrix
entries $a_\ell$ and definition~(\ref{eq:spec-gap}) of the spectral
gap $\sigma$. The third inequality uses that the projections $p_\pm$
are orthogonal, hence of norm $\le 1$.
The final inequality four is by the, uniform in $T$, Sobolev
estimate~(\ref{eq:sob-emb-T}).
This concludes the proof of Lemma~\ref{le:projection-onto-E}.
\end{proof}

\boldmath
%%%%%%%%%%%%%%%%%%%%%%%%%%%%%%%%%%%
%%%%%%% Subsection:  %%%%%%%%%%%%%%%%%%
%%%%%%%%%%%%%%%%%%%%%%%%%%%%%%%%%%%
\subsection{Infinitesimal gluing map $\Gamma_T$}
\label{sec:Gamma_T}
\unboldmath
%\medskip
%\noindent\boldmath
%\textbf{Infinitesimal gluing map $\Gamma_T$.}
%\unboldmath

\begin{definition}
For $T\ge 3$ we call the linear map defined by the composition
\begin{equation}\label{eq:Gamma_T-def}
   \Gamma_T:=\Pi_T\circ d\wp_T(0_+,0_-)
   \colon \E^+\times \E^-\to\Wbb_T\to \E_T
\end{equation}
the \textbf{infinitesimal gluing map}.
It acts as shown in~(\ref{eq:Gamma_T}) and Figure~\ref{fig:fig-Gamma_T}.
\end{definition}

To obtain a formula for $\Gamma_T$ we proceed in three steps
I--III. Fix elements $\xi=(\xi^+,0)\in \E^+$ and $\eta=(0,\eta^-)\in \E^-$;
see~(\ref{eq:xi-eta-equiv}).

I. Time $s=-T$: By definition~(\ref{eq:wp} of $\wp_T$
and~(\ref{eq:Pp_T}) of the pre-gluing map $\Pp_T$
(using that $1-\beta(-T+2)={\color{cyan}1}$ and
$\beta(-T-2)={\color{brown}0}$ for $T\ge 3$) we obtain
\begin{equation*}
\begin{split}
   d\wp_T(0_+,0_-)(\xi,\eta)(-T)
   &
   \stackrel{(\ref{eq:wp})}{=}
   \left(
   \left.\tfrac{d}{d\tau}\right|_{0}
   \Pp_T(\eps\xi,\eps\eta)
   \right)
   (-T)
\\
   &
   \stackrel{(\ref{eq:Pp_T})}{=}
   {\color{cyan}1}\cdot\xi(T-T)+{\color{brown}\,0}\cdot \eta(-T-T)
   =(\xi^+(0),0).
\end{split}
\end{equation*}
In view of the direct sum $\Wbb_T=\K_T\oplus\E_T$
and since $\Gamma_T(\xi,\eta):=P_{\E_T,\K_T}\circ
d\wp_T(0_+,0_-)(\xi,\eta)$ is the projection to $\E_T$
there is an element $\zeta$ in the projection kernel $\K_T$ such that
$d\wp_T(0_+,0_-)(\xi,\eta)=\zeta+\Gamma_T(\xi,\eta)$.
So we get identity {\color{gray} $1$} in
\begin{equation}\label{eq:Gamma-xi}
   p_+\Bigl(\Gamma_T(\xi,\eta)(-T)\Bigr)
   \stackrel{{\color{gray} 1}}{=}
    p_+\Bigl(d\wp_T(0_+,0_-)(\xi,\eta)(-T)-\zeta(-T)\Bigr)
   \stackrel{{\color{gray} 2}}{=}
   \xi^+(0) .
\end{equation}
Identity {\color{gray} $2$} holds as $p_+(\zeta(-T))=0$
by condition one in definition~(\ref{eq:Kk_T}) of $\K_T$.

II. Time $s=T$: Similarly as in I. we obtain that
\[
   d\wp_T(0_+,0_-)(\xi,\eta)(T)
   \stackrel{(\ref{eq:Pp_T})}{=}\eta(0)
   =(0,\eta^-(0)).
\]
Now use condition two in definition~(\ref{eq:Kk_T}) of $\K_T$ to
conclude that
\begin{equation}\label{eq:Gamma-eta}
   p_-\left(\Gamma_T(\xi,\eta)(T)\right)=\eta^-(0) .
\end{equation}

III. Time $s\in[-T,T]$:
Since $\Gamma_T(\xi,\eta)$ lies in $\E_T$ it satisfies the ODE given by
$\p_s\Gamma_T(\xi,\eta)+A\Gamma_T(\xi,\eta)=0$ and so,
by~(\ref{eq:Gamma-xi}) and~(\ref{eq:Gamma-eta}), we get the formula
\begin{equation}\label{eq:Gamma_T}
   \Gamma_T(\xi,\eta)(s)
   =\left( e^{-(s+T)A_+}\xi^+(0),e^{(s-T)A_-}\eta^-(0)\right) .
\end{equation}
In particular, due to the choice of the complement $\K_T$
which just involves the ends $-T$ and $T$,
the infinitesimal gluing map \textbf{does not depend}
on the choice of the cutoff function $\beta$ used to define the
pre-gluing map~(\ref{eq:Pp_T}).
\begin{figure} %[h]
  \centering
  \includegraphics%[width=0.9\textwidth]
                             %[height=3.5cm]
                             {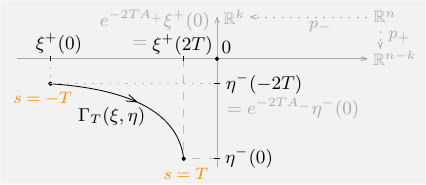}
  \caption{Infinitesimal gluing isomorphism
                 $\Gamma_T\in\Ll(\E^+\times\E^-,\E_T)$\small; cf.~(\ref{eq:Gamma_T})}
%  \caption{Infinitesimal gluing isomorphism
%                 $\Gamma_T\colon\E^+\times\E^-\to\E_T$; see~(\ref{eq:Gamma_T})}
  \label{fig:fig-Gamma_T}
\end{figure}
\begin{lemma}[Norm]\label{le:Gamma_T-iso}
Let $T\ge 3$.
The linear map $\Gamma_T\colon \E^+\times \E^-\to \E_T$ is an
isomorphism of norm $\norm{\Gamma_T}\le1$
and the norm of the inverse is bounded, uniformly in $T$,
   \todo[color=yellow!40]{\small used in {Thm.~\ref{thm:gluing-diffeo}}}
by $k:=1/(1-e^{-12\sigma})$ where $\sigma$ is the spectral gap~(\ref{eq:spec-gap}) of $A$.
\end{lemma}

\begin{proof}
We saw that $\dim \E^+=n-k$, $\dim \E^-=k$, and $\dim \E_T=n$.
Hence it suffices to show injectivity, i.e. that the kernel of
$\Gamma_T$ is trivial.
Given $(\xi,\eta)\in\ker \Gamma_T$, then by~(\ref{eq:Gamma_T})
we have $\xi^+(0)=0$ and $\eta^-(0)=0$. So $\xi^+\equiv 0$ and
$\eta^-\equiv 0$,
since $\xi^+$ and $\eta^-$ are solutions of a linear first order ODE;
see~(\ref{eq:xi-eta-equiv}). Therefore $\xi=(\xi^+,0)\equiv 0$ and
$\eta=(0,\eta^-)\equiv 0$.
This shows that $\Gamma_T$ is an isomorphism whenever $T\ge 3$.

To see that $\Gamma_T$ and ${\Gamma_T}^{-1}$ are bounded, 
uniformly in $T$, consider the identities
\begin{equation*}
\begin{split}
   \norm{\Gamma_T(\xi,\eta)}_{\Wbb_T}^2
   &\stackrel{{\color{white}\text{(3.11)}}}{=}
   \norm{\Gamma_T(\xi,\eta)}_{\V_T}^2
   +\norm{\tfrac{d}{ds}\Gamma_T(\xi,\eta)}_{\V_T}^2
\\
   &\stackrel{(\ref{eq:Gamma_T})}{=}
   \int_{-T}^T
   \Abs{\Bigl(e^{-(s+T)A_+}\xi^+(0),e^{(s-T)A_-}\eta^-(0)\Bigr)}^2\, ds
   \\
   &\qquad +\int_{-T}^T
   \Abs{\Bigl(-A_+ e^{(-s-T)A_+}\xi^+(0),A_- e^{(s-T)A_-}\eta^-(0)\Bigr)}^2\, ds
\\
   &\stackrel{(\ref{eq:p})}{=}
   \int_0^{2T}
   \Abs{e^{-sA_+}\xi^+(0)}^2 +\Abs{A_+ e^{-sA_+}\xi^+(0)}^2\, ds
   \\
   &\qquad +\int_{-2T}^0 \Abs{e^{sA_-}\eta^-(0)}^2
   +\Abs{A_- e^{sA_-}\eta^-(0)}^2
   \, ds
\\
   &\stackrel{(\ref{eq:A})}{=}
   \sum_{i=1}^{n-k}\int_0^{2T}
   (1+a_i^2) e^{-2s a_i}\xi_i(0)^2 \, ds
   \\
   &\quad +\sum_{j=n-k+1}^n\int_{-2T}^0
   (1+a_j^2) e^{-2s a_j}\eta_j(0)^2 \, ds
\\
   &\stackrel{{\color{white}\text{(3.11)}}}{=}
   \sum_{i=1}^{n-k}(1+a_i^2) \xi_i(0)^2\tfrac{1-e^{-4Ta_i}}{2a_i}
%   \\
%   &\quad
   +\sum_{j=n-k+1}^n (1+a_j^2) \eta_j(0)^2 \tfrac{1-e^{4Ta_j}}{-2a_j}
\end{split}
\end{equation*}
where equality two is by formula~(\ref{eq:Gamma_T}) for $\Gamma_T$.
Equality three is by the $g_0$-orthogonal splitting~(\ref{eq:p}).
Equality four uses that $A=\diag(A_+,-A_-)$, by~(\ref{eq:A}),
and the $a_i>0>a_j$ are ordered by~(\ref{eq:A-diag}).\
Equality five is by integration.

The $\Wbb_+$ norm of $\xi=(\xi^+,0)\in \E^+$,
see~(\ref{eq:xi-eta-equiv}), is given by
\begin{equation}\label{eq:W-norm-solution}
\begin{split}
   \norm{\xi}_{\Wbb_+}^2
   &\stackrel{{\color{white}\text{(3.11)}}}{=}
   \norm{\xi}_{\V_+}^2+\norm{\tfrac{d}{ds}\xi}_{\V_+}^2
\\
   &\stackrel{\text{(\ref{eq:xi-eta-equiv})}}{=}
   \int_0^\infty\abs{e^{-sA_+}\xi^+(0)}^2\, ds
   +\int_0^\infty\abs{A_+e^{-sA_+}\xi^+(0)}^2\, ds
\\
   &\stackrel{\text{(\ref{eq:A-diag})}}{=}
   \sum_{i=1}^{n-k}\int_0^\infty (1+a_i^2)e^{-2sa_i}\xi_i(0)^2\, ds
\\
   &\stackrel{{\color{white}\text{(3.11)}}}{=}
   \sum_{i=1}^{n-k}\tfrac{1+a_i^2}{2a_i} \xi_i(0)^2
\end{split}
\end{equation}
and analogously for the $\Wbb_-$ norm of $\eta=(0,\eta^-)\in \E^-$.

To see that ${ \Gamma_T}^{-1}$ is bounded, uniformly in $T$, we estimate
$\Gamma_T$ from below
\begin{equation*}
\begin{split}
   \norm{\Gamma_T(\xi,\eta)}_{\Wbb_T}^2
   &=\sum_{i=1}^{n-k}(1+a_i^2) \xi_i(0)^2\tfrac{1-e^{-4Ta_i}}{2a_i}
   +\sum_{j=n-k+1}^n (1+a_j^2) \eta_j(0)^2 \tfrac{1-e^{4Ta_j}}{-2a_j}
   \\
   &\ge\left(1-e^{-12\sigma}\right)\left(
   \sum_{i=1}^{n-k}\tfrac{1+a_i^2}{2a_i} \xi_i(0)^2
   +\sum_{i=n-k+1}^{n}\tfrac{1+a_j^2}{-2a_j} \eta_j(0)^2\right)
   \\
   &= \left(1-e^{-12\sigma}\right) \left(\norm{\xi}_{\Wbb_+}^2+\norm{\eta}_{\Wbb_-}^2\right).
\end{split}
\end{equation*}
To obtain the inequality we use the assumption $T\ge3$ and the spectral
gap $\sigma$ of $A$ defined by~(\ref{eq:spec-gap}) of $A$.
The last equality is explained right above.

To see that $\Gamma_T$ is bounded, uniformly in $T$, we estimate
$\Gamma_T$ from above
\begin{equation*}
\begin{split}
   \norm{\Gamma_T(\xi,\eta)}_{\Wbb_T}^2
   &=\sum_{i=1}^{n-k}(1+a_i^2) \xi_i(0)^2\tfrac{1-e^{-4Ta_i}}{2a_i}
   +\sum_{j=n-k+1}^n (1+a_j^2) \eta_j(0)^2 \tfrac{1-e^{4Ta_j}}{-2a_j}
   \\
   &\le \sum_{i=1}^{n-k}\tfrac{1+a_i^2}{2a_i} \xi_i(0)^2
   +\sum_{j=n-k+1}^n \tfrac{1+a_j^2}{-2a_j} \eta_j(0)^2
   \\
   &=\norm{\xi}_{\Wbb_+}^2+\norm{\eta}_{\Wbb_-}^2
\end{split}
\end{equation*}
where the inequality holds since $1-e^{-4Ta_i}\le1$ and
$1-e^{4Ta_j}\le1$.
\end{proof}

\boldmath
%%%%%%%%%%%%%%%%%%%%%%%%%%%%%%%%%%%
%%%%%%% Subsection:  %%%%%%%%%%%%%%%%%%
%%%%%%%%%%%%%%%%%%%%%%%%%%%%%%%%%%%
\section{Newton-Picard map}
\label{sec:NP-map-appl}
\unboldmath

Given two elements $w_+\in \Ww^{\rm s}$ and $w_-\in \Ww^{\rm u}$
near the critical point, we view the pre-glued path
\[
   w_T:=\wp_T(w_+,w_-) \in \Wbb_T
\]
as an approximate flow trajectory, equivalently an approximate
zero~$x_1$ of the section~$\Ff_T$, and then detect a nearby solution
using the implicit function theorem with initial point $x_0:=0_T$, see
Appendix~\ref{sec:Newton-Picard-map}.
Thus we need that $\Ff_T(w_T)$ is suitably close to zero.
We also need a uniformly in $T$ bounded right inverse $Q_T$
of the linearization $D_T:=d\Ff_T(0_T)$.
These are the next two subsections.

\boldmath
%%%%%%% Subsubsection:  %%%%%%%%%%%%%%%%
%%%%%%%%%%%%%%%%%%%%%%%%%%%%%%%%%%%
\subsection{Approximate zero}
\label{sec:approx-zero}
\unboldmath

\begin{proposition}[Pre-glued path is approximate zero]
\label{prop:approx-zero}
Pick $\eps\in(0,\sigma)$ where $\sigma=\sigma(A)$ is the spectral
gap~(\ref{eq:spec-gap}).
Let $m\in\N_0$ and choose compact neighborhoods $K_+=K_+(m)$ of $0_+\in
\Ww^{\rm s}$ in $T^m \Ww^{\rm s}$ and $K_-=K_-(m)$ of
$0_-\in \Ww^{\rm u}$ in $T^m \Ww^{\rm u}$.
Then the following is true. 
Given two elements
\[
   W_+\in K_+\subset T^m \Ww^{\rm s}
   ,\qquad
   W_- \in K_-\subset T^m \Ww^{\rm u} ,
\]
pre-glue them with the $m$-fold tangent map of $\wp_T$, see~(\ref{eq:wp}), to get
\[
   W_T:=T^m\wp_T(W_+,W_-) .
%=\left(\wp_T(w_+,w_-),d\wp_T|_{(w_+,w_-)}(\xi,\eta)\right)
%   =:\left(w_T,\zeta_T\right).
\]
Then there is a constant $C(K_+,K_-)$, independent of $T$, such~that
\begin{equation}\label{eq:T^mFf-decay}
\begin{split}
   \Norm{T^m\Ff_T(W_T)}_{T^m\V_T}
   \le C(K_+,K_-)\cdot e^{-\eps T}
\end{split}
\end{equation}
whenever $T\ge 3$.
\end{proposition}

%The proof of the proposition will be by induction on $m$.
%The following Lemma~\ref{le:approx-zero} covers the cases $m=0,1$.
%\newpage

\begin{proof}
Let $m\in\N_0$. An
   \todo[color=yellow!40]{\tiny $m=0$, $1$ App.~\ref{app:TM}}
element $W_+\in T^m \Ww^{\rm s}$ is a map satisfying the equation
\begin{equation}\label{eq:TW^s-2}
   W_+\colon[0,\infty)\to T^m\R^n,
   \qquad
   \p_sW_++T^m\Nabla{}f(W_+)=0 .
\end{equation}
An element $W_-\in T^m \Ww^{\rm s}$ is a map that satisfies the equation
\begin{equation}\label{eq:TW^u-2}
   W_-\colon(-\infty,0]\to T^m\R^n,
   \qquad
   \p_sW_-+T^m\Nabla{}f(W_-)=0 .
\end{equation}
Since $0$ is a non-degenerate critical point of $f$ the solutions $W_+$
and $W_-$ decay with all their derivatives exponentially.
   \todo[color=yellow!40]{\small use:\\ *uniform* exp. decay}
In particular, by Theorem~\ref{thm:lin-unif-exp-decay} and compactness
of $K_\pm$, there is a constant $c=c(K_+,K_-)$ with
\begin{equation}\label{eq2:exp-dec-W}
   \abs{W_+(s)}+\abs{\p_s W_+(s)}\le c e^{-\eps s}
   ,\quad
   \abs{W_-(s)}+\abs{\p_s W_-(s)}\le c e^{\eps s} ,
\end{equation}
for every $s\ge0$, respectively $s\le0$.
Since $f$ has a critical point at the origin,
we get $T^m\Nabla{}f(0)=0$ where $0\in T^m\R^n$.\footnote{
  It holds $T^1\Nabla{}f(0,0)=\left(\Nabla{}f(0),D\Nabla{}f(0) 0\right)=(0,0)$,
%  analogously
  similarly $T^m\Nabla{}f(0,\dots,0)=(0,\dots,0)$.
}
Hence there is 
   \todo[color=yellow!40]{\small by Taylor}
a constant $\mu_m>0$ with
\begin{equation}\label{eq:mu_m}
   \Abs{T^m\Nabla{}f(W)}\le \mu_m\abs{W}
\end{equation}
whenever $\abs{W}\le c$.

\smallskip
By linearity of the pre-gluing map $\Pp_T$,
the $m-$fold tangent map is given by
\begin{equation*}
\begin{split}
   W_T:=T^m\Pp_T(W_+,W_-)
   &\stackrel{\text{(\ref{eq:Pp_T})}}{=}
   \left(1-\beta(s+2)\right) W_+(T+s)
   +\beta(s-2)\, W_-(-T+s) 
\end{split}
\end{equation*}
pointwise at $s\in[-T,T]$. Similarly,
   \todo[color=yellow!40]{\small linearity}
in analogy to~(\ref{eq:w_T}), we have
\begin{equation}\label{eq:zeta_TT2}
   W_T(s) %:=T\Pp_T(W_+,W_-)(s)=
   =
\begin{cases}
   W_+(T+s)&\text{, $s\in[-T,-3]$}\\
   \left(1-\beta(s+2)\right) W_+(T+s)&\text{, $s\in[-3,-1]$}\\
   0&\text{, $s\in[-1,1]$}\\
   \beta(s-2) W_-(-T+s)&\text{, $s\in[1,3]$}\\
   W_-(-T+s)&\text{, $s\in[3,T]$}
\end{cases}
\end{equation}
for $s\in[-T,T]$.
The tangent map of $\Ff_T(w):=\p_sw+\Nabla{}f(w)$ at $W_T$ is given by
\begin{equation}\label{eq:TFf_T}
   T^m\Ff_T(W_T)
   =\p_s W_T+T^m\Nabla{}f(W_T) .
\end{equation}
\noindent
Now there are three cases.
\newline
1)~For $s\in [-T, -3]\cup[-1,1]\cup[3,T]$ the map $T^m\Ff_T(W_T)$
vanishes:
\begin{equation*}
\begin{split}
   T^m\Ff_T(W_T)(s)
   \stackrel{\text{(\ref{eq:zeta_TT2})}}{=}
   \p_sW_+(T+s)+T^m\Nabla{}f|_{W_T(T+s)}
   &\stackrel{\text{(\ref{eq:TW^s-2})}}{=} 0 ,\quad \forall s\in[-T,-3].
\\
   T^m\Ff_T(W_T)(s)
   \stackrel{\text{(\ref{eq:zeta_TT2})}}{=}
   \p_sW_-(-T+s)+T^m\Nabla{}f|_{W_-(-T+s)}
   &\stackrel{\text{(\ref{eq:TW^u-2})}}{=} 0 ,\quad \forall s\in[3,T].
\\
   T^m\Ff_T(W_T)(s)
   =0\quad\text{since}\quad W_T(s)
   &\stackrel{\text{(\ref{eq:zeta_TT2})}}{=} 0 ,\quad \forall s\in[-1,1].
\end{split}
\end{equation*}
2) For $s\in[-3,-1]$, using~(\ref{eq:TFf_T}) for $T^m\Ff_T(W_T)$
and~(\ref{eq:zeta_TT2}) for $W_T$, we obtain
\begin{equation*}
\begin{split}
   T^m\Ff_T(W_T)(s)
   &= %\stackrel{\text{(\ref{eq:zeta_TT2})}\atop\text{(\ref{eq:TFf_T})}}{=}
   -\beta^\prime(s+2) W_+(T+s)
   +(1-\beta(s+2))\p_s W_+(T+s)
\\
   &\qquad
   +T^m \Nabla{}f|_{(1-\beta(s+2)) W_+(T+s)} .
\end{split}
\end{equation*}
Now, by~(\ref{eq:mu_m}) and~(\ref{eq:TW^s-2}), we estimate the
pointwise length by
\begin{equation*}
\begin{split}
   &\Abs{T^m\Ff_T(W_T)(s)}
\\
   &\le
   \norm{\beta^\prime}_\infty\abs{W_+(T+s)}
   +\abs{\p_s W_+(T+s)}
   +\Abs{T^m\nabla f|_{(1-\beta(s+2))W_+(T+s)}}
\\
   &\le\left(\norm{\beta^\prime}_\infty+1+\mu_m\right)
   c\cdot  e^{-\eps(s+T)} .
\end{split}
\end{equation*}
3) For $s\in[1,3]$ we obtain analogously the formula
\begin{equation*}
\begin{split}
   T^m\Ff_T(W_T)(s)
   &=
   -\beta^\prime(s-2)W_-(-T+s)
   +\beta(s-2)\p_sW_-(-T+s)
   \\
   &\qquad
   +T^m\Nabla{} f|_{\beta(s-2)W_-(-T+s)}
\end{split}
\end{equation*}
and the estimate $\Abs{T^m\Ff_T(W_T)(s)}
\le\left(\norm{\beta^\prime}_\infty+1+\mu_m\right) c e^{\eps(-T+s)}$.

Thus for the $L^2$ norm we get, by integration, the estimate
\begin{equation*}
\begin{split}
   \Norm{T^m\Ff_T (W_T)}_{T^m\V_T}^2
   \le \underbrace{2\left(\norm{\beta^\prime}_\infty+1+\mu_m\right)^2 
   c^2 \frac{e^{6\eps}-e^{2\eps}}{\eps}}_{=:C(K_+,K_-)^2}
   e^{-2\eps T} 
\end{split}
\end{equation*}
where $c=c(K_+,K_-)$ and $K_\pm$ depend on $m$. This proves
Proposition~\ref{prop:approx-zero}.
\end{proof}

\boldmath
%%%%%%%%%%%%%%%%%%%%%%%%%%%%%%%%%%%
%%%%%%% Subsection:  %%%%%%%%%%%%%%%%%%
%%%%%%%%%%%%%%%%%%%%%%%%%%%%%%%%%%%
\subsection{Surjectivity and right inverse}
\label{sec:surjectivity}
\unboldmath

%\smallskip
%\noindent
\textsc{Surjective linearization at $0_T$.}
Let $T>0$.
The linearization at general $w\in\Wbb_T$ is given by
\begin{equation}\label{eq:dFf}
   d\Ff_T(w)
   \colon\Wbb_T \to \V_T  
   ,\quad
   \zeta\mapsto\p_s\zeta+A_w\zeta 
   ,\qquad A_w(s):=\nabla\Nabla{}f(w(s)) ,
\end{equation}
where $\nabla\Nabla{}f(w(s))$ is the Jacobian of the vector field
$\Nabla{}f\colon\R^n\to\R^n$ at $w(s)$.
The linearization at the origin $0_T$, namely the operator
\begin{equation}\label{eq:D_T}
   D_T:=d\Ff_T(0_T)
   \colon\Wbb_T=\E_T\oplus \K_T\to\V_T
   ,\quad
   \zeta\mapsto\p_s\zeta+A\zeta ,
\end{equation}
where $A_{0_T}=A=\diag(A_+,-A_-)$ is the diagonal block matrix~(\ref{eq:A}),
is surjective: Given $\eta\in\V_T$, then the element defined by
\begin{equation}\label{eq:var-const}
   \zeta(s)
   :=e^{-sA}\left(\zeta_0+\int_0^s e^{\sigma A}\eta(\sigma)\, d\sigma\right)
   ,\quad
   \zeta_0\in\R^n,
\end{equation}
for $s\in[-T,T]$,
lies in $\Wbb_T$ and
satisfies $d\Ff_T(0_T)\zeta=\eta$.\footnote{
  Since the interval $[-T,T]$ is finite, the Morse condition is not
  needed here.
  }
For $\Mm_T:=\Ff_T^{-1}(0)$ there is the natural inclusion
$\E_T:=T_{0_T}\Mm_T\subset\ker d\Ff_T(0_T)$.
On the other hand, both spaces are determined by the initial
conditions which are given by $\R^n$, thus
$\dim \E_T=n=\dim\ker D_T$. Therefore the two spaces coincide
\begin{equation}\label{eq:D_T=E_T}
   \ker D_T=\E_T .
\end{equation}
Since $\K_T$ is a complement of $\E_T$,
the restriction
\begin{equation}\label{eq:F_T}
{\color{gray}
   \begin{pmatrix}F_T^+&0\\0&F_T^-\end{pmatrix}
   =
}\,
   F_T
   :=d\Ff_T(0_T)|_{\K_T}
{\color{gray}\,
   =
   \begin{pmatrix}\frac{d}{ds}+A_+&0\\0&\frac{d}{ds}-A_-\end{pmatrix}
   \biggr|_{\K_T}
}
   \colon \K_T\to \V_T
\end{equation}
is injective, hence a continuous linear bijection.
Hence, by the open mapping theorem, the inverse
\begin{equation}\label{eq:Q_T}
{\color{gray}
   \begin{pmatrix}Q_T^+&0\\0&Q_T^-\end{pmatrix}
   =
}\,
   Q_T
   := {F_T}^{-1}
{\color{gray}\,
   =\begin{pmatrix}(F_T^+)^{-1}&0\\0&(F_T^-)^{-1}\end{pmatrix}
}
   \colon\V_T\to \K_T
\end{equation}
is also continuous.
Thus the map $F_T\colon \K_T\to \V_T$ is a
Hilbert space isomorphism.

\medskip
\noindent
\textsc{Right inverse $Q_T$ of $D_T$.}
\begin{remark}[$Q_T$ is a right inverse of $D_T$]
Given $\eta\in\V_T$, then $\zeta:=Q_T\eta\in \K_T$, hence
$d\Ff_T(0_T)\zeta=F_T\zeta$. Therefore
\[
   d\Ff_T(0_T)\circ Q_T \eta
   =F_T\circ Q_T \eta
   =\eta .
\]
\end{remark}

\begin{lemma}\label{le:Q_T}
There is a constant $c$, independent of $T>0$,
such that $\norm{Q_T}\le c$.
\end{lemma}

\begin{proof}
In the proof we distinguish three cases.

\smallskip
\noindent
\textsc{I. $A$ positive definite:} In this case $k=0$ and $p_-=0$,
see~(\ref{eq:p}), in particular $a_1\ge \dots \ge a_n>0$.
Thus $\K_T=\{\xi\in\Wbb_T\mid \xi(-T)=0\}$;
see~(\ref{eq:Kk_T}).
Given $\eta\in\V_T$, let $\zeta:=Q_T\eta$,
equivalently $\eta=F_T\zeta$.
Since $\zeta\in \K_T$, we know that $\zeta(-T)=0$.
By~(\ref{eq:var-const}), where we changed the start of the integration
from $0$ to $-T$, we get the formula
\begin{equation*}
\begin{split}
   \zeta(s)
   &=e^{-sA}
   \left(\zeta(-T)+\int_{-T}^s e^{\sigma A}\eta(\sigma)\,
     d\sigma\right)
   \stackrel{\zeta(-T)=0}{=}
   \int_{-T}^s e^{-(s-\sigma)A}\eta(\sigma)\, d\sigma\\
   &=\int_{[-T,s]\cup{\color{red}(s,T}]}\phi(s-\sigma)\eta(\sigma)\, d\sigma
   =(\phi * \eta)(s)
\end{split}
\end{equation*}
whenever $s\in[-T,T]$ and where the function $\phi$ is defined by
$\phi(r):=e^{-rA}$ for $r\ge 0$, and by {\color{red} $0$ for $r<0$}.
Hence, by Young's inequality, we have
\begin{equation*}
   \norm{\zeta}_2\le\norm{\phi}_1\norm{\eta}_2
   \le \frac{1}{a_n}\norm{\eta}_2
\end{equation*}
where the $L^2$ and $L^1$ norms are over $[-T,T]$
and since
\begin{equation*}
\begin{split}
   \norm{\phi}_1=\int_{-T}^T\norm{\phi(s)}_{\Ll(\R^n)}\, ds
   &=\int_0^T\norm{e^{-sA}}_{\Ll(\R^n)}\, ds\\
   &=\int_0^T e^{-sa_n}\, ds
   =\tfrac{1-e^{-a_n T}}{a_n}
   \le\tfrac{1}{a_n}.
\end{split}
\end{equation*}
Note that $a_n>0$ is the smallest eigenvalue of the positive definite
operator $A$.
Since $\p_s\zeta=\eta-A\zeta$, and by the triangle inequality and
$ \norm{\zeta}_2\le \frac{1}{a_n}\norm{\eta}_2$,
   \todo[color=yellow!40]
      {\tiny $\norm{A\zeta}_2\le\norm{\phi}_1\norm{A\eta}_2$
      \\\mbox{ } \\
      $\norm{A\eta}_2^2\le
      a_1^2\norm{\eta_1}_2^2+\dots+a_n^2\norm{\eta_n}_2^2$
      $\le a_1^2\norm{\eta}_2^2$}
we get
\begin{equation*}
\begin{split}
   \norm{Q_T\eta}_{1,2}^2
   =\norm{\zeta}_{1,2}^2
   &=\norm{\p_s\zeta}_{2}^2+\norm{\zeta}_{2}^2\\
   &=\norm{\eta-A\zeta}_{2}^2+\norm{\zeta}_{2}^2\\
   &\le(\norm{\eta}_{2}+\norm{A\zeta}_{2})^2+\norm{\zeta}_{2}^2\\
   &\le\left(1+\tfrac{a_1}{a_n}\right)^2 \norm{\eta}_{2}^2
   +\tfrac{1}{a_n^2} \norm{\eta}_{2}^2 .
\end{split}
\end{equation*}
This proves Step~1 for
   \todo{\small $a_1$ does NOT generalize to dim $\infty$}
$c^2=\frac{(a_1+a_n)^2+1}{a_n^2}$.

\smallskip
\noindent
\textsc{II. $A$ negative definite:} So $k=n$ and $p_-=\1$
and $\K_T=\{\xi\in\Wbb_T\mid \xi(T)=0\}$.
Given $\eta\in\V_T$, let $\zeta:=Q_T\eta$,
equivalently $\eta=F_T\zeta$.
Since $\zeta\in \K_T$, we know that $\zeta(T)=0$.
By~(\ref{eq:var-const}), where we changed the start of the integration
from $0$ to $T$, we obtain the formula
\begin{equation*}
\begin{split}
   \zeta(s)
   &=e^{s A}
   \left(\zeta(T)+\int_{T}^s e^{-\sigma A}\eta(\sigma)\,
     d\sigma\right)
   \stackrel{\zeta(T)=0}{=}
   -\int_{s}^T e^{(s-\sigma)A}\eta(\sigma)\, d\sigma\\
   &=-\int_{{\color{red}[-T,s)}\cup[s,T]} \phi(s-\sigma)\eta(\sigma)\, d\sigma
   =-(\phi * \eta)(s)
\end{split}
\end{equation*}
whenever $s\in[-T,T]$ and where $\phi$ was defined
in Step~1. Continue as in Step~1.

\smallskip
\noindent
\\
\textsc{III. General case:}
Given $\eta=(\eta^+,\eta^-)\in\V_T$.
Let $\zeta:=Q_T\eta$, then
\[
   \zeta
   =(\zeta^+,\zeta^-)
   ,\qquad \zeta^+=Q_T^+\eta^+
   ,\quad \zeta^-=Q_T^-\eta^- .
\]
From Step~I and Step~II there exists a constant $c>0$ such that
$\norm{\zeta^+}_{1,2}\le c\norm{\eta^+}_2$ and
$\norm{\zeta^-}_{1,2}\le c\norm{\eta^-}_2$.
Since the splitting $\R^{n-k}\times\R^k$ is orthogonal, we have
\[
   \norm{\zeta}_{1,2}^2
   \stackrel{\perp}{=}\norm{\zeta^+}_{1,2}^2+\norm{\zeta^-}_{1,2}^2
   \le c^2 \norm{\eta^+}_{2}^2+c^2 \norm{\eta^-}_{2}^2
   \stackrel{\perp}{=} c^2\norm{\eta}_2^2.
\]
This proves Step~III and Lemma~\ref{le:Q_T}.
\end{proof}

\boldmath
%%%%%%% Subsubsection:  %%%%%%%%%%%%%%%%
%%%%%%%%%%%%%%%%%%%%%%%%%%%%%%%%%%%
\subsection{Definition of $\Nn_T$}
\label{sec:def-NP}
\unboldmath

Let $c$ be the right inverse bound from Lemma~\ref{le:Q_T}.
In order to use later on Remark~\ref{rem:NP-map-mu}
to satisfy hypothesis~(\ref{eq:A.3.4-NEW-modified}), 
as opposed to only~(\ref{eq:A.3.4-NEW}),
we define, for $\mu\ge 2$, a nested family
   \todo[color=yellow!40]{\small HERE $\mu$ APPEARS}
of open neighborhoods of $0$ in $\R^n$
as the pre-image of $[0,1/\mu c)$ under the continuous
map $\norm{d\Nabla{} f(\cdot)-A}\colon\R^n\to[0,\infty)$, in symbols
\[
   B^\mu
   :=\Norm{d\Nabla{} f(\cdot)-A}^{-1}[0,\tfrac{1}{\mu c}),
   \quad B^\mu\subset B^2 .
\]
For $T>0$ define an open neighborhood $\Bb^\mu_T$ of $0_T$ in
$\Wbb_T$
   \todo[color=green!40]{\small $\Bb^\mu_T$ defined}
by
$$
   \Bb^\mu_T:=\{w\in\Wbb_T\mid w(s)\in B^\mu\; \forall s\in[-T,T]\}
   \subset\Bb^2_T.
$$

\begin{lemma}\label{le:bd-diff-lin}
Let $T>0$ and $\mu\ge 2$. If $w\in \Bb^\mu_T$, then
$\norm{d\Ff_T(w)-D_T}\le\frac{1}{\mu c}$.
\end{lemma}

\begin{proof}
Given $w\in \Bb^\mu_T$, there is the estimate
\[
   \norm{(d\Ff_T(w)-D_T)\zeta}_2
   \stackrel{(\ref{eq:dFf})}{=}\norm{(d\Nabla{}f(w)-A)\zeta}_2
   \le\tfrac{1}{\mu c}\norm{\zeta}_2
   \le\tfrac{1}{\mu c}\norm{\zeta}_{1,2} 
\]
for every $\zeta\in \Wbb_T$.
\end{proof}

\begin{corollary}[{to Lemma~\ref{le:sob-emb-T}}]\label{cor:delta}
There is a monotone decreasing function
$\delta\colon [2,\infty)\to(0,\infty)$,
$\mu\mapsto\delta(\mu)=:\delta_\mu$,
independent of $T$,
such that for $\mu\in [2,\infty)$
the $\delta_\mu$-ball about $0_T$ in $\Wbb_T$ 
is contained in $\Bb^\mu_T$,
in symbols $B_{\delta_\mu}(0_T;\Wbb_T)\subset \Bb^\mu_T$.
\end{corollary}

\begin{proof}
By Lemma~\ref{le:sob-emb-T}
for $w\in B_{\delta_\mu}(0_T;\Wbb_T)$ we have
$\norm{w}_\infty\le 2\delta_\mu$.
\end{proof}

\begin{definition}[Newton-Picard map]%\label{def:NP-map-1}
Let $c>0$ be the right inverse bound of Lemma~\ref{le:Q_T} and let
\begin{equation}\label{eq:delta_4}
   \delta:=\delta_4>0
\end{equation}
be the value in Corollary~\ref{cor:delta} for $\mu=4$.\footnote{
  To define the Newton-Picard map via the McDuff-Salamon
  Proposition~\ref{prop:wonder-prop-NEW} and obtain $C^0$-convergence
  (Theorem~\ref{thm:loc-gluing-C^m} with $m=0$)
  for the gluing map it is sufficient
  to pick $\delta_\mu$ for $\mu=2$.
  However, to obtain $C^1$-convergence
  (Theorem~\ref{thm:loc-gluing-C^m} with $m=1$)
  we need to choose $\mu=4$ in order to satisfy
  assumption~(\ref{eq:A.3.4-NEW-modified})
  ($\frac{1}{4c}$ as opposed to $\frac{1}{2c}$)
  in the tangent map Theorem~\ref{thm:tangent-map}.
  }
%  {eq:A.3.4-NEW}
For $T\ge 1$ we can now, in view of Lemma~\ref{le:bd-diff-lin} with
$\mu=4$ and the McDuff-Salamon
Proposition~\ref{prop:wonder-prop-NEW} with $x_0=0_T$,
define a \textbf{Newton-Picard map}
\begin{equation}\label{eq:NP-map-app}
   \Nn_T\colon\Wbb_T
   \supset\Bb^4_T\supset 
   U_0(\delta)\stackrel{\text{(\ref{eq:U_0(delta)})}}{:=}
   \Bigl(
      B_{\frac{\delta}{8}}(0_T;\Wbb_T)\cap\left\{\norm{\Ff_T}<\tfrac{\delta}{4c}\right\}
   \Bigr)
   \to \Wbb_T .
\end{equation}
Here the inclusion $U_0(\delta)\subset \Bb^4_T$ holds by Corollary~\ref{cor:delta}.
\end{definition}

By~(\ref{eq:wonder-prop-NEW}) the Newton-Picard map $\Nn_T$
enjoys the following properties:
\[
   \Ff_T\circ \Nn_T(w)=0
   ,\qquad
   \Nn_T(w)-w\in\im Q_T
   ,\qquad
   \Nn_T(w)\in B_\delta(0_T;\Wbb_T) ,
\]
and, moreover, one has the estimate
\begin{equation}\label{eq:wonder-prop-diff}
   \norm{(\Nn_T-\id) (w)}_{\Wbb_T}\le 2c\norm{\Ff_T(w)}_{\V_T} .
\end{equation}
Furthermore, by Corollary~\ref{cor:NP-map}, respectively
identity~(\ref{eq:1-P})), we have
\begin{equation}\label{eq:NP-0_T}
   \Nn_T(0_T)=0_T ,\qquad
   d\Nn_T|_{0_T}\stackrel{\text{(\ref{eq:1-P})}}{=}\Pi_T ,
\end{equation}
where the projection $\Pi_T$, see~(\ref{eq:Pi_T}), is uniformly
bounded in $T$, by Lemma~\ref{le:projection-onto-E}.

\boldmath
%%%%%%% Subsubsection:  %%%%%%%%%%%%%%%%
%%%%%%%%%%%%%%%%%%%%%%%%%%%%%%%%%%%
\subsubsection*{Pre-gluing takes values in domain of Newton-Picard map}
%\label{sec:neqf}
\unboldmath

%\medskip
The next lemma and Proposition~\ref{prop:approx-zero}
show that, for $T\ge 3$ large enough, the
pre-gluing map $\Pp_T$ takes values in the domain of the Newton-Picard
map $\Nn_T$.

\begin{lemma}[The neighborhoods $\Uu^\mu_\pm$.]
\label{le:Uu_pm}
Let $\delta\colon [2,\infty)\to(0,\infty)$, $\mu\mapsto\delta_\mu$,
be the monotone decreasing function
in Corollary~\ref{cor:delta}. We abbreviate $\delta:=\delta_4$.
Then there exists a nested family of open and bounded neighborhoods
$\Uu^\mu_+\subset \Ww^{\rm s}$ of $0_+$ and
$\Uu^\mu_-\subset \Ww^{\rm u}$ of $0_-$ such that
$\norm{\Pp_T(w_+,w_-)}_{1,2}<\min\{\frac{\delta}{8},\delta_\mu\}$
whenever $\mu\ge2$, $w_+\in \Uu^\mu_+$, $w_-\in \Uu^\mu_-$, and $T\ge 3$.
\end{lemma}

While the estimate by $\frac{\delta}{8}$ serves in~(\ref{eq:NP-map-app}),
the estimate by $\delta_\mu$ for some $\mu\ge5$ will be used in the proof of
Theorem~\ref{thm:gluing-diffeo} further below.

\begin{proof}
For $\delta^\prime>0$ let $B_{\delta^\prime}(0_+)\subset \Wbb_+$ be the
open radius $\delta^\prime$ ball about $0_+$, analogously for
$B_{\delta^\prime}(0_-)$. 
Pick $w_+\in B_{\delta^\prime}(0_+)\cap\Ww^{\rm s}$
and $w_-\in B_{\delta^\prime}(0_-)\cap\Ww^{\rm u}$
and, for $T\ge 3$, abbreviate $w_T:=\Pp_T(w_+,w_-)\colon[-T,T]\to\R^n$.
Since by~(\ref{eq:w_T}) at each time $s$
at most one of $w_+(T+s)$ and $w_-(-T+s)$ 
comes with a nonzero factor,
we obtain inequality one in the following estimate
\begin{equation*}
\begin{split}
   &\norm{w_T}_{1,2}^2\\
   &=\norm{w_T}_{2}^2+\norm{\p_s w_T}_{2}^2\\
   &\le\norm{w_+}_2^2+\norm{w_-}_2^2
   +2\norm{\beta^\prime}_\infty^2
   \left(\norm{w_+}_2^2+\norm{w_-}_2^2\right)
   +2\left(\norm{\p_sw_+}_2^2+\norm{\p_sw_-}_2^2\right)
  \\
   &\le
   2(1+\norm{\beta^\prime}_\infty^2) \left(\norm{w_+}_{1,2}^2
   +\norm{w_-}_{1,2}^2\right)
   \\
   &\le 4(1+\norm{\beta^\prime}_\infty^2)(\delta^\prime)^2 .
   \\
\end{split}
\end{equation*}
Choose $\delta^\prime=\delta^\prime(\mu):=
\frac{1}{4}\min\{\frac{\delta}{8},\delta_\mu\}
\sqrt{1+\norm{\beta^\prime}_\infty^2}$
and define the open $\delta^\prime$-neighborhoods
in the stable, respectively unstable, manifolds as follows
\[
   \Uu^\mu_+
   := \left( B_{\delta^\prime}(0_+)\cap\Ww^{\rm s}\right)
   \subset \Uu_+:=\Uu^2_+
   ,\quad
   \Uu^\mu_-
   :=\left( B_{\delta^\prime}(0_-)\cap\Ww^{\rm u}\right)
   \subset \Uu_-:=\Uu^2_- .
\]
Then the lemma holds by the previous displayed estimate.
\end{proof}

\boldmath
%%%%%%% Subsubsection:  %%%%%%%%%%%%%%%%
%%%%%%%%%%%%%%%%%%%%%%%%%%%%%%%%%%%
\section{Gluing}
\label{sec:gluing}
\unboldmath

Pick $\eps\in(0,\sigma)$ where $\sigma=\sigma(A)$ is the spectral
gap~(\ref{eq:spec-gap}).
Let $c>0$ be the constant in the right inverse estimate,
Lemma~\ref{le:Q_T}.
Let $\delta=\delta_4>0$ be the constant in Corollary~\ref{cor:delta} and let
\begin{equation}\label{eq:Uu_+-}
   \Uu_{+/-}:=\Uu^2_{+/-}\subset \Ww^{\rm s/u}
   ,\qquad
   K_\pm:=\cl \Uu_\pm,
\end{equation}
be the open sets in Lemma~\ref{le:Uu_pm} and, respectively,
the compact sets given by the closure of
$\Uu_\pm$ in the (finite dimensional) stable/unstable manifold.
Thus, by Proposition~\ref{prop:approx-zero} for $m=0$,
we get a constant $C=C(K_+,K_-)>0$.
Pick $T_0\ge3$ such that
\begin{equation}\label{eq:T_0}
   C e^{-\eps T_0}<\frac{\delta}{4c}
\end{equation}
see~(\ref{eq:T^mFf-decay}).
By Proposition~\ref{prop:approx-zero} for $m=0$ and
Lemma~\ref{le:Uu_pm} it holds that
\begin{equation}\label{eq:cond-1}
   \norm{\Ff_T\circ \wp_T(w_+,w_-)}
   \stackrel{(\ref{eq:T^mFf-decay})}{<}\tfrac{\delta}{4c}
   ,\qquad
   \norm{\wp_T(w_+,w_-)}<\tfrac{\delta}{8} ,
\end{equation}
whenever $T\ge T_0$ and $w_\pm\in\Uu_\pm$
and were $\wp_T:=\Pp_T|_{\Ww^{\rm s}\times \Ww^{\rm u}}
   \colon \Ww^{\rm s}\times \Ww^{\rm u}\to \Wbb_T$,
see~(\ref{eq:wp}), is the restriction of the (linear) pre-gluing map
$\Pp_T\colon\Wbb_+\times\Wbb_-\to\Wbb_T$, see~(\ref{eq:Pp_T}).
In other words, the pre-gluing map $\wp_T$ maps $\Uu_+\times\Uu_-$
into the domain of the Newton-Picard map $\Nn_T$,
see~(\ref{eq:NP-map-app}), whenever $T\ge T_0$.

\begin{definition}[Gluing map]\label{def:gluing-map}
For $T\ge T_0$ the gluing map is the composition of smooth maps
\begin{equation}\label{eq:gamma_T}
   \gamma_T:=\Nn_T\circ\wp_T\colon \Uu_+\times\Uu_-
   \longrightarrow
   U_0(\delta)
   \longrightarrow
   \Wbb_T.
\end{equation}
The linearized gluing map is the composition
\[
   d\gamma_T|_{(w_+,w_-)}
   =d\Nn_T|_{\wp_T(w_+,w_-)}\circ d\wp_T|_{(w_+,w_-)}
   \colon T_{w_+}\Ww^{\rm s}\times T_{w_-}\Ww^{\rm u}\to\Wbb_T.
\]
\end{definition}

\begin{lemma}\label{le:d-gamma}
It holds $\gamma_T(0_+,0_-)=0_T$.
Furthermore, the differential of the gluing map $\gamma_T$ at 
$(0_+,0_-)$ is the infinitesimal gluing map $\Gamma_T$, in
symbols
$$
    d\gamma_T|_{(0_+,0_-)}=\Pi_T\circ d\wp_T|_{(0_+,0_-)} 
   \stackrel{(\ref{eq:Gamma_T-def})}{=:}
   \Gamma_T\colon\E^+\times\E^-\to\E_T .
$$
\end{lemma}

\begin{proof}
We get that $\gamma_T(0_+,0_-)
\stackrel{(\ref{eq:0_T})}{=}
\Nn_T(\wp_T (0_+,0_-))
\stackrel{(\ref{eq:NP-0_T})}{=}
\Nn_T(0_T)=0_T$.

By definition of $\gamma_T$ and the chain rule we get the first
equality
\begin{equation}\label{eq:bkjbkj3bn}
\begin{split}
   d\gamma_T|_{(0_+,0_-)}
   &=d\Nn_T|_{\underbrace{\wp_T(0_+,0_-)}_{\text{$=0_T$ by~(\ref{eq:0_T})}}}\circ
   d\wp_T|_{(0_+,0_-)}
   \\
   &=\left(\Id - Q_T D_T\right)\circ d\wp_T|_{(0_+,0_-)}
   \\
   &=
   \Pi_T\circ d\wp_T|_{(0_+,0_-)}
   \\
   &=:\Gamma_T
\end{split}
\end{equation}
and the second equality holds since $d\Nn_T(0_T)=\Id-Q_T D_T$,
by Corollary~\ref{cor:NP-map} with $x_0=0_T$ and $P=Q_T D_T$.
\\
Now $\Pi_T\colon \Wbb_T\to\Wbb_T$ is the projection onto $\E_T$ along
$\K_T$, by definition~(\ref{eq:Pi_T}), in symbols $\Pi_T=P_{\E_T,\K_T}$.
Thus, to see that
\begin{equation}\label{eq:1-P}
   {\color{gray} d\Nn_T(0_T)=\,}
   \Id-Q_TD_T=P_{\E_T,\K_T}{\color{gray} \,=:\Pi_T},
\end{equation}
it remains to show that the composition
\[
  Q_TD_T=P_{\K_T,\E_T}
\]
is the projection onto $\K_T$ along $\E_T$. This follows since 
\[
   Q_T D_T
   \stackrel{\rm (\ref{eq:Q_T})}{=}F_T^{-1} d\Ff_T|_{0_T}
   =\left(d\Ff_T|_{0_T}|_{\K_T}\right)^{-1} d\Ff_T|_{0_T}
   \colon W_T\to\V_T\to \K_T
\]
and $\E_T=\ker d\Ff_T|_{0_T}=\ker D_T$
and $Q_T=F_T^{-1}\colon\V_T\to \K_T$
where $F_T$ is the restriction of $D_T$ to $\K_T$, see~(\ref{eq:F_T}).
\end{proof}

\boldmath
%%%%%%%%%%%%%%%%%%%%%%%%%%%%%%%%%%%
%%%%%%% Subsection:  %%%%%%%%%%%%%%%%%%
%%%%%%%%%%%%%%%%%%%%%%%%%%%%%%%%%%%
\subsection{Diffeomorphism onto image}
\label{sec:diff-image}
\unboldmath
%\medskip
%\noindent
%\textbf{Path spaces and sections.}

\begin{theorem}\label{thm:gluing-diffeo}
There are open neighborhoods 
$\Oo_+\subset\Uu_+\subset\Ww^{\rm s}$ of $0_+$
and $\Oo_-\subset \Uu_-\subset\Ww^{\rm u}$ of $0_-$
   \todo[color=yellow!40]{\small key point: domain indep. of $T$!!}
such that for every $T\ge T_0$ the restricted gluing map
\[
   \gamma_T
   \colon
   \Oo_+\times\Oo_-
   \to\Mm_T
   ,\quad
   (w_+,w_-)\mapsto \Nn_T\circ\wp_T (w_+,w_-)
\]
is a diffeomorphism onto its image $\Oo_T$.
\end{theorem}

Note that the domain $\Oo_+\times\Oo_-$ of $\gamma_T$ does not depend
on $T$.

\begin{proof}
Given $T_0\ge 3$ as prior to~(\ref{eq:cond-1}), pick $T\ge T_0$.
The theorem is a consequence of the quantitative inverse function
Theorem~\ref{thm:InvFT-quant} (IFT), where
$F$ is given by a representative of $\gamma_T$ in local coordinate
charts; see~(\ref{eq:theta_+}).
   \todo[color=yellow!40]{\small more loc. triv. than loc. coords.!!}
In order to apply the quantitative IFT two conditions,
(\ref{eq:vgvh-1}) and (\ref{eq:vgvh-2}), are to be checked.

We verify~(\ref{eq:vgvh-1}):
Recall that the inverse of the infinitesimal gluing
map $\Gamma_T$ is uniformly bounded by a constant
$k=1/(1-e^{-12\sigma})>1$, see Lemma~\ref{le:Gamma_T-iso}. So, by
Lemma~\ref{le:d-gamma}, the inverse of $d\gamma_T|_{(0_+,0_-)}=dF|_0$
is bounded by $k$, uniformly~in~$T$.

We verify~(\ref{eq:vgvh-2}):
\todo[color=yellow!40]{\small $\eps=1/8kd$ $\mu=4k+1$}
To check this condition we choose
\begin{equation}\label{eq:Oo_+-}
\boxed{
   \Oo_\pm:=\Oo^\pm_{1/8kd}\cap\Uu_\pm^{4k+1},
}
   \qquad \eps=\tfrac{1}{8kd},\quad \mu=4k+1\ge 5,
\end{equation}
where $d$ is the ($T$-independent) bound of $\Pi_T$ from
Lemma~\ref{le:projection-onto-E}
and where the open origin neighborhoods $\Oo^\pm_\eps$ and
$\Uu_\pm^\mu$ in $\Ww^{\rm s/u}$
were defined in Lemma~\ref{le:Theta_T} and Lemma~\ref{le:Uu_pm},
respectively. Since $k\ge 1$, Lemma~\ref{le:Uu_pm}
tells that $\Uu_\pm^{4k+1}\subset \Uu_\pm^2=:\Uu_\pm$,
hence $\Oo_\pm\subset \Uu_\pm$.
Pick $(w_+,w_-)\in\Oo_+\times\Oo_-\subset \Uu_+\times\Uu_-$.

Recalling~(\ref{eq:theta_+})
we shall investigate the operator norm of the difference
\begin{equation*}
\begin{split}
   d\gamma_T|_{(w_+,w_-)}\circ (\theta_{w_+}^{-1}, \theta_{w_-}^{-1})
   -d\gamma_T|_{(0_+,0_-)}
   \colon \E^+\times \E^-\to\Wbb_T.
\end{split}
\end{equation*}
Abbreviate $x_1:=\wp_T(w_+,w_-)$. By definition
of $\gamma_T$ and of $\Theta_T(w_+,w_-)$ we get
\begin{equation}\label{eq:gbhjbjh-4}
\begin{split}
   &d\gamma_T|_{(w_+,w_-)}\circ (\theta_{w_+}^{-1}, \theta_{w_-}^{-1})
   -d\gamma_T|_{(0_+,0_-)}
\\
%   &=d\Nn_T|_{\wp_T(w_+,w_-)}\circ d\wp_T|_{(w_+,w_-)}
%   \circ (\theta_{w_+}^{-1}, \theta_{w_-}^{-1})
%   -d\Nn_T|_{0_T}\circ d\wp_T|_{(0_+,0_-)}
%\\
   &\stackrel{\text{(\ref{eq:gamma_T})}}{=}
   d\Nn_T|_{x_1}\circ \Pp_T
   \circ (\theta_{w_+}^{-1}, \theta_{w_-}^{-1})
   -d\Nn_T|_{0_T}\circ \Pp_T
%   \\
%   &\quad
   {\color{gray}
   \; -d\Nn_T|_{x_1}\circ \Pp_T
   +d\Nn_T|_{x_1}\circ \Pp_T
   }
\\
   &\stackrel{\text{(\ref{eq:Theta_T})}}{=}
   -d\Nn_T|_{x_1}\circ \Theta_T(w_+,w_-)
   +\left(d\Nn_T|_{x_1}-d\Nn_T|_{0_T}\right)\circ\Pp_T .
\end{split}
\end{equation}
By Remark~\ref{rem:A} for $Q:=Q_T$ and $P:=Q_TD_T$
and since the projection $\Id-P=\Pi_T$, see~(\ref{eq:1-P}),
has a ($T$-independent) bound $d$ by
Lemma~\ref{le:projection-onto-E} we get that
\begin{equation}\label{eq:gbhjbjh-5}
   \norm{d\Nn_T|_{x_1}}
   \stackrel{\text{(\ref{eq:dNn})}}{\le}
   \norm{\left(\Id+Q_T \,df(x_1)-P\right)^{-1}}\cdot\norm{\Id-P}
  \stackrel{\text{(\ref{eq:inverse-2})}}{\le} 2 d .
\end{equation}
Since $(w_+,w_-)\in\Oo^+_{1/8kd}\times\Oo^-_{1/8kd}$,
by Lemma~\ref{le:Theta_T}, we have that
\begin{equation}\label{eq:gbhjbjh-6}
   \norm{\Theta_T(w_+,w_-)}\le\tfrac{1}{8kd} .
\end{equation}
Furthermore, abbreviating $x_0:=0_T$, then $d\Nn_T|_{0_T}=\Id-P$
by~(\ref{eq:dNn}) with $x_1$ replaced by $x_0$ and using that
$Q_T \,df(x_0)=Q_TD_T=P$. Thus we get that
\begin{equation}\label{eq:gbhjbjh-1}
   \left(d\Nn_T|_{x_1}-d\Nn_T|_{0_T}\right) \circ\Pp_T
   \stackrel{\text{(\ref{eq:dNn})}}{=}
   \left(\bigl(\Id+Q_T \,df|_{x_1}-P\bigr)^{-1}-\Id\right)\left(\Id-P\right)
   \circ\Pp_T.
\end{equation}
Observe that
\begin{equation}\label{eq:gbhjbjh-2}
   \left(\Id-P\right)\circ\Pp_T
   \stackrel{\text{(\ref{eq:dwp=P_T})}}{=}
   \left(\Id-Q_TD_T\right) \circ d\wp_T|_{(0_+,0_-)}
   \stackrel{\text{(\ref{eq:bkjbkj3bn})}}{=}
   \Gamma_T .
\end{equation}
Since $(w_+,w_-)\in\Uu_+^{4k+1}\times\Uu_-^{4k+1}$,
it follows from Lemma~\ref{le:Uu_pm} that
$\norm{x_1}_{1,2}
=\norm{\Pp_T(w_+,w_-)}_{1,2}
<\min\{\frac{\delta}{8},\delta_{4k+1}\}$.
Hence $x_1$ lies in $\Uu^{4k+1}_T$
by Corollary~\ref{cor:delta}.
Therefore, by Lemma~\ref{le:bd-diff-lin},
it follows that $\norm{d\Ff_T(x_1)-D_T}\le\frac{1}{(4k+1) c}$.
\\
In view of Remark~\ref{rem:NP-map-mu} with $U_0$ given by $\Uu_0$
we obtain
\begin{equation}\label{eq:gbhjbjh-3}
   \norm{\bigl(\Id+Q_T \,df|_{x_1}-P\bigr)^{-1}-\Id}
   \le\tfrac{1}{4k}.
\end{equation}
By Lemma~\ref{le:Gamma_T-iso} we have $\norm{\Gamma_T}\le 1$.
Combining this fact with~(\ref{eq:gbhjbjh-1}),~(\ref{eq:gbhjbjh-2}),
and~(\ref{eq:gbhjbjh-3}) we conclude
\begin{equation}\label{eq:gbhjbjh-7}
   \norm{\left(d\Nn_T|_{x_1}-d\Nn_T|_{0_T}\right) \circ\Pp_T}
   \le\norm{\bigl(\Id+Q_T \,df|_{x_1}-P\bigr)^{-1}-\Id}
   \cdot\norm{\Gamma_T}
   \le\tfrac{1}{4k} .
\end{equation}
By~(\ref{eq:gbhjbjh-4}),~(\ref{eq:gbhjbjh-5}),~(\ref{eq:gbhjbjh-6}),
and~(\ref{eq:gbhjbjh-7}) we conclude
\[
   \norm{d\gamma_T|_{(w_+,w_-)}\circ (\theta_{w_+}^{-1}, \theta_{w_-}^{-1})
   -d\gamma_T|_{(0_+,0_-)}}
   \le \tfrac{2d}{8kd}+\tfrac{1}{4k}=\tfrac{1}{2k} .
\]
This verifies~(\ref{eq:vgvh-2}).
Corollary~\ref{cor:InvFT-quant} concludes the proof of
Theorem~\ref{thm:gluing-diffeo}.
\end{proof}

\boldmath
%%%%%%%%%%%%%%%%%%%%%%%%%%%%%%%%%%%
%%%%%%% Subsection:  %%%%%%%%%%%%%%%%%%
%%%%%%%%%%%%%%%%%%%%%%%%%%%%%%%%%%%
\subsection{Evaluation maps and convergence in $C^m$}
\label{sec:Cm-convergence}
\unboldmath
\begin{definition}
Consider the evaluation maps defined by
\[
   \ev\colon\Wbb^+\times\Wbb^-\to\R^n\times\R^n
   ,\quad
   (w_+,w_-)\mapsto \left(w_+(0),w_-(0)\right)
\]
and, for $T>0$, by
\[
   \ev_T\colon\Wbb_T\to\R^n\times\R^n
   ,\quad
   w\mapsto \left(w(-T),w(T)\right) .
\]
\end{definition}

Observe that both evaluation maps are linear.
Furthermore, for $T\ge3$ we have
$\ev_T\circ \Pp_T(w_+,w_-)=(w_+(0),w_-(0)=\ev(w_+,w_-)$.\footnote{
By definition~(\ref{eq:Pp_T}) of $w_T$ and the cut-off function
$\beta$, we get the identities
\begin{equation*}
\begin{split}
   w_T(-T) -w_+(0)
   &=
   \left(1-\beta(-T+2)\right) w_+(0)
   +\beta(-T-2)\, w_-(-2T)
   -w_+(0)
   \\
   &=
   -\beta(-T+2) w_+(0)+\beta(-T-2)\, w_-(-2T)
   \\
   &=0\quad\text{for $T\ge 3$}
\\
   w_T(T) -w_-(0)
   &=
   \left(1-\beta(T+2)\right) w_+(2T)
   +\beta(T-2)\, w_-(0)
   -w_-(0)
   \\
   &=
   \left(1-\beta(T+2)\right) w_+(2T)
   -\left(1-\beta(T-2)\right) w_-(0)
   \\
   &=0\quad\text{for $T\ge 3$} .
  \end{split}
\end{equation*}
Thus, by definition of the evaluation maps, for $T\ge3$ we get that
\[
   \ev_T(w_T)-\ev(w_+,w_-)
   =\left(w_T(-T)-w_+(0),w_T(T)-w_-(0)\right)
   =(0,0) .
\]
}
So there is the identity
\begin{equation}\label{eq:ev-ev_T}
   \ev_T\circ \Pp_T=\ev \colon
   \Wbb^+\times\Wbb^-\to\R^n\times\R^n
\end{equation}
whenever $T\ge3$.
Therefore, for tangent maps, we get the identity
\begin{equation}\label{eq:Tev-Tev_T}
   T^m\ev_T\circ T^m\Pp_T=T^m\ev \colon
   T^m\Wbb^+\times T^m\Wbb^-\to T^m\R^n\times T^m\R^n
\end{equation}
whenever $m\in\N$ and $T\ge3$.

\begin{lemma}\label{le:ev_T}
$\norm{\ev_T}\le 2\sqrt{2}$.
\end{lemma}

\begin{proof}
$\abs{\ev_T w}_{\R^n\times\R^n}^2
=\abs{w(-T)}^2+\abs{w(T)}^2
\le 2\norm{w}_\infty^2
\le 8\norm{w}_{\Wbb_T}^2$
by~(\ref{eq:sob-emb-T}).
\end{proof}

To motivate Theorem~\ref{thm:loc-gluing-C^m} below
we first check the infinitesimal version in case $m=0$,
see~(\ref{eq:comm-diag-C0}).
The linearized evaluation maps are given by
\begin{equation*}
\begin{split}
   d\ev|_{(0_+,0_-)}=\ev\colon\E^+\times\E^-&\to\R^n\times\R^n
   \\
   (\xi,\eta)&\mapsto\left(\xi(0),\eta(0)\right)
\end{split}
\end{equation*}
and
\begin{equation*}
\begin{split}
   d\ev_T|_{0_T}=\ev_T\colon\E_T&\to\R^n\times\R^n
   \\
   \zeta&\mapsto\left(\zeta(-T),\zeta(T)\right) .
\end{split}
\end{equation*}
By Lemma~\ref{le:d-gamma} we get that
\[
   d\left(\ev_T\circ\gamma_T\right)|_{(0_+,0_-)}
   =\ev_T\circ d\gamma_T|_{(0_+,0_-)}
   =\ev_T\circ\Gamma_T
   \colon\E^+\times\E^-\to\E_T.
\]
Thus, for $(\xi,\eta)\in\E^+\times\E^-$
and by~(\ref{eq:Gamma_T}), we obtain
\begin{equation*}
\begin{split}
   d\left(\ev_T\circ\gamma_T\right)|_{(0_+,0_-)}\left(\xi,\eta\right)
   &=
   \left(
      \Gamma_T(\xi,\eta)(-T), \Gamma_T(\xi,\eta)(T)
   \right)
\\
   &=
   \left(
      \xi(0)+e^{-2TA_-}\eta(0),e^{-2TA_+}\xi(0)+\eta(0)
   \right)
\\
   &\stackrel{T\to\infty}{\to}
   \left(
      \xi(0),\eta(0)
   \right)
   =d\ev|_{(0_+,0_-)} (\xi,\eta) .
\end{split}
\end{equation*}
This confirms the infinitesimal version of
Theorem~\ref{thm:loc-gluing-C^m} in case $m=0$.

\begin{theorem}[Local gluing -- $C^m$]\label{thm:loc-gluing-C^m}
Let $m\in\N_0$.
Consider the gluing map $\gamma_T\colon \Uu_+\times\Uu_-\to\Wbb_T$
from~(\ref{eq:gamma_T}). In the limit $T\to\infty$ the tangent map diagram
\begin{equation*}
\begin{tikzcd} [column sep=tiny] %[row sep=tiny]
T^m\Uu_+\times T^m\Uu_-
\arrow[rrrr, "T^m\ev"]
\arrow[drr, "{T^m\gamma_T=T^m\Nn_T\circ T^m\Pp_T}"']
  &&&&T^m\R^n\times T^m\R^n
\\
  &&T^m\Mm_T
  \arrow[urr, "T^m\ev_T"']
\end{tikzcd}
\end{equation*}
commutes. More precisely, it holds that
\[
   \lim_{T\to\infty}T^m\ev_T\circ T^m\gamma_T=T^m\ev
\]
in $C^0(T^m\Uu_+\times T^m\Uu_-,T^m\R^n\times T^m\R^n)$.
\end{theorem}

\begin{proof}
Let $W_\pm\in T^m\Uu_\pm$. For $W_T:=T^m\Pp_T(W_+,W_-)\in T^m\Wbb_T$ we obtain
\begin{equation*}
\begin{split}
   \overbrace{(V_-,V_+)}^{{\color{gray}T^m\R^n\times T^m\R^n}}
   :&=\left(T^m\ev_T\circ T^m\gamma_T\right)(W_+,W_-)-T^m\ev(W_+,W_-)\\
   &=T^m\ev_T\circ T^m\Nn_T(W_T) -T^m\ev(W_+,W_-)\\
   &=T^m\ev_T\bigl(T^m\Nn_T(W_T)-W_T\bigr)
   +\underbrace{T^m\ev_T(W_T) -T^m\ev(W_+,W_-)}_{\text{$=0$ by~(\ref{eq:Tev-Tev_T})}}
\end{split}
\end{equation*}
where equality one is definition~(\ref{eq:gamma_T}) of $\gamma_T$ and
equality two by adding zero.

In view of Lemma~\ref{le:le-1} and Lemma~\ref{le:le-2} below,
by Theorem~\ref{thm:tangent-map-knackig}
there exists a constant $c$, independent of $T$, such that
\begin{equation*}
\begin{split}
   &\Norm{(T^m\Nn_T-\id) (W_T)}_{T^m\Wbb_T}
\\
   &\stackrel{\text{(\ref{eq:wonder-prop-diff-T})}}{\le}
   c\Norm{T^m\Ff_T(W_T)}_{T^m\V_T}
   \Bigl(1+\Norm{T^m\Ff_T(W_T)}_{T^m\V_T}\Bigr)
\\
   &\stackrel{(\ref{eq:T^mFf-decay})}{\le}
   cC e^{-\eps T}
   \Bigl(1+C e^{-\eps T}\Bigr)
\end{split}
\end{equation*}
where the second inequality is by exponential decay~(\ref{eq:T^mFf-decay})
with constant $C=C(K_+,K_-)$ and $K_\pm$ depending on $m$.
In particular, there exists a constant $T_m=T_m(K_+,K_-)>0$ such that
if $T>T_m$, then
$$
   1+C e^{-\eps T}
   \le 2.
$$
Therefore, by the uniform--in--$T$ Sobolev estimate~(\ref{eq:sob-emb-T})
we get
\begin{equation}\label{eq:mnn,m}
\begin{split}
   \Norm{(T^m\Nn_T-\id) (W_T)}_{L^\infty_{[-T,T]}}
   &\le
   2\Norm{(T^m\Nn_T-\id) (W_T)}_{T^m\Wbb_T}\\
   &\le
   4cC e^{-\eps T}
  \end{split}
\end{equation}
for every $T>T_m$.
Putting things together, using that $\ev_T\circ \Pp_T=\ev$
by~(\ref{eq:ev-ev_T}), we obtain exponential decay
\begin{equation*}
\begin{split}
   &\Abs{
      \left(T^m\ev_T\circ T^m\gamma_T\right)(W_+,W_-)-T^m\ev(W_+,W_-)
   }_{T^m\R^n\times T^m\R^n}^2
   \\
   &=\Abs{V_-}_{T^m\R^n}^2+\Abs{V_+}_{T^m\R^n}^2
   \\
   &\le 2(4cC)^2 e^{-2\eps T}
  \end{split}
\end{equation*}
whenever $T>T_m$.
But in finite dimensions pointwise convergence implies
convergence of the operators.\footnote{
  In finite dimension, given a sequence of matrizes, then weak (and
  strong) convergence means that each matrix entry converges.
  In particular, the two notions of convergence are equivalent.
  }
By the uniformity of exponential decay in
Proposition~\ref{prop:approx-zero}
we have uniform convergence in
$C^0(T^m\Uu_+\times T^m\Uu_-,T^m\R^n\times T^m\R^n)$.
\end{proof}

In the following, by iterated identification of the space with the zero
section of its tangent space, we can interpret $0_T\in\Wbb_T$ as an element
of $T^m \Wbb_T$.
Since $\Nn_T(0_T)=0_T$, see~(\ref{eq:NP-0_T}), we have $T^m\Nn_T(0_T)=0_T$
and therefore
$$
   dT^m\Nn_T(0_T)\colon T_{0_T}T^m\Wbb_T\to T_{0_T}T^m\Wbb_T.
$$
Since $\Wbb_T$ itself is a vector space, we have a canonical isomorphism of
$T_{0_T}T^m\Wbb_T$ with $(\Wbb_T)^{\times 2^m}$.

\begin{lemma}\label{le:le-1}
Given $m\in\N_0$ and $T\ge 1$, let $d$ be the $T$-independent constant
provided by Lemma~\ref{le:projection-onto-E}, then
$\norm{dT^m\Nn_T(0_T)}\le d^{2^m}$.
In particular, the norm is uniformly bounded independent of $T$.
\end{lemma}

\begin{proof}
With respect to the splitting $T_{0_T}T^m\Wbb_T=(\Wbb_T)^{\times 2^m}$
we have the block decomposition
$
   dT^m\Nn_T(0_T)
   =\diag\left(d\Nn_T(0_T),\dots, d\Nn_T(0_T)\right) 
$.
By formula~(\ref{eq:NP-0_T}) we have $d\Nn_T(0_T)=\Pi_T$.
Hence the estimate follows from Lemma~\ref{le:projection-onto-E}.
\end{proof}

\begin{lemma}\label{le:le-2}
Given $W_\pm\in T^m\Uu_\pm$,
the norm of $W_T:=T^m\Pp_T(W_+,W_-)$
is uniformly bounded in terms of the norms of $W_+$
and $W_-$, independent of $T$.
\end{lemma}

\begin{proof}
By~(\ref{eq:zeta_TT2}) the same estimate
as in the proof of Lemma~\ref{le:Uu_pm} for $w_T$ yields
\begin{equation*}
   \norm{W_T}_{T^m \Wbb_T}^2
   \le
   2(1+\norm{\beta^\prime}_\infty^2) \left(\norm{W_+}_{T^m \Wbb_+}^2
   +\norm{W_-}_{T^m \Wbb_-}^2 \right) .
\end{equation*}
\end{proof}

\appendix
%%%%%%%%%%%%%%%%%%%%%%%%%%%%%%%%%%%
%%%%%%%%%%%%%%%%%%%%%%%%%%%%%%%%%%%
%%%%%%% Appendix:  %%%%%%%%%%%%%%%%%%%
%%%%%%%%%%%%%%%%%%%%%%%%%%%%%%%%%%%
%%%%%%%%%%%%%%%%%%%%%%%%%%%%%%%%%%%

\boldmath
%%%%%%%%%%%%%%%%%%%%%%%%%%%%%%%%%%%
%%%%%%%%%%%%%%%%%%%%%%%%%%%%%%%%%%%
%%%%%%% Section:  %%%%%%%%%%%%%%%%%%%%%
%%%%%%%%%%%%%%%%%%%%%%%%%%%%%%%%%%%
%%%%%%%%%%%%%%%%%%%%%%%%%%%%%%%%%%%
\section{Quantitative inverse function theorem}
\label{sec:quant-InvFT}
\unboldmath

Let $F\colon X\to Y$ be a map between Banach spaces.
Suppose that at a point $x\in X$ the derivative $dF|_x\colon X\to Y$
exists. If this bounded linear map is bijective then its inverse
${dF|_x}^{-1}$ is not only linear but, by the open mapping theorem,
also bounded.

The quantitative version of the inverse function theorem
({\bf IFT}) follows of the proof of the usual IFT explained
in~\cite[App.\,A.3]{mcduff:2004a},
although McDuff-Salamon never state explicitly the quantitative version.
Therefore, for the reader's convenience, we state the 
quantitative version of the IFT and explain how it follows from the
\todo[color=yellow!40]{\small
  $\bullet$ Christ 1985 Ann. Math. \\
  $\bullet$ Liverani \tiny pdf
}
arguments in~\cite[App.\,A.3]{mcduff:2004a}.

We denote by $B_r(x;X)$ the open ball of radius $r$ centered at $x$ in
the Banach space $X$. We often abbreviate $B_r(x):=B_r(x;X)$ and
$B_r:=B_r(0;X)$.

\begin{theorem}[Quantitative inverse function theorem]\label{thm:InvFT-quant}
Let $k,\delta>0$ be constants.
Let $F\colon X\to Y$ be a map between
Banach spaces, continuously differentiable on the open radius-$\delta$
ball $B_\delta$ about $0\in X$, such that $dF|_0$ is bijective~and
\begin{equation}\label{eq:vgvh-1}
   \norm{{dF|_0}^{-1}}\le k
\end{equation}
and
\begin{equation}\label{eq:vgvh-2}
   \norm{dF|_x-dF|_0}\le\tfrac{1}{2k}
   ,\quad
   \forall x\in B_\delta.
\end{equation}
In this case the following is true. The restriction of $F$ to
$B_\delta$ is injective,
the image $F(B_\delta)$ is open and contains the ball $B_{\delta/2k}$,
the inverse $F^{-1}\colon F(B_\delta)\to B_\delta$ is of class $C^1$, and
\begin{equation}\label{eq:vgvh}
   d(F^{-1})|_y
   =(dF|_{F^{-1}(y)})^{-1}
\end{equation}
for every $y\in F(B_\delta)$.
\end{theorem}

\begin{corollary}\label{cor:InvFT-quant}
If in Theorem~\ref{thm:InvFT-quant} in addition $F\colon B_\delta\to Y$
is of class $C^\ell$ for some $\ell\in\N$, then so is $F^{-1}$.
In particular, in case $\ell=\infty$ the restriction
$F|\colon B_\delta\to F(B_\delta)$ is a diffeomorphism onto its image.
\end{corollary}

\begin{proof}
Induction, using the chain and Leibniz rules, together with~(\ref{eq:vgvh}).
\end{proof}

The proof of Theorem~\ref{thm:InvFT-quant} is based on the following
lemma.

\begin{lemma}[{\hspace{-.001cm}\cite[Le.\,A.3.2]{mcduff:2004a}}]\label{le:A.3.2}
Let $\gamma<1$ and $R$ be positive real numbers.
Let $X$ be a Banach space, $x_0\in X$, and $\psi\colon B_R(x_0)\to X$
be a continuously differentiable map such that
\[
   \norm{\1-d\psi(x)}\le\gamma
\]
for every $x\in B_R(x_0)$.
Then the following holds.
The map $\psi$ is injective and $\psi$ maps $B_R(x_0)$
into an open set in $X$ such that
\begin{equation}\label{eq:psi-incl}
   B_{R(1-\gamma)}(\psi(x_0))
   \subset\psi(B_R(x_0))
   \subset B_{R(1+\gamma)}(x_0) .
\end{equation}
The inverse $\psi^{-1}\colon\psi(B_R(x_0))\to B_R(x_0)$
is continuously differentiable and
\begin{equation}\label{eq:psi-inverse}
   d(\psi^{-1})|_y
   =(d\psi|_{\psi^{-1}(y)})^{-1} .
\end{equation}
\end{lemma}

\begin{proof}[Proof of Theorem~\ref{thm:InvFT-quant}]
This is basically the proof of the usual IFT given
in~\cite[App.\,A.3]{mcduff:2004a}.
We assume without loss of generality $x_0=0$ and $F(0)=0$.
We consider the map $\psi\colon B_\delta\to X$ defined by
\[
   \psi(x):=D^{-1} F(x)
\]
where $D:=dF|_0$. For $x\in B_\delta$ we estimate
\[
   \norm{\1-d\psi|_x}
   =\norm{D^{-1}(D-dF|_x)}
   \le \norm{D^{-1}}\cdot\norm{D-dF|_x}
   \le k\cdot \tfrac{1}{2k}=\tfrac12.
\]
It follows from Lemma~\ref{le:A.3.2}
with $R=\delta$ and $\gamma=\frac12$
that $\psi$ has a continuously differentiable inverse on $B_\delta(0;X)$
and that $\psi(B_\delta(0;X))$ is an open set containing $B_{\delta/2}(0;X)$.
Since $F=D\circ\psi$ and $\norm{{dF|_0}^{-1}}\le k$ we get, respectively,
inclusion one and two
\[
   F(B_\delta(0;X))=D\circ\psi(B_\delta(0;X))
   \supset D B_{\delta/2}(0;X)
   \supset B_{\delta/2k}(0;Y) .
\]
The inverse of $F=D\psi$ is given by
\[
   F^{-1}(y)=\psi^{-1}(D^{-1} y) .
\]
The inverse $F^{-1}$ is continuously differentiable, since $\psi^{-1}$
is, and the formula $d(F^{-1})|_y=(dF|_{F^{-1}(y)})^{-1}$
follows by differentiating $F\circ F^{-1}=\id_Y$.
\end{proof}

\boldmath
%%%%%%%%%%%%%%%%%%%%%%%%%%%%%%%%%%%
%%%%%%%%%%%%%%%%%%%%%%%%%%%%%%%%%%%
%%%%%%% Section:  %%%%%%%%%%%%%%%%%%%%%
%%%%%%%%%%%%%%%%%%%%%%%%%%%%%%%%%%%
%%%%%%%%%%%%%%%%%%%%%%%%%%%%%%%%%%%
\section[Newton-Picard map without quadratic estimates]{Newton-Picard without quadratic estimates}
\label{sec:Newton-Picard-map}
\unboldmath

The Newton-Picard map is usually defined via the Newton-Picard
iteration method.
To show that Newton-Picard iteration is a contraction
one needs to calculate troublesome quadratic estimates.
Based on~\cite[App.\,A.3]{mcduff:2004a} we explain how the
Newton-Picard map can as well be defined even if there are no
quadratic estimates available.
The Newton-Picard map $\Nn$ obtained in this way is still continuously
differentiable. This fact is not mentioned
in~\cite[App.\,A.3]{mcduff:2004a} and therefore we prove this
fact in the present article; see Appendix~\ref{sec:NP-map}.

For induction arguments, e.g. the one in Section~\ref{sec:Cm-convergence},
tangent maps are much more suitable than differentials.
Therefore we estimate, in Appendix~\ref{sec:NP-tangent},
the tangent map difference $T\Nn-\Id$.

\smallskip
\noindent
\textbf{Notation.}
Throughout Appendix~\ref{sec:Newton-Picard-map}
the letter $f$ denotes a map between Banach spaces, not
a Morse function as in the principal part of this article.

\boldmath
%%%%%%%%%%%%%%%%%%%%%%%%%%%%%%%%%%%
%%%%%%% Subsection:  %%%%%%%%%%%%%%%%%%%
%%%%%%%%%%%%%%%%%%%%%%%%%%%%%%%%%%%
\subsection{Newton-Picard map}
\label{sec:NP-map}
\unboldmath

The definition of the Newton-Picard map
requires the following proposition
from~\cite[App.\,A.3]{mcduff:2004a}.
The proof can actually be interpreted
in terms of the Newton-Picard iteration as is explained
in~\cite[Rmk.\,A.3.5]{mcduff:2004a} in case $x_0=x_1$.

\begin{proposition}[\hspace{-.1pt}{\cite[Prop.\,A.3.4]{mcduff:2004a}}]
\label{prop:wonder-prop-NEW}
Let $X$ and $Y$ be Banach spaces, $U\subset X$ be an open
set, and $f\colon U\to Y$ be a continuously differentiable map.
Let $x_0\in U$ be a suitable initial point in the sense
that $D:=df(x_0)\colon X\to Y$ is surjective
and has a (bounded linear) right inverse $Q\colon Y\to X$.
Choose positive constants $\delta$ and $c$ such that $\norm{Q}\le c$,
the open radius-$\delta$ ball about $x_0$ satisfies $B_\delta(x_0;X)\subset U$, and
\begin{equation}\label{eq:A.3.4-NEW}
   \norm{x-x_0}<\delta\qquad\Rightarrow\qquad
   \norm{df(x)-D}\le\frac{1}{2c} .
\end{equation}
Suppose that $x_1\in X$ is an approximate zero of $f$ near $x_0$ in
the sense that
\begin{equation}\label{eq:A.3.5-NEW}
   \norm{x_1-x_0}<\frac{\delta}{8} ,\qquad
   \norm{f(x_1)}<\frac{\delta}{4c} .
\end{equation}
Then there exists a unique zero $x\in X$ near the initial point
$x_0$ such that
\begin{equation}\label{eq:wonder-prop-NEW}
   f(x)=0,\qquad
   x-x_1\in\im Q,\qquad
   \norm{x-x_0}<\delta .
\end{equation}
Moreover, the distance between the detected zero $x$ and the
chosen approximate zero $x_1$ is controlled by $f(x_1)$,
more precisely
\begin{equation}\label{eq:wonder-prop-main}
   \norm{x-x_1}\le 2c\norm{f(x_1)} .
\end{equation}
\end{proposition}

\begin{definition}%\label{def:NP-map-2}
Based on the proposition we define the \textbf{Newton-Picard map}
$\Nn$ as follows. Define an open subset of $U$ by
%\[
%   U_0:=
%   \Bigl(B_{\delta/8}(x_0;X)\cap\norm{f}^{-1}[0,\tfrac{\delta}{4c})\Bigr)
%   \subset \norm{df(\cdot)-D}^{-1}[0,\tfrac{1}{2c})
%\]
\begin{equation}\label{eq:U_0(delta)}
\begin{split}
   U_0=U_0(\delta):
   &=
   \Bigl(B_{\frac{\delta}{8}}(x_0;X)
   \cap\bigl\{\Norm{f(\cdot)}_Y <\tfrac{\delta}{4c}\bigr\}\Bigr)
\\
   &\stackrel{\text{(\ref{eq:A.3.4-NEW})}}{\subset}
   \bigl\{\Norm{df(\cdot)-D}_{\Ll(X,Y)}<\tfrac{1}{2c}\bigr\}
\end{split}
\end{equation}
and a map
\begin{equation}\label{eq:NP-map}
  \Nn=\Nn^f_{x_0,Q}\colon X\supset U\supset U_0\to X,\quad
   x_1\mapsto x
\end{equation}
which maps a point $x_1$, thought of as an approximate zero of $f$,
to the unique zero $x$ in the $\delta$-ball about $x_0$ whose
difference $x-x_1$ lies in the image of the right inverse $Q$.
Note that the domain $U_0$ of definition of the Newton-Picard map
$\Nn$ depends on the choice of equivalent norms on $X$ and $Y$.
\end{definition}

\begin{remark}\label{rem:N-id}
The uniqueness statement implies that
$\Nn|_{f^{-1}(0)\CAP U_0}=\id$.
\end{remark}

%\newpage%.\newpage
\begin{theorem}\label{thm:NP-map}
The Newton-Picard map $\Nn$ is continuously differentiable.
\end{theorem}

\begin{proof}
We first recall how the zero $x$ in
Proposition~\ref{prop:wonder-prop-NEW}
is found from a given approximate zero $x_1\in U_0$.
One considers the map defined by
\begin{equation}\label{eq:psi}
   \psi_{x_1}\colon X\supset U\supset B_\delta(x_0)\to X
   ,\quad
   x\mapsto
   x+Q\bigl(f(x)-D\left(x-x_1\right)\bigr) .
\end{equation}
The map $\psi_{x_1}$ is continuously differentiable,
because $f$ is, and by~(\ref{eq:A.3.4-NEW}) the differential at any
$x\in B_\delta(x_0)$ satisfies
\begin{equation}\label{eq:dpsi-Id}
   \norm{\Id-d \psi_{x_1}(x)}
   \le c\norm{df(x)-D}
   \le\frac12 .
\end{equation}
Moreover, according to~\cite[Proof of Prop.\,A.3.4]{mcduff:2004a}
the map $\psi_{x_1}\colon B_\delta(x_0)\to X$
is injective and the Newton-Picard map is given by
$\Nn(x_1):=\psi_{x_1}^{-1}(x_1)$.\footnote{
  In~\cite[Proof of Prop.\,A.3.4]{mcduff:2004a} it is shown that
  $x_1\in B_{\delta/2}(\psi_{x_1}(x_0))\subset\psi_{x_1}(B_\delta(x_0))$.
  }

To show that $\Nn$ is differentiable we consider the map
$$
   \Psi\colon B_\delta(x_0)\times U_0\to X\times U_0
   ,\quad
   (x,x_1)\mapsto \left(\psi_{x_1}(x),x_1\right).
$$
The differential of $\Psi$ at a point
$(x,x_1)\in B_\delta(x_0)\times U_0$ is the linear map
$$
   d\Psi \bigr|_{(x,x_1)}
   =
   \begin{pmatrix}d\psi_{x_1}\bigr|_{x}
      &(\p_{x_1}\psi_{x_1}) \bigr|_{x} \\ 0&\Id
   \end{pmatrix}
   \stackrel{\text{(\ref{eq:psi})}}{=}
   \begin{pmatrix}d\psi_{x_1}\bigr|_{x}
      &P \\ 0&\Id
   \end{pmatrix}
   \colon X\times X\to X\times X
$$
where
\[
   P:=QD\colon X\to X
\]
is a projection.\footnote{
  Since $DQ=\Id$, the map
\begin{equation}\label{eq:P-NEW}
   P:=QD
   ,\quad
   P^2=P
   ,\quad
   \im P=\im Q
   ,\quad
   \ker P=\ker D ,
\end{equation}
is the projection $P=P_{\im Q,\ker D}$ onto the image of $Q$ along the
kernel of $D$.
Indeed $P^2=QDQD=QD=P$.
It holds $\im P=\im Q$: '$\subset$' true by definition of $P$.
'$\supset$' If $\xi=Q\eta$, then
$P\xi=PQ\eta=QDQ\eta=Q\eta=\xi$.
It holds $\ker P=\ker D$: '$\subset$'
If $P\xi=0$, then $D\xi=DQD\xi=DP\xi=0$.
'$\supset$' true by definition of $P$.
  }
Since the bound in~(\ref{eq:dpsi-Id}) is $<1$,
the linear map $d\psi_{x_1}(x)\colon X\to X$ is invertible,
therefore so is $d\Psi \bigr|_{(x,x_1)}$ with inverse
\[
   (d\Psi \bigr|_{(x,x_1)})^{-1}
   =\begin{pmatrix}
      {\color{brown}\left(d\psi_{x_1}\bigr|_{x}\right)^{-1}}
      &{\color{brown}-\left(d\psi_{x_1}\bigr|_{x}\right)^{-1} P}
      \\
      0&\Id
      \end{pmatrix} .
\]
Therefore, by the inverse function 
theorem~\cite[Thm.\,A.3.1]{mcduff:2004a},
the map $\Psi$ is injective in a neighborhood of
$(x,x_1)\in B_\delta(x_0)\times U_0$
with continuously differentiable inverse.
\\
It follows that the Newton-Picard map $\Nn$
is differentiable, too, with differential
$d\Nn(x_1)\colon X\to X$ given by the formula
\begin{equation}\label{eq:dNn}
   d\Nn(x_1)
   ={\color{brown}
   \bigl(d\psi_{x_1}\bigr|_{x_1}\bigr)^{-1} \left(\Id-P\right)}
   \stackrel{\text{(\ref{eq:psi})}}{=}
   \Bigl(\Id+Q \,df(x_1)-P\Bigr)^{-1} \left(\Id-P\right) .
\end{equation}
The differential $d\Nn(x_1)$ depends continuously on $x_1$ since
$df(x_1)$ does.
\end{proof}

\begin{remark}[The inverse in~(\ref{eq:dNn})]
\label{rem:A}
Let $A:=Q\left(D-df|_{x_1}\right)=P-Q\,df|_{x_1}$,
then $\norm{A}\le\frac12$ since $\norm{Q}\le c$ and by~(\ref{eq:A.3.4-NEW}).
The inverse for $\Id-A$ is given by
\begin{equation}\label{eq:Neumann}
   \sum_{n=0}^\infty A^n
   =\left(\Id-A\right)^{-1}
   =\Bigl(\Id+Q \,df|_{x_1}-P\Bigr)^{-1}
\end{equation}
where the sum, called \textbf{Neumann series}, converges in $\Ll(X)$
whenever $A$ has operator norm less than $1$;
for details see e.g.~\cite[(VI.2) p.\,191]{reed:1980a} with $\lambda=1$.
Thus there is the estimate
\begin{equation}\label{eq:inverse-2}
   \norm{\left(\Id+Q \,df|_{x_1}-P\right)^{-1}}
   \le \sum_{n=0}^\infty \norm{A}^n
   \le \sum_{n=0}^\infty \frac{1}{2^n}
   =\frac{1}{1-\frac12}=2
\end{equation}
and therefore
\[
   \norm{\left(\Id+Q \,df|_{x_1}-P\right)^{-1}-\Id}
   =\Norm{\sum_{n=1}^\infty A^n}
   \le \sum_{n=1}^\infty \norm{A}^n
%   \le \sum_{n=1}^\infty \frac{1}{2^n}
   =2-\frac{1}{2^0}=1 .
\]
\end{remark}

\begin{remark}[Smoothness]\label{rem:NP-map}
If $f$ in Proposition~\ref{prop:wonder-prop-NEW} is assumed to be not
only continuously differentiable, but smooth, then it follows
from~(\ref{eq:dNn}) that the Newton-Picard map $\Nn$ is smooth as well.
\end{remark}

\begin{corollary}\label{cor:NP-map}
By~(\ref{eq:dNn}), it holds $d\Nn(x_0)=\Id-P$.
If additionally $f(x_0)=0$ in Proposition~\ref{prop:wonder-prop-NEW},
then $\Nn(x_0)=x_0$ by uniqueness.
\end{corollary}

\begin{remark}\label{rem:NP-map-mu}
Given $\mu\ge 2$,
the restriction of the Newton-Picard map $\Nn:X\supset U_0\to X$
in~(\ref{eq:NP-map}) to the subset
\[
   U_0^\mu
   := U_0\cap \norm{df(\cdot)-D}^{-1}[0,\tfrac{1}{\mu c})
   ,\qquad U_0^2\stackrel{\text{(\ref{eq:U_0(delta)})}}{=} U_0 ,
\]
satisfies, just as above, an estimate of the form
\begin{equation*}
\begin{split}
   \norm{\bigl(\Id+Q \,df|_{x_1}-P\bigr)^{-1}-\Id}
%   &\stackrel{\text{(\ref{eq:Neumann})}}{=}
%   \Norm{\sum_{n=1}^\infty A^n}
   \le \sum_{n=1}^\infty \norm{A}^n
   \le \sum_{n=1}^\infty \tfrac{1}{\mu^n}
   =\frac{1}{\mu}\cdot \frac{1}{1-\frac{1}{\mu}}
   =\frac{1}{\mu-1}
\end{split}
\end{equation*}
for every $x_1\in U_0^\mu$ and where we used that
\[
   \norm{A}=\norm{Q(D-df|_{x_1})}\le\norm{Q}\cdot\norm{D-df|_{x_1}}
   \le\tfrac{1}{\mu} .
\]
\end{remark}

\boldmath
%%%%%%%%%%%%%%%%%%%%%%%%%%%%%%%%%%%
%%%%%%% Subsection:  %%%%%%%%%%%%%%%%%%%
%%%%%%%%%%%%%%%%%%%%%%%%%%%%%%%%%%%
\subsection{Tangent map}
\label{sec:NP-tangent}
\unboldmath

\begin{hypothesis}\label{rem:NP-tangent}
Consider the situation of Proposition~\ref{prop:wonder-prop-NEW}.
In this section we assume, in addition, that the map
$f\colon X\supset U\to Y$ is two times continuously differentiable.
Recall that $x_0\in U$ is a suitable initial point and $\delta$ and
$c$ are positive constants, the three of them related by
assumption~(\ref{eq:A.3.4-NEW}). Choose $\delta>0$ smaller,
if necessary, such that
\begin{equation}\label{eq:A.3.4-NEW-modified}
   \norm{x-x_0}<\delta\qquad\Rightarrow\qquad
   \norm{df(x)-D}\le\frac{1}{4c} .
\end{equation}

\smallskip
Suppose that there is a constant $c_2>0$ such that
$\norm{d^2f(x)}\le c_2$ for all 
$x\in B_\delta(x_0)$.
   %$x\in U_0(\delta)\subset U$; see~(\ref{eq:U_0(delta)}). 
Define
\[
   \hat\delta:=\min\bigl\{\delta,\tfrac{1}{4c c_2}\bigr\}
   \subset (0,1) .
\]
\end{hypothesis}

\begin{theorem}\label{thm:tangent-map}
Under Hypothesis~\ref{rem:NP-tangent}
suppose that $x_1\in U_0(\delta)$; see~(\ref{eq:U_0(delta)}).
Abbreviate $\Nn:=\Nn^f_{x_0,Q}$.
Then for each $\xi_1\in X$ which is small in the sense that
\begin{equation}\label{eq:xi_1-small}  % NEW
   \norm{\xi_1}<\frac{\delta}{8(1{\color{red}\,+\norm{d\Nn|_{x_0}}})}
   ,\qquad
   \norm{df|_{x_1}\xi_1}_Y
   <\frac{\delta}{4c} ,
\end{equation}
there is the estimate
\begin{equation}\label{eq:Delta}
   \norm{d\Nn|_{x_1}\xi_1-\xi_1}_{X}
   \le
   2c
   \max\left\{{\color{cyan}\tfrac{\delta}{\hat\delta}}\norm{f(x_1)}_Y,
   \norm{df|_{x_1}\xi_1}_Y\right\}
   .
%   <\frac{\delta}{2} .
\end{equation}
\end{theorem}

\begin{corollary}\label{cor:tangent-map}
For $(x_1,\xi_1)\in T U_0(\delta)$, see~(\ref{eq:U_0(delta)}), 
there is the estimate
\begin{equation*}%\label{eq:Delta-cor}
\begin{split}
   &\norm{d\Nn|_{x_1}\xi_1-\xi_1}_{X}
   \\  
   &\le
   2c
   \max\left\{
%1
   {\color{cyan}\tfrac{9(1{\color{red}\,+\norm{d\Nn|_{x_0}}})\norm{\xi_1}}{\hat\delta}}\norm{f(x_1)}_Y,
%2
{\color{cyan}\tfrac{5c\, \norm{df|_{x_1}\xi_1}}{\hat\delta}}\norm{f(x_1)}_Y,
   \norm{df|_{x_1}\xi_1}_Y\right\}
   .
\end{split}
\end{equation*}
\end{corollary}

\begin{proof}[Proof of Corollary~\ref{cor:tangent-map}]
Let $x_1\in U_0(\delta)$; see~(\ref{eq:U_0(delta)}).
For $\xi_1\in X$ we define
\[
   \lambda=\lambda(x_1,\xi_1)
   :=\min\left\{
   \frac{\delta}{9(1{\color{red}\,+\norm{d\Nn|_{x_0}}})\norm{\xi_1}}
   ,
   \frac{\delta}{5c\, \norm{df|_{x_1}\xi_1}}
   \right\} .
\]
Observe that
\begin{equation*}%\label{eq:Delta-cor}
\begin{split}
   \frac{1}{\lambda}
   =\max\left\{
   \frac{9(1{\color{red}\,+\norm{d\Nn|_{x_0}}})\norm{\xi_1}}{\delta}
   ,
   \frac{5c\, \norm{df|_{x_1}\xi_1}}{\delta}
   \right\} .
\end{split}
\end{equation*}
Then $\lambda\xi_1$ meets condition~(\ref{eq:xi_1-small}), so
there is the estimate
\begin{equation*}
\begin{split}
   &\norm{d\Nn|_{x_1}\xi_1-\xi_1}_{X}
   \\
   &=\frac{1}{\lambda}\norm{(d\Nn|_{x_1}-\Id)\lambda\xi_1}_{X}
   \\
   &\le
   \frac{2c}{\lambda}
   \max\left\{{\color{cyan}\tfrac{\delta}{\hat\delta}}\norm{f(x_1)}_Y,
   \norm{df|_{x_1}\lambda\xi_1}_Y\right\}
   \\
   &\le
   2c
   \max\left\{{\color{cyan}\tfrac{\delta}{\hat\delta}}\tfrac{1}{\lambda}\norm{f(x_1)}_Y,
   \norm{df|_{x_1}\xi_1}_Y\right\}
   \\
   &=
   2c
   \max\left\{
%1
   {\color{cyan}\tfrac{9(1{\color{red}\,+\norm{d\Nn|_{x_0}}})\norm{\xi_1}}{\hat\delta}}\norm{f(x_1)}_Y,
%2
{\color{cyan}\tfrac{5c\, \norm{df|_{x_1}\xi_1}}{\hat\delta}}\norm{f(x_1)}_Y,
   \norm{df|_{x_1}\xi_1}_Y\right\}
   .
\end{split}
\end{equation*}
\end{proof}

We can summarize the result of this section
more compactly in tangent map notation by the following theorem.

\begin{theorem}\label{thm:tangent-map-knackig}
Under Hypothesis~\ref{rem:NP-tangent},
given $W_1=(x_1,\xi_1)\in T U_0(\delta)$, see~(\ref{eq:U_0(delta)}), 
there is a constant $C$ depending on $\norm{\xi_1}_X$ and
$\norm{d\Nn|_{x_0}}$ such that
\begin{equation}\label{eq:wonder-prop-diff-T}
\begin{split}
   \norm{T\Nn(W_1) -W_1}_{TX}
   \le
   C \norm{Tf(W_1)}_{TY}\Bigl(1+\norm{Tf(W_1)}_{TY}\Bigr)
   .
\end{split}
\end{equation}
\end{theorem}

\begin{proof}[Proof of Theorem~\ref{thm:tangent-map-knackig}]
Write $W_1=(x_1,\xi_1)$ and observe that
\[
   T\Nn(W_1) -W_1
   =\left(\Nn(x_1)-x_1, d\Nn|_{x_1}\xi_1-\xi_1\right) .
\]
To the first component apply estimate~(\ref{eq:wonder-prop-main})
in the McDuff-Salamon Proposition~\ref{prop:wonder-prop-NEW}
and to the second component apply Corollary~\ref{cor:tangent-map}.
\end{proof}

\begin{proof}[Proof of Theorem~\ref{thm:tangent-map}]
We first consider the Newton-Picard map $\Nn^{Tf}$ for
the tangent map $Tf$ and then compare it with the tangent map $T\Nn^f$
of $\Nn^f$.

\smallskip
On $\hat X:=TX=X\oplus X\ni (x,\xi)$ and
$\hat Y:=TY=Y\oplus Y\ni (y,\eta)$
We define norms for
$(x,\xi)\in TX=X\oplus X$, respectively $(y,\eta)\in TY=Y\oplus Y$ by
\[
   \norm{(x,\xi)}:=\max\{ %{\color{cyan}\tfrac{\delta}{\hat\delta}}
\norm{x},
 {\color{red}\tfrac{\hat\delta}{\delta}}\norm{\xi}\}
   ,\qquad
   \norm{(y,\eta)}:=\max\{ %{\color{cyan}\tfrac{\delta}{\hat\delta}}
\norm{y},
 {\color{red}\tfrac{\hat\delta}{\delta}}\norm{\eta}\} .
\]
In the following we study the tangent map $\hat f:=Tf$ which is
defined by $Tf(x,\xi)=\left(f(x),df|_x\xi\right)$.
The task at hand is to choose the corresponding quantities
$\hat {x_0}$, $\hat D$, $\hat Q$, $\hat c$, $\hat \delta$, and $\hat
{U_0}$ in order to apply Proposition~\ref{prop:wonder-prop-NEW} to
$\hat f$.
\\
As initial point for $\hat f$ on $\hat U:=TU=U\times X$ we pick
$\hat{x_0}:=(x_0,0)$.
The operator $\hat D:=d\hat f|_{(x_0,0)}=D\oplus D$ is
onto. A right inverse is given by the sum $\hat Q:=Q\oplus Q$
with bound $\hat c=c$.\footnote{
Use the bound $\norm{Q}\le c$ to obtain the inequality in what follows
\begin{equation*}
\begin{split}
   \norm{Q\oplus Q}_{\Ll(TY,TX)}
   :&=\sup_{\norm{(\hat y,\hat \eta)}_{TY}=1}
   \norm{(Q\oplus Q)(\hat y,\hat \eta)}_{TX}
\\
   &=\sup_{\norm{(\hat y,\hat \eta)}_{TY}=1}
   \max\{\norm{Q\hat y}, {\color{red}\tfrac{\hat\delta}{\delta}} \norm{Q\hat\eta}\}
\\
   &\le c \sup_{\norm{(\hat y,\hat \eta)}_{TY}=1}
   \underbrace{\max\{\norm{\hat y},
     {\color{red}\tfrac{\hat\delta}{\delta}}
     \norm{\hat\eta}\}}_{\norm{(\hat y,\hat\eta)}_{TY}} =c .
%\\
%   &= c \sup_{\norm{\hat y,\hat \eta)}_{TY}=1}
%   \norm{(\hat y,\hat\eta)}_{TY}
%\\
%   &= c .
\end{split}
\end{equation*}
}
Observe that $B_{\delta}(\hat{x_0};TX)\subset U\times X=\hat U$.
Suppose that $x,\xi\in X$ satisfy the estimate
\[
   \norm{(x,\xi)-(x_0,0)}
   =\max\{ %{\color{cyan}\tfrac{\delta}{\hat\delta}}
\norm{x-x_0},
   {\color{red}\tfrac{\hat\delta}{\delta}}\norm{\xi}\}
   <\delta .
\]
In particular, we have $x\in B_\delta(x_0)$,
hence $\norm{d^2f(x)}\le c_2$ and $\norm{df|_x-D}\le \frac{1}{4c}$,
by~(\ref{eq:A.3.4-NEW-modified}).
For elements $\tilde x$ and $\tilde\xi$ of $X$ consider the operator
difference
\begin{equation*}
\begin{split}
   &\Norm{\left(d\hat f|_{(x,\xi)}-\hat D\right)
{\color{gray} 
   (\tilde x,\tilde\xi)
}}_{TY}
\\
   &=\Norm{
   \begin{bmatrix}
      df|_x-D&0\\d^2f|_x(\xi,\cdot)&df|_x-D
   \end{bmatrix}
{\color{gray}
   \begin{bmatrix}
      \tilde x\\\tilde\xi
   \end{bmatrix}
}}_{TY}
\\
   &=\max\left\{ %{\color{cyan}\tfrac{\delta}{\hat\delta}}
   \norm{(df|_x-D)\tilde x}, {\color{red}\tfrac{\hat\delta}{\delta}}
   \norm{(df|_x-D)\tilde\xi+d^2f|_x(\underline{\xi},\tilde x)}
   \right\}
\\
   &\stackrel{\text{(\ref{eq:A.3.4-NEW-modified})}}{\le}
   \max\left\{
   \tfrac{1}{4c} %{\color{cyan}\tfrac{\delta}{\hat\delta}}
\norm{\tilde x},
   {\color{red}\tfrac{\hat\delta}{\delta}}\tfrac{1}{4c}\norm{\tilde\xi}
   \right\}
   + {\color{red}\tfrac{\hat\delta}{\delta}}
   \underline{\delta}\norm{d^2f(x)}\cdot\norm{\tilde x}
\\
   &=
   \tfrac{1}{4c}\norm{(\tilde x,\tilde\xi)}
   +\hat\delta\norm{d^2f(x)}\cdot\norm{\tilde x} .
%   \quad{\color{gray}\text{\small ,
%   $\max\{\alpha,\beta+\gamma\}\stackrel{\gamma\ge 0}{\le}
%   \max\{\alpha,\beta\}+\gamma$}} .
\end{split}
\end{equation*}
Take
   \todo[color=yellow!40]{\tiny $\max\{\alpha,\beta+\gamma\}
   \stackrel{\gamma\ge 0}{\le}
   \max\{\alpha,\beta\}+\gamma$
   }
the supremum over all $\norm{(\tilde x,\tilde\xi)}=1$ to get the
operator norm estimate
\[
   \norm{d\hat f|_{(x,\xi)}-\hat D}
   \le\tfrac{1}{4c}+\hat\delta\norm{d^2f(x)}
   \le\tfrac{1}{4c}+\tfrac{1}{4cc_2} c_2
   =\tfrac{1}{2c} .
\]
Thus we have verified condition~(\ref{eq:A.3.4-NEW})
in Proposition~\ref{prop:wonder-prop-NEW}
with $\hat f$ and $\hat\delta$ in place of $f$ and $\delta$.
We next check condition~(\ref{eq:A.3.5-NEW}) for $\hat f$ and
$\hat\delta$ and any
\[
   \hat{x_1}:=(x_1,\xi_1)\in X\oplus X ,\quad
   \text{where $x_1\in U_0(\delta)$ and $\xi_1$ satisfies~(\ref{eq:xi_1-small})} .
\]
By $x_1\in U_0(\delta)$ it holds
$\norm{f(x_1)}<\tfrac{\delta}{4c}$
and $\norm{x_1-x_0}<\tfrac{\delta}{8}$.
By~(\ref{eq:xi_1-small}) we get~(\ref{eq:A.3.5-NEW})
\begin{equation}\label{eq:jhghjvhg}
   \norm{\hat f(\hat{x_1})}
   =
\max\{
   \norm{f(x_1)},
   {\color{red}\tfrac{\hat\delta}{\delta}} \norm{df|_{x_1}\xi_1}
\}
   <
\max\{
   \tfrac{\delta}{4c}, \tfrac{\hat\delta}{4c}
\}
   =\frac{\delta}{4c}
\end{equation}
and
\[
   \norm{\hat{x_1}-\hat{x_0}}
   =
\max\{
   %{\color{cyan}\tfrac{\delta}{\hat\delta}} 
   \norm{x_1-x_0},
   {\color{red}\tfrac{\hat\delta}{\delta}} \norm{\xi_1}
\}
   <\max\{\tfrac{\delta}{8}, \tfrac{\hat\delta}{8}
   \}
  =\frac{\delta}{8}
   .
\]
Then Proposition~\ref{prop:wonder-prop-NEW} for $\hat f$ and
$\delta$ yields a unique zero $\hat x=(x,\xi)$ of
$\hat f$ such that
\begin{equation}\label{eq:claim-unique}
   f(x)=0,\; df|_{x}\xi=0
   ,\quad
   x-x_1,\xi-\xi_1\in\im Q
   ,\quad
   \max\{\norm{x-x_0},
   {\color{red}\tfrac{\hat\delta}{\delta}} \norm{\xi}\}<\delta .
\end{equation}
In particular, since $\norm{x-x_0}<\delta$
the element $x$ is the same as
the one \underline{uniquely} determined by~(\ref{eq:wonder-prop-NEW})
and baptized $\Nn^f_{x_0,Q}(x_1)$ in~(\ref{eq:NP-map}).
Moreover, Proposition~\ref{prop:wonder-prop-NEW} for $\hat f$ and
$\delta$ yields that
\begin{equation}\label{eq:wonder-Tf}
   \max\{ %{\color{cyan}\tfrac{\delta}{\hat\delta}}
\norm{x-x_1}, {\color{red}\tfrac{\hat\delta}{\delta}}\norm{\xi-\xi_1}\}
   \le 2c\max\{ %{\color{cyan}\tfrac{\delta}{\hat\delta}}
\norm{f(x_1)}, {\color{red}\tfrac{\hat\delta}{\delta}}\norm{df|_{x_1}\xi_1}\}
   .
\end{equation}

To conclude the proof of Theorem~\ref{thm:tangent-map}
we need

\begin{proposition}\label{prop:TNf=NTf}
%\textbf{Claim.}
There is the identity
$T\Nn^f_{x_0,Q}=\Nn^{Tf}_{(x_0,0),Q\oplus Q}$, that is
\begin{equation}\label{eq:claim-TN}
   x=\Nn(x_1)
   ,\qquad
   \xi=d\Nn|_{x_1}\xi_1 .
\end{equation}
where we abbreviated $\Nn=\Nn^f_{x_0,Q}$.
\end{proposition}

\begin{proof}[Proof of Proposition~\ref{prop:TNf=NTf}]
%\medskip
%\noindent
%We prove the claim.
After~(\ref{eq:claim-unique}) we already proved $x=\Nn(x_1)$.
By uniqueness it suffices to verify the (three) properties in~(\ref{eq:claim-unique})
for $d\Nn|_{x_1}\xi_1$ in place of $\xi$.

\smallskip\noindent
{\sc Property} 1: $d\Nn|_{x_1}\xi_1\in\ker df|_x$. Since $x=\Nn(x_1)$ and since
$x_1$ lies in the open domain $U_0(\delta)$
of the Newton-Picard map $\Nn$
in~(\ref{eq:NP-map}), hence so does $x_1+\tau\xi_1$ for any
sufficiently small $\tau>0$, we obtain that
\begin{equation*}
\begin{split}
   df|_x\circ d\Nn|_{x_1}\xi_1
   &= \left.\tfrac{d}{d\tau}\right|_{0} 
   f\circ\underbrace{\Nn(x_1+\tau\xi_1)}_{\in f^{-1}(0)}
   =0 .
\end{split}
\end{equation*}
\\
{\sc Property} 2: $d\Nn|_{x_1}\xi_1-\xi_1\in\im Q$. Observe that
\begin{equation*}
\begin{split}
   -\left(d\Nn|_{x_1}\xi_1-\xi_1\right)
   &\stackrel{\text{(\ref{eq:dNn})}}{=}
   \Bigl(\Id+Q \,df|_{x_1}-QD\Bigr)^{-1} Q\circ df|_{x_1}\xi_1
   =Q\eta .
\end{split}
\end{equation*}
That the last equality indeed holds for some $\eta\in Y$ is equivalent to
\[
   Q\circ df|_{x_1}\xi_1
   =\Bigl(\Id+Q \,df|_{x_1}-QD\Bigr) Q\eta
{\color{gray}\;
   =Q \,df|_{x_1}Q\eta
}
\]
for some $\eta\in Y$. Since $Q$ is injective it remains to find
an $\eta\in Y$ such that
\[
   df|_{x_1}\xi_1 =df|_{x_1}Q\eta .
\]
But the operator $df|_{x_1}Q$ is invertible since it is of the form
$\Id-B$ where $B:=(D-df|_{x_1})Q$ has norm $\norm{B}\le\frac{1}{4}$
due to $\norm{Q}\le c$ and by~(\ref{eq:A.3.4-NEW-modified});
cf. Remark~\ref{rem:A}.

\smallskip\noindent
{\sc Property} 3: $\norm{d\Nn|_{x_1}\xi_1}<\delta$.
As $\Id-QD=d\Nn|_{x_0}$, see Corollary~\ref{cor:NP-map}, we get
\begin{equation*}
\begin{split}
   \norm{d\Nn|_{x_1}\xi_1}
   &\stackrel{\text{(\ref{eq:dNn})}}{=}
   \norm{\left(\Id+Q \,df|_{x_1}-QD\right)^{-1} (\Id-QD)\xi_1}
   \\
   &\stackrel{{\color{white}(B.62)}}{\le}
   \norm{\left(\Id+Q \,df|_{x_1}-QD\right)^{-1}}\cdot
   {\color{red}\,\norm{d\Nn|_{x_0}}}\cdot \norm{\xi_1}
   \\
   & \stackrel{\text{(\ref{eq:inverse-2})}\atop\text{(\ref{eq:xi_1-small})}}{<}
   2 \tfrac{\delta}{8}.
\end{split}
\end{equation*}
This proves the identities~(\ref{eq:claim-TN})
and Proposition~\ref{prop:TNf=NTf}
\end{proof}

\medskip
We continue and  conclude the proof of Theorem~\ref{thm:tangent-map}.
Since $\xi=d\Nn|_{x_1}\xi_1$, by~(\ref{eq:claim-TN}),
the estimate~(\ref{eq:wonder-Tf}) multiplied by
${\color{cyan}\tfrac{\delta}{\hat\delta}}$ leads to
\begin{equation*}
\begin{split}
   \norm{(d\Nn|_{x_1}-\Id)\xi_1}_{X}
   &\le
   2c
   \max\left\{{\color{cyan}\tfrac{\delta}{\hat\delta}}\norm{f(x_1)}_Y,
   \norm{df|_{x_1}\xi_1}_Y\right\} .
\end{split}
\end{equation*}
This concludes the proof of Theorem~\ref{thm:tangent-map}.
\end{proof}

\boldmath
%%%%%%%%%%%%%%%%%%%%%%%%%%%%%%%%%%%
%%%%%%%%%%%%%%%%%%%%%%%%%%%%%%%%%%%
%%%%%%% Section:  %%%%%%%%%%%%%%%%%%%%%
%%%%%%%%%%%%%%%%%%%%%%%%%%%%%%%%%%%
%%%%%%%%%%%%%%%%%%%%%%%%%%%%%%%%%%%
\section{Exponential decay}
\label{sec:exponential}
\unboldmath

\begin{theorem}[Linear uniform exponential decay]
\label{thm:lin-unif-exp-decay}
Pick $\eps\in(0,\sigma)$ where $\sigma=\sigma(A)$ is the spectral
gap~(\ref{eq:spec-gap}).
Let $m\in\N_0$. Suppose $W$ is a map of Sobolev class $W^{1,2}$
such that
\begin{equation}\label{eq:TW^s-app}
   W\colon[0,\infty)\to T^m\R^n,
   \qquad
   \p_sW+T^m\Nabla{}f(W)=0 .
\end{equation}
Then there is a positive constant $c(W)$,
depending continuously on $W$, such that
\begin{equation}\label{eq:exp-decay-W-app}
   \abs{W(s)}+\abs{\p_s W(s)}\le c(W)\cdot e^{-\eps s}
\end{equation}
for every $s\ge0$, 
\end{theorem}

\boldmath
%%%%%%%%%%%%%%%%%%%%%%%%%%%%%%%%%%%
%%%%%%% Subsection:  %%%%%%%%%%%%%%%%%%%
%%%%%%%%%%%%%%%%%%%%%%%%%%%%%%%%%%%
\subsection*{Preparation of proof}
%\label{sec:???}
\unboldmath

By $\Pp^*(\N)$ we denote the collection of all finite non-empty subsets of $\N$.
The evaluation map is defined by
\[
   e\colon\Pp^*(\N)\to\N,\quad
   D\mapsto\sum_{j\in D} 2^{j-1}
\]
and its inverse is the digit map
\[
   \Dd:=e^{-1}\colon\N\to\Pp^*(\N) .
\]
It can be described as follows.
Write $k\in\N$ in binary representation and map it
to the subset of $\N$ consisting of all positions of the binary
representation of $k$ at which you can find a $1$,
for example $9=1001\mapsto \{1,4\}$.

Given a finite non-empty subset $D\subset \N$, in symbols
$D\in\Pp^*(\N)$, we consider all partitions of $D$ into $\ell\in\N$ non-empty
subsets, namely
\[
   \mathrm{Part}_\ell(D):=\{\{A_1,\dots,A_\ell\}\subset\Pp(D)\mid
   \cup_{i=1}^\ell A_i=D, 
   A_i\cap A_j\stackrel{i\not=j}{=}\emptyset,
   \forall i\colon A_i\not=\emptyset\} .
\]
Given $\ell\in\N_0$, the ODE~(\ref{eq:TW^s-app}) for the map
$W\colon[0,\infty)\to T^\ell\R^n$ is equivalent to a system of $2^\ell$
ODEs for $2^\ell$ maps $W_0,W_1,\dots, W_{2^\ell-1}\colon[0,\infty)\to\R^n$,
namely
\begin{equation}\label{eq:super-eq-0}
   \p_sW_0+\Nabla{}f(W_0)=0
\end{equation}
and the $2^\ell-1$ equations
\begin{equation}\label{eq:super-eq}
\begin{split}
   0
   &=\p_s W_k
   +\sum_{\ell\in\N}
   \biggl(
      \sum_{\{A_1,\dots,A_\ell\}\in \mathrm{Part}_\ell \Dd(k)}
   D^\ell\Nabla{}f|_{W_0} [W_{e(A_1)}, \dots ,W_{e(A_\ell)}]
   \biggr)
\end{split}
\end{equation}
where $k=1,\dots,2^\ell-1$.

\begin{remark}[Reformulation of~(\ref{eq:super-eq})]
Given $k\in\N$, let $\mathfrak{S}(k)$ be the digit sum
of the binary representation of $k$, also
referred to as the Hamming weight.
Observe that $\mathfrak{S}(k)$ is the cardinality of $\Dd(k)$.
Note that for $\ell>\mathfrak{S}(k)$ the partition set
$\mathrm{Part}_\ell \Dd(k)=\emptyset$ is empty.
Note also that $\mathrm{Part}_1 \Dd(k)=\{\{\Dd(k)\}\}$.
Therefore we can write~(\ref{eq:super-eq})
equivalently as the finite sum
\begin{equation}\label{eq:super-eq-mod}
\begin{split}
   0
   &=\p_s W_k
   +D\Nabla{}f|_{W_0} [W_k]
\\
   &\quad
   +
   \underbrace{
   \sum_{\ell=2}^{\mathfrak{S}(k)}
   \biggl(
      \sum_{\{A_1,\dots,A_\ell\}\in \mathrm{Part}_\ell \Dd(k)}
   D^\ell\Nabla{}f|_{W_0} [W_{e(A_1)}, \dots ,W_{e(A_\ell)}]
   \biggr)
   }_{=:\eta}
   .
\end{split}
\end{equation}
In the special case where $k=2^m$ we have
$\mathfrak{S}(k)=1$, hence~(\ref{eq:super-eq-mod}) simplifies to
\begin{equation}\label{eq:super-eq-k=2^m}
   0=\p_s W_{2^m}+D\Nabla{}f|_{W_0} [W_{2^m}].
\end{equation}
\end{remark}
The following table illustrates~(\ref{eq:super-eq-mod})
for $k=0,\dots,7$. It is written in binary notation, so the structure of the
system becomes visible
\begin{equation*}
\begin{split}
   \text{0)\:\: }0&=\p_s W_0+\Nabla{}f(W_0)
\\
   \text{1) \:\: }0&=\p_s W_1+D\Nabla{}f|_{W_0} W_1
\\
   \text{10) \:\: }0&=\p_s W_{10}+D\Nabla{}f|_{W_0} W_{10}
\\
   \text{11) \:\: }0&=\p_s W_{11}+D^2\Nabla{}f|_{W_0} [W_1,W_{10}]
   +D\Nabla{}f|_{W_0} W_{11} 
\\
   \text{100) \:\: }0&=\p_s W_{100}+D\Nabla{}f|_{W_0} W_{100}
\\
   \text{101) \:\: }0&=\p_s W_{101}+D^2\Nabla{}f|_{W_0} [W_1,W_{100}]
   +D\Nabla{}f|_{W_0} W_{101} 
\\
   \text{110) \:\: }0&=\p_s W_{110}+D^2\Nabla{}f|_{W_0} [W_{10},W_{100}]
   +D\Nabla{}f|_{W_0} W_{110} 
\\
   \text{111) \:\: }0&=\p_s W_{111}
   +D^3\Nabla{}f|_{W_0} [W_1,W_{10},W_{100}]
   +D^2\Nabla{}f|_{W_0} [W_{10},W_{101}]
\\
   &\quad
   +D^2\Nabla{}f|_{W_0} [W_1,W_{110}]
   +D^2\Nabla{}f|_{W_0} [W_{11},W_{100}]
   +D\Nabla{}f|_{W_0} W_{111} .
\end{split}
\end{equation*}

\begin{lemma}\label{le:lin-m}
Given $m\in\N_0$, consider maps
$W_0,W_1,\dots, W_{2^m-1}\in W^{1,2}([0,\infty),\R^n)$ that
satisfy the ODE system~(\ref{eq:super-eq-0}) and~(\ref{eq:super-eq})
for every $k=1,\dots,2^m-1$.
Then the tuple $W:=(W_0,\dots,W_{2^m -1}) \in W^{1,2}([0,\infty),\R^{n\cdot 2^m})$ lies
in the $m$-fold tangent space $T^m \Ww^{\rm s}$ which means that
\[
   \p_s W+T^m\Nabla{}f(W)=0 .
\]
\end{lemma}

\begin{proof}
The proof is by induction on $m\in\N$.
\smallskip
\newline
\textbf{Case \boldmath$m=0$.}
True by assumption.
\smallskip
\newline
\textbf{Induction step \boldmath$m\Rightarrow m+1$.}
There are three cases I-III.
I. For $k\in\{1,\dots,2^m-1\}$ equation~(\ref{eq:super-eq})
holds directly by induction hypothesis.
II. For $k\in\{2^m+1,\dots,2^{m+1}-1\}$ we linearize~(\ref{eq:super-eq})
with respect to $W_{k-2^m}$. This yields
\begin{equation}\label{eq:super-eq-ind}
\begin{split}
   0
   &=\p_s W_k
   +\sum_{\ell\in\N}
   \sum_{\{A_1,\dots,A_\ell\}\atop\in\mathrm{Part}_\ell \Dd(k-2^m)}
\\
   &\biggl(
   \sum_{j=1}^\ell
   D^\ell\Nabla{}f|_{W_0} [W_{e(A_1)}, \dots ,W_{e(A_{j-1})},
   W_{e(A_j)+2^m}, W_{e(A_{j+1})}, \dots ,W_{e(A_{\ell})}]
\\
   &+D^{\ell+1}\Nabla{}f|_{W_0}
   [W_{2^m},W_{e(A_1)}, \dots ,W_{e(A_{\ell})}]
   \biggr)
\\
   &=\p_s W_k
   +\sum_{\ell\in\N}
   \biggl(
      \sum_{\{A_1,\dots,A_\ell\}\in \mathrm{Part}_\ell \Dd(k)}
   D^\ell\Nabla{}f|_{W_0} [W_{e(A_1)}, \dots ,W_{e(A_\ell)}]
   \biggr)
\end{split}
\end{equation}
To see why the second equation in~(\ref{eq:super-eq-ind}) holds
note the identity of digit sets
\[
   \Dd(k)=\Dd(k-2^m)\cup\{m+1\} .
\]
Moreover, consider the injections defined for $j=1,\dots,\ell$ by
\begin{equation*}%\label{eq:XXX}
\begin{split}
   \iota_j\colon \mathrm{Part}_\ell (\Dd(k-2^m))
   &\INTO \mathrm{Part}_\ell (\Dd(k)) =
   \mathrm{Part}_\ell (\Dd(k-2^m)\cup\{m+1\})
\\
   \{A_1,\dots,A_\ell\}
   &\mapsto \{A_1,\dots,A_{j-1},A_j\cup\{m+1\},A_{j+1},\dots,A_\ell\}
\end{split}
\end{equation*}
and the injection defined by
\begin{equation*}%\label{eq:XXX}
\begin{split}
   I\colon \mathrm{Part}_{\ell-1} (\Dd(k-2^m))
   &\INTO \mathrm{Part}_\ell (\Dd(k))
\\
   \{A_1,\dots,A_{\ell-1}\}
   &\mapsto \{\{m+1\},A_1,\dots,A_{\ell-1}\}.
\end{split}
\end{equation*}
Using this notion we can write $\mathrm{Part}_\ell \Dd(k)$
as the union of pairwise disjoint subsets, namely
\begin{equation}\label{eq:union}
\begin{split}
   \mathrm{Part}_\ell \Dd(k) 
   &=\biggl(\bigcup_{j=1}^\ell\iota_j(\mathrm{Part}_\ell (\Dd(k-2^m)))\biggr)
   \cup I(\mathrm{Part}_{\ell-1} (\Dd(k-2^m))).
\end{split}
\end{equation}
Now the second equation in~(\ref{eq:super-eq-ind})
follows from~(\ref{eq:union}).

\smallskip
III. It remains to consider the case $k=2^m$.
Linearizing~(\ref{eq:super-eq-0}) with respect to $W_0$
in direction $W_{2^m}$ we obtain
\begin{equation*}%\label{eq:union}
\begin{split}
   0=\p_s W_{2^m}
   +D\Nabla{}f|_{W_0} W_{2^m} 
\end{split}
\end{equation*}
and this equation coincides with~(\ref{eq:super-eq-k=2^m}).
This proves Lemma~\ref{le:lin-m}.
\end{proof}

\boldmath
%%%%%%%%%%%%%%%%%%%%%%%%%%%%%%%%%%%
%%%%%%% Subsection:  %%%%%%%%%%%%%%%%%%%
%%%%%%%%%%%%%%%%%%%%%%%%%%%%%%%%%%%
\subsection*{Proof of exponential decay}
%\label{sec:???}
\unboldmath

\begin{proof}[Proof of Theorem~\ref{thm:lin-unif-exp-decay} -- Exponential decay]
The proof is by induction on $m$.
\smallskip
\newline
\textbf{Case \boldmath$m=0$.}
This follows for instance from the action-energy inequality;
see e.g.~\cite{Frauenfelder:2022g}.

\smallskip
\noindent
\textbf{Induction step \boldmath$m\Rightarrow m+1$.}
Suppose~(\ref{eq:exp-decay-W-app}) is true for $m$.
Then we want to show~(\ref{eq:exp-decay-W-app}) for $m+1$.
By induction hypothesis $W_k$ and its derivative $\p_s W_k$ decay
exponentially for $k=0,\dots2^m-1$.
It remains to show that as well $W_k$ and its derivative
$\p_s W_k$ decay exponentially for $k=2^m,\dots,2^{m+1}-1$.
This follows from Lemma~\ref{le:exp-dec} below
in view of~(\ref{eq:super-eq-mod}) combined with the induction hypothesis.
More precisely, we prove this by induction on $k$.
In the notation $A$, $\xi$, $\eta$
of Lemma~\ref{le:exp-dec} we have
$W_k=\xi$, $D\Nabla{}f|_{W_0}=A$ and $\eta$ is the sum
indicated in~(\ref{eq:super-eq-mod}).

Observe that if $\ell\ge 2$ and $\{A_1,\dots,A_\ell\}\in \mathrm{Part}_\ell(\Dd(k))$
then $e(A_j)<k$ for $j=1,\dots,\ell$.
Therefore by induction hypothesis $W_{e(A_j)}$ decays exponentially
so that $\eta$ decays exponentially.
Now the exponential decay of $W_k$ follows from
Lemma~\ref{le:exp-dec}.
\end{proof}

\begin{lemma}\label{le:exp-dec}
Consider a continuously differentiable family of quadratic matrizes
$\A\colon [0,\infty)\to\R^{n\times n}$ and an invertible symmetric
matrix $A\in \R^{n\times n}$ with
\[
   \lim_{s\to\infty} \norm{\A(s)-A}=0=\lim_{s\to\infty}
   \norm{\A^\prime(s)} 
   ,\qquad
   \A^\prime(s):=\tfrac{d}{ds}\A(s) .
\]
Let $\sigma=\sigma(A)>0$ be the spectral gap, see~(\ref{eq:spec-gap}).
Let $\xi,\eta\colon[0,\infty)\to\R^n$ be continuously differentiable
maps such that $\xi$ is of Sobolev class $W^{1,2}$ and
\begin{equation}\label{eq:12}
   \xi^\prime(s)+\A(s)\xi(s)=\eta(s)
\end{equation}
for every $s\ge 0$. Suppose that there are constants $C>0$
and $\eps\in(0,\sigma)$ such that
\begin{equation}\label{eq:13}
   \abs{\eta(s)}+\abs{\eta^\prime(s)}\le C e^{-\eps s}
\end{equation}
for every $s\ge 0$.
Then there is a positive constant $c$, depending continuously on the $W^{1,2}$
norm of $\xi$ and the constant $C$, such that
\[
   \abs{\xi(s)}
   \le c e^{-\eps s} 
\]
for every $s\ge 0$.
\end{lemma}

Observe that the exponential decay rate of $\eta$ is inherited by $\xi$,
as opposed to~\cite[Le.\,3.1]{Robbin:2001a}.

\begin{proof}
We follow the proof of~\cite[Le.\,3.1]{Robbin:2001a}.
We shall employ the following facts and assumptions.
The norms of a quadratic real matrix $B$ and its
transpose $B^t$ are equal.
   \todo[color=yellow!40]{\small [Conway:1985a] $\norm{A}=\norm{A^t}$}
By definition of the spectral gap $\sigma>0$ it holds that
\[
   \abs{Av}\ge \sigma\abs{v}
\]
for every $v\in\R^n$. Given $\delta>0$ and $\eps\in(0,\sigma)$,
by assumption there is a large
time $s_0=s_0(\delta;\sigma,\eps)>0$ such that
\begin{equation}\label{eq:ass-exp}
   \Bigl(\norm{\A^\prime(s)}
   +\bigl(
   \tfrac{13}{4}
   +16\tfrac{\sigma^2}{(\sigma^2-\eps^2)}
    \bigr)
   \norm{\A(s)-A}^2\Bigr)
   \le\frac{\sigma^2-\eps^2}{4}
\end{equation}
pointwise for $s\ge s_0$.
The function defined for $s\ge 0$ by
\[
   \alpha(s):=\tfrac12\abs{\xi(s)}^2
\]
has derivatives
\[
   \alpha^\prime
   =\INNER{\xi}{\xi^\prime}
   =\INNER{\xi}{\eta-\A\xi}
\]
and
\[
   \alpha^{\prime\prime}
   =
   \INNER{\xi^\prime}{\eta-(\A+\A^t)\xi}
   +\INNER{\xi}{\eta^\prime-\A^\prime\xi}.
\]
Substitute $\xi^\prime$ according
to~(\ref{eq:12}), then add $-A+A$ various times, to obtain
\begin{equation*}
\begin{split}
   \alpha^{\prime\prime}
%1
   &=\abs{\A\xi}^2 
   +\abs{\eta}^2
   -2\INNER{\A\xi}{\eta}
   -\INNER{\eta}{\A^t\xi}
   +\INNER{\xi}{\eta^\prime-\A^\prime\xi}
   +\INNER{\A\xi}{\A^t\xi}
\\
%2
   &=\abs{(\A-A+A)\xi}^2 
   +\abs{\eta}^2
   -2\INNER{(\A-A)\xi}{\eta}   -2\INNER{A\xi}{\eta}
   -\INNER{\eta}{(\A^t-A)\xi}
\\
   &\quad
   -\INNER{\eta}{A\xi}
   +\INNER{\xi}{\eta^\prime-\A^\prime\xi}
   +\INNER{(\A-A+A)\xi}{(\A^t-A+A)\xi}
\\
%3
   &=\abs{(\A-A)\xi}^2 +\abs{A\xi}^2
   +2\INNER{(\A-A)\xi}{A\xi}
   \\
   &\quad
   +\abs{\eta}^2
   -2\INNER{(\A-A)\xi}{\eta}   -3\INNER{A\xi}{\eta}
   -\INNER{\eta}{(\A-A)^t\xi}
   +\INNER{\xi}{\eta^\prime}
   -\INNER{\xi}{\A^\prime\xi}
\\
   &\quad
   +\INNER{(\A-A)\xi}{(\A-A)^t\xi}
   +\INNER{A\xi}{(\A-A)^t\xi}
   +\INNER{(\A-A)\xi}{A\xi}
   +\abs{A\xi}^2 .
\end{split}
\end{equation*}
Observe that $\abs{A\xi}^2$ appears twice and,
in the following, we write this coefficient in the form
$2=\frac{\sigma^2+\eps^2}{\sigma^2}+\frac{\sigma^2-\eps^2}{\sigma^2}$.
By Cauchy-Schwarz and Peter-Paul\footnote{
  $ab\le\frac{a^2+b^2}{2}$ whenever $a,b\ge0$
  }
we obtain
\begin{equation*}
\begin{split}
   \alpha^{\prime\prime}
%%% 1
   &\ge
   \tfrac{\sigma^2+\eps^2}{\sigma^2} \abs{A\xi}^2
   +\tfrac{\sigma^2-\eps^2}{\sigma^2} \abs{A\xi}^2 
   + \abs{\eta}^2
   {\color{brown} \;-\;3\norm{\A-A}\cdot\abs{\xi}\cdot\abs{\eta}\;}
   \underline{\;-\;3\abs{A\xi}\cdot\abs{\eta}}
\\
   &\quad
   -\abs{\xi}\cdot\abs{\eta^\prime}-\norm{\A^\prime}\cdot\abs{\xi}^2
   -\norm{\A-A}^2\cdot\abs{\xi}^2
   {\color{cyan} \;-\;4\abs{A\xi}\cdot \norm{\A-A}\cdot\abs{\xi}}
\\
%%% 2
   &\ge
   (\sigma^2+\eps^2)\abs{\xi}^2
   +\tfrac{\sigma^2-\eps^2}{{\color{red} 2}\sigma^2} \underbrace{\abs{A\xi}^2}_{\ge \sigma^2\abs{\xi}^2}
   +
\underbrace{
   \left(\tfrac{\sigma^2-\eps^2}{{\color{red} 2}\sigma^2}
   \underline{\;-\;\tfrac{\sigma^2-\eps^2}{4\sigma^2}}
   {\color{cyan} \;-\;\tfrac{\sigma^2-\eps^2}{4\sigma^2}}
   \right)
}_{=0}
   \abs{A\xi}^2 
   -\abs{\xi}\cdot\abs{\eta^\prime}
\\
   &\quad
   -
\underbrace{
   \Bigl(\norm{\A^\prime}
   +\bigl(
   1{\color{brown}+\tfrac{9}{4}} 
{\color{cyan} \;+\; 4^2 \tfrac{\sigma^2}{(\sigma^2-\eps^2)}
}
    \bigr)
   \norm{\A-A}^2\Bigr)
}_{\text{$\le\frac{\sigma^2-\eps^2}{4}$ by~(\ref{eq:ass-exp})}}
   \abs{\xi}^2
   +
   (1{\color{brown}-1}\underline{\;-\;\tfrac{3^2\sigma^2}{(\sigma^2-\eps^2)}}
) \abs{\eta}^2 
%\\
%   &\quad
%   -\abs{\xi}\cdot\abs{\eta^\prime}
\\
%%% 3
   &\ge
   (\sigma^2+\eps^2)\abs{\xi}^2
   +\tfrac{\sigma^2-\eps^2}{2} (1-\tfrac12-\tfrac12)\abs{\xi}^2
   -\tfrac{3^2\sigma^2}{\sigma^2-\eps^2}\abs{\eta}^2
   -\tfrac{1}{\sigma^2-\eps^2}\abs{\eta^\prime}^2
\\
%%% 4
   &\ge (2\delta)^2\alpha - c_0 e^{-2\eps s}
   ,\qquad
   2\delta^2:=\sigma^2+\eps^2
   ,\quad
   c_0:=\tfrac{9\sigma^2+1}{\sigma^2-\eps^2} C^2 ,
\end{split}
\end{equation*}
pointwise for $s\ge s_0$.
Inequality two and three is by $\abs{A\xi}\ge\sigma\abs{\xi}$,
the final inequality by the $\eta$, $\eta^\prime$ decay assumption~(\ref{eq:13}).
Observe the estimate
$$
   2\delta=2\sqrt{\tfrac{\sigma^2+\eps^2}{2}}
   =2\sqrt{\eps^2+\tfrac{\sigma^2-\eps^2}{2}}
   >2\eps .
$$
The function defined by
\[
   \beta(s)
   :=\alpha(s)+\frac{c_0 e^{-2\eps s}}{(2\eps)^2-(2\delta)^2}
\]
satisfies
\begin{equation*}
\begin{split}
   \beta^{\prime\prime}(s)
   &=\alpha^{\prime\prime}(s)
   +\tfrac{c_0(2\eps)^2 e^{-2\eps s}}{(2\eps)^2-(2\delta)^2} 
\\
   &\ge (2\delta)^2\alpha
   +\tfrac{c_0(2\eps)^2 e^{-2\eps s}}{(2\eps)^2-(2\delta)^2} 
   -c_0e^{-2\eps s}
{\color{gray}\,
   \tfrac{(2\eps)^2-(2\delta)^2}{(2\eps)^2-(2\delta)^2}
}
\\
   &=(2\delta)^2\beta(s)
\end{split}
\end{equation*}
for $s\ge s_0$. This implies, exactly as in the proof
of~\cite[Le.\,3.1]{Robbin:2001a}, the following.
Firstly $\frac{d}{ds} e^{2\delta s}\beta(s)\le 0$ for $s\ge s_0$,\footnote{
  Here boundedness of $\abs{\xi(s)}$ enters which is true by the assumption
  $\xi\in W^{1,2}$.
  }
so secondly
$
   e^{2\delta s_0}\beta(s_0)\ge
   e^{2\delta s} \beta(s)
$,
and therefore thirdly
$\beta(s)\le e^{-2\delta(s-s_0)}\beta(s_0)$ decays even faster than
$e^{-2\eps s}$.
Thus
\[
   \alpha(s)
   =\beta(s)-\tfrac{c_0 e^{-2\eps s}}{(2\eps)^2-(2\delta)^2}
   <\left(e^{-(2\delta-2\eps)s} e^{2\delta s_0}\beta(s_0)
   +\tfrac{c_0}{(2\delta)^2-(2\eps)^2}\right) e^{-2\eps s}
\]
and therefore
\[
   \abs{\xi(s)}
   =\sqrt{2\alpha(s)}
   <
   \sqrt{2e^{-(2\delta-2\eps)s} e^{2\delta s_0}\beta(s_0)+\tfrac{2c_0}{(2\delta)^2-(2\eps)^2}}
   \; e^{-\eps s}
\]
for $s\ge s_0$.
\end{proof}

%\newpage
%%%%%%%%%%%%%%%%%%%%%%%%%
%%%%%%%%% REFERENCES %%%%%%
%%%%%%%%%%%%%%%%%%%%%%%%
%\renewcommand{\bibname}{References}
%\bibliographystyle{plain}
         %   erzeugt:     [1] Joa Weber
%\bibliographystyle{abbrv}
         %  erzeugt:      [1] J. Weber and 
\bibliographystyle{alpha}
         %  article:    [Web05]  J. Weber
         %  book:      [Web05]  Joa Weber
         % more authors: [HZ87]
%%%%%%%%%%%%%%%%%%%%%%%%%
%% include Bibliography in TOC %%
% en.wikibooks.org/wiki/LaTeX/Bibliography_Management#Using_tocbibind
%%%%%%%%%%%%%%%%%%%%%%%%%
% Using hyperref, one should say:
%\cleardoublepage
%\phantomsection
\addcontentsline{toc}{section}{References}
\bibliography{$HOME/Dropbox/0-Libraries+app-data/Bibdesk-BibFiles/library_math,$HOME/Dropbox/0-Libraries+app-data/Bibdesk-BibFiles/library_math_2020,$HOME/Dropbox/0-Libraries+app-data/Bibdesk-BibFiles/library_physics}{}
%$
%%%%%%%%%%%%%%%%%%%%%%%%%
%%%%%%%%% standard %%%%%%%%%
%%%%%%%%%%%%%%%%%%%%%%%%%
%\begin{thebibliography}{00000}
%\small
%\end{thebibliography}

%%%%%%%%%%%%%%%%%%%%%%%%%%%%%%%%%%%%
%%%%%%%%%%%%% GLOSSARY %%%%%%%%%%%%%%%
%%%%%%%%%%%%%%%%%%%%%%%%%%%%%%%%%%%%
% Using hyperref, one should say:
%\cleardoublepage
%\phantomsection
%\printnomenclature
%
%This is $F$\label{nomen:F} 
%\nomenclature[EF]{$F$}{Objective function}{}{\pageref{nomen:F}}
%\clearpage

%%%%%%%%%%%%%%%%%%%%%%%%%
%%%%%%%%% INDEX %%%%%%%%%%
%%%%%%%%%%%%%%%%%%%%%%%%%
% Using hyperref, one should say:
%\cleardoublepage
%\phantomsection
%\addcontentsline{toc}{chapter}{Index}
%\printindex

\end{document}